\documentclass[11pt,letterpaper]{amsart}
\usepackage{mathrsfs}
\usepackage{amssymb}
\usepackage{graphicx}
\usepackage[T1]{fontenc}
\usepackage[all]{xy}
\usepackage{arydshln}
\usepackage{amscd}
\usepackage{amsmath, amssymb}
\usepackage{amsthm}
\usepackage{amsfonts}
\usepackage[colorlinks,linkcolor=blue,citecolor=blue, pdfstartview=FitH]
{hyperref}
\usepackage{backref}
\usepackage{amscd}
\usepackage{easybmat}
\usepackage{mathrsfs}
\usepackage{amsfonts}
\usepackage{color}
\usepackage{pifont}
\usepackage{upgreek}
\usepackage{bm}
\usepackage{hyperref}
\usepackage{shorttoc}
\usepackage{amsmath,amstext,amsthm,a4,amssymb,amscd}
\usepackage[mathscr]{eucal}
\usepackage{mathrsfs}
\usepackage{epsf}
\usepackage{tikz}
\usepackage{CJK}
\usetikzlibrary{arrows}
\usepackage{float}
\usepackage{caption}
\numberwithin{equation}{section}

\def\mb{\mathbb}
\def\mc{\mathcal}
\def\mr{\mathrm}

\def\wt{\widetilde}

\def\ra{\rightarrow}

\def\br{\mathbb R}

\def\bs{\boldsymbol}
\def\rro{\bs{\rho}}

\theoremstyle{plain}
\newtheorem{thm}{Theorem}[section]
\newtheorem{lemma}[thm]{Lemma}
\newtheorem{prop}[thm]{Proposition}
\newtheorem{cor}[thm]{Corollary}
\theoremstyle{definition}
\newtheorem{rem}[thm]{Remark}

\theoremstyle{definition}
\newtheorem{defn}[thm]{Definition}

\newcommand{\comment}[1]{}

\usepackage{amscd}
\usepackage{fancyhdr}
\makeatletter
\@namedef{subjclassname@2020}{2020 Mathematics Subject Classification}

\begin{document}
\begin{CJK}{UTF8}{gbsn}

\title{Topological components of surface group representations and  signature}
\author{InKang Kim}

\author{Xueyuan Wan}

\address{Inkang Kim: School of Mathematics, KIAS, Heogiro 85, Dongdaemun-gu Seoul, 02455, Republic of Korea}
\email{inkang@kias.re.kr}

\address{Xueyuan Wan: Mathematical Science Research Center, Chongqing University of Technology, Chongqing 400054, China}
\email{xwan@cqut.edu.cn}
\begin{abstract}
We study the topological components of the surface group representations into \(\mathrm{SL}(2,\mathbb{R}) \) and \(\mathrm{PSL}(2,\mathbb{R}) \). Utilizing the signature formula established in \cite{KPW}, we determine the number of connected components of the representation spaces with boundary elliptic, hyperbolic, and parabolic holonomies.
\end{abstract}

 \subjclass[2020]{57M50, 57JM07, 22E40}  
 \keywords{Signature,  surface group representations, connected components}
 
\maketitle
\tableofcontents

\section*{Introduction}

In the theory of surface group representations, understanding the topology of representation spaces is a fundamental topic which influenced many branches of mathematics including non-abelian Hodge theory, Higgs bundle theory and many others \cite{Hitchin1,Simpson1992,Corlette}. In particular, determining the number of connected components of these spaces is a basic yet crucial question \cite{Gold1,Hitchin1992}. Topological invariants serve as powerful tools in addressing such problems. Among these invariants, the Toledo invariant plays a central role in the representation theory of surface groups \cite{Toledo,BIW,BGG2003}.
 A foundational result in this direction was established by Domic and Toledo~\cite{DT}, who proved the celebrated Milnor--Wood inequality \cite{Milnor, Wood}. This landmark result was further generalized by Burger, Iozzi, and Wienhard~\cite{BIW}, who extended the definition of the Toledo invariant to surfaces with boundary for any Hermitian Lie group, utilizing the framework of bounded cohomology. They also showed that the Milnor--Wood inequality continues to hold in this broader setting.

In his seminal work~\cite{Gold1}, Goldman introduced the notion of the \emph{relative Euler class} for boundary non-elliptic representations, showing it takes integer values. In our previous work~\cite{KPW1}, we proved that the Toledo invariant coincides precisely with the relative Euler class (see Proposition~\ref{pre-prop1}). By exploiting both the continuity of the relative Euler class and its integrality in the case of closed surfaces, Goldman \cite[Theorem A]{Gold1} proved that the representation space $\mathrm{Hom}(\pi_1(\Sigma_{g,0}), \mathrm{SL}(2,\mathbb{R}))$ has exactly $2^{2g + 1} + 2g - 3$ connected components. Goldman further extended these ideas to surfaces with boundary, showing that the space of representations in $\mathrm{PSL}(2,\mathbb{R})$ with hyperbolic boundary holonomy decomposes into connected components indexed by the relative Euler class~\cite[Theorem D]{Gold1}. Recently, Ryu and Yang~\cite{RY} investigated the topology of the connected components of boundary parabolic representations into \(\mathrm{PSL}(2,\mathbb{R})\) and determined their number, see also~\cite{DT0, Kas, MPY, Yang} for related discussions. In contrast, much less is known about the topology of representation spaces with elliptic boundary holonomies. In particular, determining the number of connected components in this setting remains open.

In this paper, we address this problem by employing the signature invariant to systematically study the connected components of surface group representations into $\mathrm{SL}(2,\mathbb{R})$ and $\mathrm{PSL}(2,\mathbb{R})$ with fixed boundary holonomies of elliptic, hyperbolic, or parabolic type, and we aim to count the number of connected components of these representation spaces. A key difficulty arises from the fact that, for surfaces with boundary, the Toledo invariant is generally no longer an integer~\cite[Theorem 1(2)]{BIW}, making it insufficient to distinguish between connected components. Instead, we turn to the \emph{signature} invariant as a more refined topological invariant that remains integer-valued and is well-suited for our purpose.

In \cite{Hitchin1}, Hitchin studied representations of closed surface groups into \( G = \mathrm{SL}(2,\mathbb{R}) \), and established a real-analytic correspondence between irreducible representations \( \rho: \pi_1(\Sigma) \to \mathrm{SL}(2,\mathbb{C}) \) and stable Higgs bundles. Since then, the Higgs bundle approach has become a powerful and widely adopted method for studying moduli spaces of surface group representations (see, for example, \cite{Corlette, Donaldson, Simpson, Simpson1}). In particular, for surfaces with punctures, the topology of the moduli space of Higgs bundles has been explored in works such as \cite{Boden-Yokogawa, Mondello, Nasatyr-Steer}. However, these studies primarily focus on the case of nonzero relative Euler classes and do not compute the precise number of connected components (see e.g. \cite{Mondello}).

Building on the signature formula established in \cite{KPW}, we systematically and comprehensively investigate the topological structure of the connected components of boundary elliptic, hyperbolic, and parabolic representations. In particular, we provide a detailed counting of the number of these connected components.
To the best of our knowledge, the present work is the first systematic attempt to count the connected components of representation spaces with boundary elliptic holonomy using the signature invariant.
The signature invariant is another fundamental object in the study of surface group representations. Lusztig~\cite{Lus} and Meyer~\cite{Meyer, Meyer1} studied the signature of flat vector bundles over closed, even-dimensional manifolds, providing explicit integral formulas. For surfaces with boundary, Atiyah~\cite{Atiyah} provided a geometric interpretation of the signature as the sum of the integral of the first Chern class over the surface and the $\eta$-invariant contributions arising from the boundary. Unlike the Toledo invariant, the signature is always integer-valued, which makes it a more robust invariant for distinguishing connected components in moduli spaces of representations.

In~\cite{KPW}, we conducted a detailed study of the relationship between the Toledo invariant and the signature for representations $\phi \colon \pi_1(\Sigma) \to G$, where $G = \mr{U}(p,q)$ or $\mr{Sp}(2n,\mathbb{R})$. In particular, for $G = \mathrm{SL}(2,\mathbb{R})$, we established the following identity:
\begin{equation}\label{Sig-Tol}
\mathrm{sign}(\phi) = 2\, \mathrm{T}(\phi) + \rro(\phi),
\end{equation}
where $\rro(\phi)$ denotes the rho invariant (see Section~\ref{defn of rho}), and $\mathrm{T}(\phi)$ is the Toledo invariant~\cite[Theorem 4]{KPW}. 
If $\Sigma$ is the connected sum of two (connected) surfaces $\Sigma_i$ along a separating loop, then the signature satisfies
\[
\mathrm{sign}(\phi)=\mathrm{sign}(\phi|_{\pi_1(\Sigma_1)})+\mathrm{sign}(\phi|_{\pi_1(\Sigma_2)}).
\]
Moreover, reversing the orientation of the surface reverses the sign of the signature. These properties also hold for the Toledo invariant (cf. \cite{BIW}).

Let $\Sigma = \Sigma_{g,n}$ be a compact oriented surface of genus $g$ with $n$ boundary components $c_1, \dots, c_n$. Its fundamental group has the standard presentation:
\[
\pi_1(\Sigma) = \left\langle a_1, b_1, \dots, a_g, b_g, c_1, \dots, c_n \,\middle\vert\, \Pi_{i=1}^g [a_i, b_i] \Pi_{j=1}^n c_j = e \right\rangle.
\]
The conjugacy classes of $\mathrm{SL}(2,\mathbb{R}) \setminus \{\pm I\}$ naturally decompose into three types: elliptic, parabolic, and hyperbolic.

 Given a representation $\phi \in \mathrm{Hom}(\pi_1(\Sigma), \mathrm{SL}(2,\mathbb{R}))$, we say that it is a \emph{boundary elliptic} (respectively, \emph{parabolic} or \emph{hyperbolic}) representation if each $\phi(c_j)\in \mr{SL}(2,\mb{R})$ is elliptic (respectively, parabolic or hyperbolic).

\subsection{Boundary elliptic representations}  

An element \( A \in \mathrm{SL}(2,\mathbb{R}) \) is called \emph{elliptic} if \( |\mathrm{tr}(A)| < 2 \). Such an element is conjugate to a rotation matrix of the form
\[
R(\theta) = 
\begin{pmatrix}
  \cos\theta & -\sin\theta \\
  \sin\theta & \cos\theta
\end{pmatrix}
\in \mathrm{SO}(2) \setminus \{\pm I\}, \quad \theta \in (0,\pi) \cup (\pi,2\pi).
\]
We say an element \( A \) is \emph{elliptic-unipotent} if it is conjugate to \( R(\theta) \) for some \( \theta \in (0,2\pi) \). 

Let \( \mathrm{Ell}(\Sigma_{g,n}) \) (resp. \( \mathrm{E}(\Sigma_{g,n}) \)) denote the space of all boundary elliptic (resp. elliptic-unipotent) representations in \( \mathrm{Hom}(\pi_1(\Sigma_{g,n}), \mathrm{SL}(2,\mathbb{R})) \). 
Given \( \phi \in \mathrm{Ell}(\Sigma_{g,n}) \) or \( \mathrm{E}(\Sigma_{g,n}) \), we may assume that \( \phi(c_i) \sim R(\theta_i) \). We define the sigma map:
\[
\sigma\colon \mathrm{Ell}(\Sigma_{g,n}) \to \{-1,1\}^n, \quad 
\sigma(\phi) := \left( \mathrm{sgn}(\theta_1 - \pi), \dots, \mathrm{sgn}(\theta_n - \pi) \right).
\]

From the identity \eqref{Sig-Tol} and the values of the rho invariant in Table \ref{tab:rho-invariant}, we obtain:
\[
\mathrm{sign}(\phi) = \mathrm{T}(\phi) + 2\sum_{i=1}^n (1 - \tfrac{\theta_i}{\pi}).
\]
This expression shows that the Toledo invariant is generally not an integer, due to the second term on the right-hand side. However, for representations in \( \mathrm{Ell}(\Sigma_{g,n}) \) or \( \mathrm{E}(\Sigma_{g,n}) \), the signature takes values in the set
\[
\mathrm{sign}(\mathrm{Ell}(\Sigma)) = \mathrm{sign}(\mathrm{E}(\Sigma)) 
= \left[2\chi(\Sigma), 2|\chi(\Sigma)|\right] \cap \left\{2|\chi(\Sigma)| - 4\mathbb{Z} \right\}.
\]

For any path \( \phi(t) \in \mathrm{Ell}(\Sigma_{g,n}) \) or \( \mathrm{E}(\Sigma_{g,n}) \), the signature remains constant, since \( \theta_i(t) \notin \{0, 2\pi\} \) along the path. Thus, the signature is locally constant on these spaces and can be used to count connected components. We denote by \( \#\{\mathrm{Ell}(\Sigma_{g,n})\} \) (resp. \( \#\{\mathrm{E}(\Sigma_{g,n})\} \)) the number of connected components.

We now state our first main result:

\begin{thm}\label{main thm1}
Let \( \Sigma_{g,n} \) be a  compact oriented surface of genus \( g \) with \( n \geq 1 \) boundary components, and assume \( \chi(\Sigma_{g,n})=2-2g-n \leq 0 \). Then:
\begin{itemize}
  \item The number of connected components of boundary elliptic-unipotent representations is
  \[
  \#\{\mathrm{E}(\Sigma_{g,n})\} = 2^{2g+1} + 2g + n - 3.
  \]
  \item The number of connected components of boundary elliptic representations is
  \[
  \#\{\mathrm{Ell}(\Sigma_{g,n})\} =
  \begin{cases}
    2^{n-1}(n - 1) & \text{if } g = 0, \\
    2^{2g + n} + 2^{n-1}(4g + n - 5) & \text{otherwise}.
  \end{cases}
  \]
\end{itemize}
\end{thm}

To prove this result, we first show that any \( \phi \in \mathrm{Ell}(\Sigma_{g,n}) \) can be deformed to a boundary \( \mathrm{SO}(2) \setminus \{\pm I\} \)-valued representation via a path within \( \mathrm{Ell}(\Sigma_{g,n}) \); see Lemma~\ref{lemma6}. We then analyze the contributions from the signature and the sigma map to distinguish different connected components.

We also consider the representation space \( \mathrm{Hom}(\pi_1(\Sigma), \mathrm{PSL}(2,\mathbb{R})) \). Let
\[
\pi\colon \mathrm{SL}(2,\mathbb{R}) \to \mathrm{PSL}(2,\mathbb{R})
\]
be the natural projection. As shown in \cite[Lemma 2.2]{Gold1}, this induces a covering map onto its inverse image of each connected component:
\[
\pi_*\colon \mathrm{Hom}(\pi_1(\Sigma_{g,n}), \mathrm{SL}(2,\mathbb{R})) 
\to \mathrm{Hom}(\pi_1(\Sigma_{g,n}), \mathrm{PSL}(2,\mathbb{R})), 
\, \phi \mapsto [\phi] := \pi \circ \phi.
\]

Let \( \mathrm{Ell}_P(\Sigma_{g,n}) \) denote the space of boundary elliptic representations into \( \mathrm{PSL}(2,\mathbb{R}) \), and \( \#\{\mathrm{Ell}_P(\Sigma_{g,n})\} \) the number of connected components.
For any \( \phi \in \mathrm{Ell}(\Sigma_{g,n}) \), we define its signature in \( \mathrm{PSL}(2,\mathbb{R}) \) by
\[
\mathrm{sign}([\phi]) = 2\,\mathrm{T}(\phi) + 2(n - \sum_{i=1}^n \tfrac{[\theta_i]}{\pi} ),
\]
where \( [\theta_i] \in (0,\pi) \) is taken to be the unique representative of \( \theta_i \) mod \( \pi \) in that interval. This signature is well-defined and continuous on \( \mathrm{Ell}_P(\Sigma_{g,n}) \). In Lemma~\ref{lemma10}, we show that for any \( k \in \mathbb{Z} \), the set
\(
\mathrm{sign}^{-1}(k) \cap \mathrm{Ell}_P(\Sigma_{g,n})
\)
is either empty or a single connected component.
We obtain the following theorem:

\begin{thm}\label{main thm2}
The set of possible values of the signature for boundary elliptic representations into \( \mathrm{PSL}(2,\mathbb{R}) \) is given by
\[
\mathrm{sign}(\mathrm{Ell}_P(\Sigma_{g,n})) = 
\begin{cases}
  [2, 2(n - 1)] \cap 2\mathbb{Z} & \text{if } g = 0, \\
  [-4g + 4, 4g - 4 + 2n] \cap 2\mathbb{Z} & \text{otherwise}.
\end{cases}
\]
In particular, the number of connected components is
\[
\#\{\mathrm{Ell}_P(\Sigma_{0,n}) \}= n - 1, \,
\#\{\mathrm{Ell}_P(\Sigma_{g,n})\} = 4g + n - 3 \quad \text{for } g \geq 1.
\]
\end{thm}

\subsection{Boundary hyperbolic representations}  

An element \( A \in \mathrm{SL}(2,\mathbb{R}) \) is said to be \emph{hyperbolic} if \( |\mathrm{tr}(A)| > 2 \). The set of all hyperbolic elements is denoted by \( \mathrm{Hyp} \), and it decomposes into two connected components:
\[
\mathrm{Hyp}_0 := \left\{ A \in \mathrm{SL}(2,\mathbb{R}) : \mathrm{tr}(A) > 2 \right\}, \,
\mathrm{Hyp}_1 := -I \cdot \mathrm{Hyp}_0.
\]
Each component corresponds to a family of one-parameter subgroups in \( \mathrm{SL}(2,\mathbb{R}) \). We define the signature function \( \sigma \colon \mathrm{Hyp} \to \{0,1\} \) by
\[
\sigma(A) := 
\begin{cases}
1 & \text{if } \mathrm{tr}(A) > 2, \\
0 & \text{if } \mathrm{tr}(A) < -2.
\end{cases}
\]

Let \( \mathrm{Hyp}(\Sigma_{g,n}) \) (resp. \( \mathrm{Hyp}_P(\Sigma_{g,n}) \)) denote the space of all boundary hyperbolic representations in 
\(
\mathrm{Hom}(\pi_1(\Sigma_{g,n}), \mathrm{SL}(2,\mathbb{R})) 
\) (resp. $\mathrm{Hom}(\pi_1(\Sigma_{g,n}), \mathrm{PSL}(2,\mathbb{R}))$).
In his seminal work \cite{Gold1}, Goldman studied the topology of \( \mathrm{Hyp}_P(\Sigma_{g,n}) \) and proved that the relative Euler class map classifies its connected components:
\[
e\colon \mathrm{Hyp}_P(\Sigma_{g,n}) \to \mr{H}^2(\Sigma, \partial \Sigma; \mathbb{Z}) \cong \mathbb{Z},
\]
see \cite[Theorem D]{Gold1}. However, the number of connected components of the corresponding space in \( \mathrm{SL}(2,\mathbb{R}) \), i.e., \( \mathrm{Hyp}(\Sigma_{g,n}) \), was not computed in that work. For any representation \( \phi \in \mathrm{Hyp}(\Sigma_{g,n}) \), its rho invariant vanishes. Consequently, the signature of \( \phi \) satisfies
\[
\mathrm{sign}(\phi) = 2\, \mathrm{T}(\phi),
\]
 In particular, the values of the signature is the set
\[
\mathrm{sign}(\mathrm{Hyp}(\Sigma)) =[-2|\chi(\Sigma)|,\, 2|\chi(\Sigma)|] \cap 2\mathbb{Z}.
\]

For completeness—and also to prepare for the analysis of boundary parabolic representations—we determine the number of connected components of \( \mathrm{Hyp}(\Sigma_{g,n}) \) as follows:

\begin{thm}\label{main thm3}
Let \( \Sigma_{g,n} \) be a compact oriented surface of genus \( g \) with \( n \geq 1 \) boundary components, and $\chi(\Sigma_{g,n})\leq 0$. Then the number of connected components of the space of boundary hyperbolic representations is
\[
\#\left\{ \mathrm{Hyp}(\Sigma_{g,n}) \right\} 
= 2^{2g + n} + 2^{n-1}(4g + 2n - 5).
\]
\end{thm}

\subsection{Boundary parabolic representations}

An element \( A \in \mathrm{SL}(2,\mathbb{R}) \setminus \{\pm I\} \) is called \emph{parabolic} if \( |\mathrm{tr}(A)| = 2 \). The set of parabolic elements is denoted by \( \mathrm{Par} \). Every parabolic element \( C \in \mathrm{SL}(2,\mathbb{R}) \) is conjugate to a matrix of the form
\[
\alpha^{\pm}(s) = \pm
\begin{pmatrix}
1 & s \\
0 & 1
\end{pmatrix},
\]
where \( s = \pm 1 \). If \( C \sim \alpha^\pm(s) \), we define the associated \( \sigma \)-invariant by
\begin{equation}
\sigma(C) :=
\begin{cases}
- \mathrm{sgn}(s) & \text{if } \mathrm{tr}(C) = 2, \\
\mathrm{sgn}(s) \sqrt{-1} & \text{if } \mathrm{tr}(C) = -2.
\end{cases}
\end{equation}
In particular, the rho invariant satisfies \( \rro(C) = \mathrm{Re}(\sigma(C)) \), and we have the identity \( \sigma(C^{-1}) = -\sigma(C) \).

Let \( \mathrm{Par}(\Sigma) \) denote the space of boundary parabolic representations of \( \pi_1(\Sigma) \) into \( \mathrm{SL}(2,\mathbb{R}) \), where \( \Sigma = \Sigma_{g,n} \). Define the sigma map
\[
\sigma \colon \mathrm{Par}(\Sigma) \to \{ \pm 1, \pm \sqrt{-1} \}^n, \quad
\phi \mapsto \left( \sigma(\phi(c_1)), \dots, \sigma(\phi(c_n)) \right).
\]

Suppose two representations \( \phi, \phi' \in \mathrm{Par}(\Sigma) \) are connected by a path \( \phi_t \) within \( \mathrm{Par}(\Sigma) \). Since the Toledo invariant is continuous and the rho invariant is constant on parabolic elements, it follows from the identity
\[
\mathrm{sign}(\phi_t) = 2\, \mathrm{T}(\phi_t) + \rro(\phi_t)
\]
that both \( \mathrm{sign}(\phi_t) \) and \( \rro(\phi_t) \) remain constant along the path. Furthermore, representations with distinct sigma values cannot be connected within \( \mathrm{Par}(\Sigma) \), as the rho invariant would otherwise jump. The possible values of the signature over \( \mathrm{Par}(\Sigma) \) lie in
\[
\mathrm{sign}(\mathrm{Par}(\Sigma)) = [-2|\chi(\Sigma)|,\, 2|\chi(\Sigma)|] \cap \mathbb{Z}.
\]

Now consider any \( \phi \in \mathrm{Par}(\Sigma_{g,n}) \) with \( n \) even. Via deformation, we can construct a boundary parabolic representation \( \phi' \in \mathrm{Hyp}(\Sigma_{n/2}) \). In particular, we obtain a natural map $\Phi$
from the set of connected components of $\mr{Par}(\Sigma_{g,n})$ to the set of connected components of  $\mr{Hyp}(\Sigma_{n/2})$, which is covering map with degree $8^{n/2}$ (see Corollary \ref{par-cor2}). A similar construction holds when \( n \) is odd. We thus obtain the following result:

\begin{thm}\label{main thm4}
Let \( \Sigma_{g,n} \) be a compact oriented surface of genus \( g \) with \( n \geq 1 \) boundary components, and $\chi(\Sigma_{g,n})\leq 0$. 
The number of connected components of the space of boundary parabolic representations into \( \mathrm{SL}(2,\mathbb{R}) \) is given by
\[
\#\{ \mr{Par}(\Sigma_{g,n}) \} =
\begin{cases}
16 & \text{for } g=n=1, \\
2^{2g+2n}+2^{2n-1} (4g+n-5) & \text{for } g\geq 2 \text{ is even}\\
2^{2g+2n-1}+2^{2n-1} (4g+n-4) & \text{for } g\geq 1 \text{ is odd}, g+n> 2\\
 2^{2n-1} \big(n - 3 \big) + 2^n (n + 1) &\text{for } g=0.
\end{cases}
\]
\end{thm}

\smallskip

We also consider the space of boundary parabolic representations into \( \mathrm{PSL}(2,\mathbb{R}) \), denoted by \( \mathrm{Par}_P(\Sigma) \).

\begin{thm}\label{main thm5}
The number of connected components of the space of boundary parabolic representations into \( \mathrm{PSL}(2,\mathbb{R}) \) is given by
\begin{equation*}
\#\{ \mr{Par}_P(\Sigma_{g,n}) \} =
\begin{cases}
2^{n} (4g+n-3) &  \text{for } g\geq 1 \\
2^{n} \big(n - 3 \big) + 2 (n + 1) &\text{for } g=0.
\end{cases}
\end{equation*}
\end{thm}
\begin{rem}
In \cite[Corollary 1.5]{RY}, Ryu and Yang proved Theorem~\ref{main thm5} by a completely different method.  
Their approach was to compute the number of connected components corresponding to each relative Euler class, and then sum these numbers to obtain the total.  

Our method instead proceeds by deforming to interior hyperbolic representations, thereby establishing a correspondence between boundary parabolic and boundary hyperbolic representations.  
Using the counting formula for boundary hyperbolic representations, we then deduce the number of connected components.  
\end{rem}

Here we briefly discuss some potential future research topics related to the study of topological components of surface group representations.

Let $\Sigma = \Sigma_{g,n}$ be a general surface with boundary. One may consider a surface group representation $\phi : \pi_1(\Sigma) \to G$, where $G$ is a $k$-fold covering $\pi:G\ra \mathrm{PSL}(2,\mathbb{R})$. A natural problem is to study the topological components of the spaces $\mathrm{Hom}(\pi_1(\Sigma), G)$ consisting of boundary elliptic, hyperbolic, and parabolic representations, and to compute the number of connected components. When $\Sigma$ is closed, i.e., $n = 0$, Goldman~\cite[Theorem A]{Gold1} has obtained the count of connected components in this case. For surfaces with boundary, this paper mainly considers the cases $k=1,2$. Clearly, when $g=0$, one can easily obtain a counting formula since $\pi_* : \mathrm{Hom}(\pi_1(\Sigma_{0,n}), G) \to \mathrm{Hom}(\pi_1(\Sigma_{0,n}), \mathrm{PSL}(2,\mathbb{R}))$ is a covering map of degree $k^{n-1}$. However, for general $g \geq 1$, the situation is more complicated compared with these two extreme cases, arising from the action $\mathrm{Hom}(\pi_1(\Sigma_{g,n}), \mathrm{Ker}(\pi))$.

In~\cite{KPW}, together with Pansu, we established a signature formula for surface group representations into $\mathrm{U}(p,q)$, and obtained a Milnor--Wood type inequality for the signature. Moreover, we introduced the definitions of boundary elliptic-unipotent, hyperbolic-unipotent, and unipotent (or parabolic) representations. Thus, a natural problem is to study the topological properties and to count the connected components of these subspaces. Note that $\mathrm{SL}(2,\mathbb{R})$ is isomorphic to $\mathrm{SU}(1,1)$, so from the results of this paper, the case $p=q=1$ is readily understood. However, for general $(p,q)$, this remains an open and interesting question. In addition, when $q=0$, i.e., in the $\mathrm{U}(p)$ case, the Toledo invariant is identically zero. In~\cite{KPW1}, we completely determined the possible values of signature, which makes it natural to investigate the topological components of boundary elliptic representations into $\mathrm{U}(p)$. Note that when $p=1$, $\mathrm{U}(1)$ is isomorphic to $\mathrm{SO}(2)$, which is studied systematically in Section~\ref{sec:SO(2)}, and Goldman~\cite[Theorem~A~(ii)]{Gold1} proved that for closed surfaces, the space of surface group representations into $\mathrm{SU}(2)$ is connected. 

\

This article is organized as follows. In Section~\ref{sec-Pre}, we recall the definition of the signature for surface group representations into \( \mathrm{SL}(2,\mathbb{R}) \), along with the Toledo invariant, relative Euler class, and the rho invariant. In Section~\ref{sec-elliptic}, we investigate the connected components of boundary elliptic representations and prove Theorems~\ref{main thm1} and~\ref{main thm2}. Section~\ref{sec-hyperbolic} is devoted to the study of boundary hyperbolic representations, where we establish Theorem~\ref{main thm3}. Finally, in Section~\ref{sec-parabolic}, we analyze boundary parabolic representations and prove Theorems~\ref{main thm4} and~\ref{main thm5}.

\

\text{{\bfseries{Acknowledgments.}}} 
Research by Inkang Kim is partially supported by Grant NRF-2019R1A2C1083865 and KIAS Individual Grant (MG031408). The first author is grateful for the support of IHES and MPI during his visits. Research by Xueyuan Wan is supported by Science and Technology Innovation Key R\&D Program of Chongqing and the National Key R\&D Program of China (Grant No. 2024YFA1013200) and the Natural Science Foundation of Chongqing, China (Grant No. CSTB2024NSCQ-LZX0040, CSTB2023NSCQ-LZX0042). The second author would like to thank Qiongling Li for bringing \cite{RY} to his attention.

\section{Preliminaries}\label{sec-Pre}

In this section, we recall the definition of the signature of surface group
representations into $\mathrm{SL}(2,\mb{R})$ and discuss its relation to the Toledo invariant.
For more comprehensive details, we refer the reader to \cite{Atiyah, BIW, KPW}.

\subsection{Definition of signature}

Let \( \Sigma \) be a connected, oriented surface with smooth boundary \( \partial\Sigma \), where each component of \( \partial\Sigma \) is homeomorphic to the circle \( S^1 \).

Consider the real vector space \( E = \mathbb{R}^2 \) endowed with the standard symplectic form \( \Omega \) of signature \( (1,1) \), defined by
\[
\Omega = dx \otimes dy - dy \otimes dx,
\]
whose corresponding matrix representation is
\[
J = \begin{pmatrix}
0 & 1 \\
-1 & 0 
\end{pmatrix}.
\]
By definition, the group \( \mathrm{SL}(2,\mathbb{R}) \) coincides with the symplectic group:
\[
\mathrm{SL}(2,\mathbb{R}) = \mr{Sp}(2,\mathbb{R}) := \left\{ M \in \mathrm{GL}(2,\mathbb{R}) \mid M^\top J M = J \right\},
\]
which consists of all linear transformations preserving the symplectic form \( \Omega \).

Let \( \phi: \pi_1(\Sigma) \to \mathrm{SL}(2,\mathbb{R}) \) be a representation of the fundamental group of \( \Sigma \) into \( \mathrm{SL}(2,\mathbb{R}) \), preserving the symplectic structure \( \Omega \). This induces a flat vector bundle \( \mathcal{E} = \widetilde{\Sigma} \times_\phi E \) over \( \Sigma \), where \( \widetilde{\Sigma} \) denotes the universal cover of \( \Sigma \).
Let \( \mathrm{H}^*(\Sigma, \mathcal{E}) \) and \( \mathrm{H}^*(\Sigma, \partial\Sigma, \mathcal{E}) \) denote the (relative) twisted singular cohomology groups with coefficients in \( \mathcal{E} \). Consider the image
\[
\widehat{\mathrm{H}}^1(\Sigma, \mathcal{E}) := \mathrm{Im}\left( \mathrm{H}^1(\Sigma, \partial\Sigma, \mathcal{E}) \to \mathrm{H}^1(\Sigma, \mathcal{E}) \right).
\]
On \( \widehat{\mathrm{H}}^1(\Sigma, \mathcal{E}) \), there exists a natural symmetric bilinear form
\[
Q_{\mathbb{R}}: \widehat{\mathrm{H}}^1(\Sigma, \mathcal{E}) \times \widehat{\mathrm{H}}^1(\Sigma, \mathcal{E}) \to \mathbb{R}
\]
defined by
\[
Q_{\mathbb{R}}([a],[b]) = \int_\Sigma \Omega([a] \cup [b]).
\]
By Poincar\'e duality (cf. \cite[Page 65]{APS}), the form \( Q_{\mathbb{R}} \) is non-degenerate. Moreover, since \( \Omega \) is skew-symmetric and the cup product is graded commutative, \( Q_{\mathbb{R}} \) is symmetric.

If we can decompose \( \widehat{\mathrm{H}}^1(\Sigma, \mathcal{E}) = \mathscr{H}^+ \oplus \mathscr{H}^- \) such that \( Q_{\mathbb{R}} \) is positive definite on \( \mathscr{H}^+ \) and negative definite on \( \mathscr{H}^- \), we define the {\it signature} of the symplectic bundle \( (\mathcal{E}, \Omega) \) as
\[
\mathrm{sign}(\mathcal{E}, \Omega) := \dim \mathscr{H}^+ - \dim \mathscr{H}^-.
\]

The signature of a representation is then defined as follows:

\begin{defn}
Let \( \phi: \pi_1(\Sigma) \to \mathrm{SL}(2,\mathbb{R}) \) be a representation. We define the \emph{signature} of \( \phi \) to be
\[
\mathrm{sign}(\phi) := \mathrm{sign}(\mathcal{E}, \Omega),
\]
where \( \mathcal{E} = \widetilde{\Sigma} \times_\phi \mathbb{R}^2 \) is the associated flat vector bundle equipped with the standard symplectic structure.
\end{defn}

\subsection{Toledo invariant and relative Euler class}

\subsubsection{Toledo invariant}
We now recall the definition of the Toledo invariant for surfaces with boundary, as introduced by Burger, Iozzi, and Wienhard in \cite[Section 1.1]{BIW}. The Toledo invariant can be expressed as the integral of the first Chern form with compact support over the surface (cf. \cite{KPW}).

A Lie group $G$ is said to be of {\it Hermitian type} if it is connected, semisimple with finite center, and has no compact factors, and if its associated symmetric space is Hermitian. Specifically, let $G$ be a group of Hermitian type so that the associated symmetric space $\mathscr{X}$ is Hermitian of noncompact type. Then, $\mathscr{X}$ carries a unique Hermitian (normalized) metric with minimal holomorphic sectional curvature equal to \(-1\). The corresponding K\"ahler form, denoted by \(\omega_{\mathscr{X}}\), belongs to \(\Omega^2(\mathscr{X})^G\), the space of \(G\)-invariant 2-forms on \(\mathscr{X}\). 

By the van Est isomorphism, we have
\[
\Omega^2(\mathscr{X})^G \cong \mathrm{H}^2_c(G, \mathbb{R}),
\]
where \(\mathrm{H}^{\bullet}_c(G, \mathbb{R})\) denotes the continuous cohomology of \(G\) with trivial \(\mathbb{R}\)-coefficients. There exists a unique class \(\kappa_G \in \mathrm{H}^2_c(G, \mathbb{R})\) corresponding to the K\"ahler form \(\omega_{\mathscr{X}}\), which induces a bounded K\"ahler class \(\kappa^b_G \in \mathrm{H}^2_{c,b}(G, \mathbb{R})\) via the isomorphism
\[
\mathrm{H}^2_c(G, \mathbb{R}) \cong \mathrm{H}^2_{c,b}(G, \mathbb{R}),
\]
where \(\mathrm{H}^{\bullet}_{c,b}(G, \mathbb{R})\) denotes the bounded continuous cohomology. In fact, \(\kappa^b_G \in \mathrm{H}^2_{c,b}(G, \mathbb{R})\) is defined by the following bounded cocycle:
\[
c(g_0, g_1, g_2) = \frac{1}{2\pi} \int_{\triangle(g_0x, g_1x, g_2x)} \omega_{\mathscr{X}},
\]
where \(\triangle(g_0x, g_1x, g_2x)\) is the geodesic triangle with ordered vertices \(g_0x, g_1x, g_2x\) for some base point \(x \in \mathscr{X}\).

By the Gromov isomorphism, we have 
\[
\phi_b^*(\kappa^b_G) \in \mathrm{H}^2_b(\pi_1(\Sigma), \mathbb{R}) \cong \mathrm{H}^2_b(\Sigma, \mathbb{R}),
\]
where \(\phi_b^*\) denotes the pullback of the bounded K\"ahler class. The canonical map 
\[
j_{\partial\Sigma} : \mathrm{H}^2_b(\Sigma, \partial\Sigma, \mathbb{R}) \to \mathrm{H}^2_b(\Sigma, \mathbb{R})
\]
from singular bounded cohomology relative to \(\partial\Sigma\) to singular bounded cohomology is an isomorphism. The Toledo invariant is then defined as
\[
\mr{T}(\phi)=\mathrm{T}(\Sigma, \phi) := \langle j^{-1}_{\partial\Sigma} \phi_b^*(\kappa^b_G), [\Sigma, \partial\Sigma] \rangle,
\]
where \(j^{-1}_{\partial\Sigma} \phi_b^*(\kappa^b_G)\) is viewed as an ordinary relative cohomology class, and \([\Sigma, \partial\Sigma] \in \mathrm{H}_2(\Sigma, \partial\Sigma, \mathbb{Z}) \cong \mathbb{Z}\) is the relative fundamental class.

According to \cite[Theorem 1]{BIW}, for a group \( G \) of Hermitian type and a representation \( \phi: \pi_1(\Sigma) \to G \), the Milnor-Wood inequality holds:
\begin{equation*}
|\mathrm{T}(\phi)| \leq \mathrm{rank}(\mathscr{X}) |\chi(\Sigma)|,
\end{equation*}
where \( \mathscr{X} \) is the symmetric space associated with \( G \), and \( \chi(\Sigma) \) denotes the Euler characteristic of the surface \( \Sigma \).
In the case where \( G = \mathrm{SL}(2,\mb{R}) \), the Milnor-Wood inequality is 
\begin{equation}\label{eqn25}
  |\mr{T}(\phi)|\leq |\chi(\Sigma)|.
\end{equation}

Since \( \mathrm{SL}(2,\mathbb{R}) \) and \( \mathrm{PSL}(2,\mathbb{R}) \) act identically on the Poincar\'e disk, for any representation \( \phi: \pi_1(\Sigma) \to \mathrm{SL}(2,\mathbb{R}) \), we have
\[
\mathrm{T}(\phi) = \mathrm{T}(\pi \circ \phi),
\]
where \( \pi: \mathrm{SL}(2,\mathbb{R}) \to \mathrm{PSL}(2,\mathbb{R}) \) denotes the double covering map.

\subsubsection{Relative Euler Class}

Let \( \widetilde{G} \) denote the universal covering group of \( \mathrm{SL}(2,\mathbb{R}) \), and let \( Z \) be its center, which is isomorphic to \( \mathbb{Z} \). Let \( \partial \) be the circle at infinity of the hyperbolic plane \( \mathbb{H}^2 \), and fix once and for all an orientation on \( \partial \). The group \( \widetilde{G} \) acts naturally on the universal cover \( \widetilde{\partial} \). Let \( z \in Z \) be a generator that acts by positively oriented translation on \( \widetilde{\partial} \).

By abuse of notation, we denote by \( \mathrm{Hyp}_0 \), \( \mathrm{Par}_0 \), \( \mathrm{Par}_0^+ \), and \( \mathrm{Par}_0^- \) the respective lifts to \( \widetilde{G} \) of the sets \( \mathrm{Hyp}_0 \), \( \mathrm{Par} \cap \{ \mathrm{tr} = 2 \} \), \( \mathrm{Par}^+ \), and \( \mathrm{Par}^- \subset \mathrm{SL}(2,\mathbb{R}) \), each of whose closures contain the identity element, where $\mr{Par}^{\pm}=\sigma^{-1}(\mp 1)\cap \mr{Par}$. We define the higher lifts by
\[
\mathrm{Hyp}_k := z^k \mathrm{Hyp}_0, \qquad \mathrm{Par}_k^\pm := z^k \mathrm{Par}_0^\pm.
\]
Note that when \( |k| \) is odd, the trace of any element \( A \in \mathrm{Hyp}_k \cup \mathrm{Par}_k \), defined via the projection to \( \mathrm{SL}(2,\mathbb{R}) \), is negative. See Figure~\ref{universal}.

\begin{figure}[ht]
\centering
\includegraphics[width=0.75\textwidth]{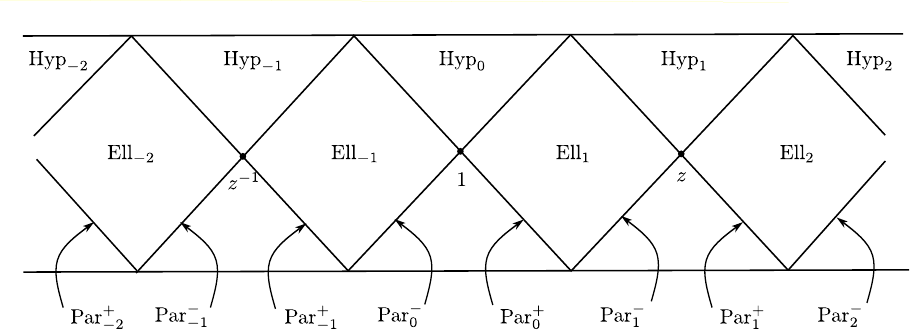}
\caption{Universal cover of \( \mathrm{SL}(2,\mathbb{R}) \)}
\label{universal}
\end{figure}

In \cite{Gold1}, W.~Goldman computed the relative Euler class of a representation \( \phi: \pi_1(\Sigma) \to G = \mathrm{PSL}(2,\mathbb{R}) \), assuming that \( \phi \) is non-elliptic along the boundary. This class lies in the cohomology group
\[
e(\phi)\in \mr{H}^2(\Sigma, \partial \Sigma; \mathbb{Z}) \cong \mathbb{Z}.
\]
Since each boundary holonomy is either hyperbolic or parabolic, the associated flat principal \( \mathrm{PSL}(2,\mathbb{R}) \)-bundle admits a canonical trivialization along each boundary component via one-parameter subgroups. Concretely, for boundary holonomy \( h \), the canonical section \( f:[0,1] \to G \) is given by \( f(t) = h_t^{-1} \), where \( \{h_t\} \) is a one-parameter subgroup with \( h_0 = \mathrm{id} \) and \( h_1 = h \). This section satisfies the equivariance condition $f(0) = \mathrm{id}$ and $f(1) = h^{-1}$.

The obstruction to extending this boundary trivialization to the interior defines a class in \( \mr{H}^2(\Sigma, \partial \Sigma; \mathbb{Z}) \). One computes this class explicitly using lifts to the universal cover \( \widetilde{G} \). Let \( \widetilde{\phi(C_j)} \) be the unique lift of \( \phi(C_j) \) lying in \( \{I\} \cup \mathrm{Hyp}_0 \cup \mathrm{Par}_0 \subset \widetilde{G} \). Then for arbitrary lifts \( \widetilde{\phi(A_i)}, \widetilde{\phi(B_i)} \in \widetilde{G} \), we have the relation:
\begin{equation*} 
\prod_{i=1}^g [\widetilde{\phi(A_i)}, \widetilde{\phi(B_i)}] \prod_{j=1}^n \widetilde{\phi(C_j)} = z^m,
\end{equation*}
where \( z \) is the fixed generator of the center of \( \widetilde{G} \) that lifts \( -I \in \mathrm{SL}(2,\mathbb{R}) \). The integer \( m \in \mathbb{Z} \) is the {\it relative Euler class} of the representation \( \phi \), denoted \( e(\phi) \).

As shown in \cite{KPW1}, the relative Euler class agrees with the Toledo invariant for such representations:
\begin{prop}[{\cite[Proposition 1.2]{KPW1}}]\label{pre-prop1}
The Toledo invariant and the relative Euler class of representations that are non-elliptic on the boundary coincide.
\end{prop}

\subsection{Rho invariant}\label{defn of rho}

We now define the rho invariant associated to the group \( \mathrm{SL}(2,\mathbb{R}) \). The following definition is due to \cite[Definition 4.1]{KPW}.

\begin{defn}
For \( L \in \mathrm{SL}(2,\mathbb{R}) \), define a flat vector bundle \( \mathcal{E} \to S^1 \). Let \( \mathbf{J} \in \mathcal{J}(\mathcal{E}, \Omega) \) be a section of compatible complex structures. Choose a fixed point \( W \) of \( L \) in the closure \( \overline{\mathrm{D}} \) of the Poincar\'e disk \( \mathrm{D} \), and let \( \alpha_W \) be the associated \( L \)-invariant primitive. The \emph{rho invariant} is defined by:
\[
\rho(L, \mathbf{J}, W) = \iota(L, \mathbf{J}, W) + \eta(L, \mathbf{J}, W) := -\frac{1}{\pi} \int_{S^1} \widetilde{\mathbf{J}}^* \alpha_W + \eta(A_{\mathbf{J}}),
\]
where \( \eta(A_{\mathbf{J}}) \) is the eta invariant of a self-adjoint elliptic operator defined by \( \mathbf{J} \).
\end{defn}

By \cite[Lemma 4.2, Lemma 4.3, and Corollary 4.4]{KPW}, the quantity \( \rho(L, \mathbf{J}, W) \) does not depend on the choices of \( \mathbf{J} \) and \( W \). Hence, we simply write:
\[
\rho(L) := \rho(L, \mathbf{J}, W).
\]

The rho invariant for the boundary of a representation is defined as follows:

\begin{defn}
For a representation \( \phi: \pi_1(\Sigma) \to \mathrm{SL}(2,\mathbb{R}) \), the rho invariant of the boundary is defined as:
\[
\rro(\phi) := \rro_\phi(\partial \Sigma) = \sum_{\text{boundary component } c} \rho(\phi(c)).
\]
\end{defn}

Let \( L \in \mathrm{SL}(2,\mathbb{R}) \) be a matrix with eigenvalue \( \lambda \). If \( \lambda \in S^1 \setminus \{\pm 1\} \), then \( L \) is conjugate to a rotation matrix $R(\theta)$, where $\theta \in (0,\pi) \cup (\pi, 2\pi)$.
Otherwise, \( L \) is conjugate to
\(
\begin{pmatrix}
\lambda & \mu \\
0 & \lambda^{-1}
\end{pmatrix},
\)
for some \( \mu \in \mathbb{R} \). The values of the rho invariant for various types of matrices in \( \mathrm{SL}(2,\mathbb{R}) \) are summarized in Table \ref{tab:rho-invariant}.
\begin{table}[htbp]
\centering
\renewcommand{\arraystretch}{1.2}
\caption{The rho invariant for matrices in \( \mathrm{SL}(2,\mathbb{R}) \)}
\begin{tabular}{|c|c|}
\hline
Matrix Type & \( \rho(L) \) \\
\hline
\( \lambda \notin S^1 \) & $0$ \\
\hline
\( \lambda \in S^1 \setminus \{\pm 1\} \) & \( 2\left(1 - \frac{\theta}{\pi}\right) \) \\
\hline
\( \lambda = 1,\, \mu = 0 \) & $0$ \\
\hline
\( \lambda = 1,\, \mu > 0 \) & $-1$ \\
\hline
\( \lambda = 1,\, \mu < 0 \) & $1$ \\
\hline
\( \lambda = -1 \) & $0$ \\
\hline
\end{tabular}

\label{tab:rho-invariant}
\end{table}

By \cite[Theorem 1]{KPW}, for a compact oriented surface \( \Sigma \) with nonempty boundary and a representation \( \phi: \pi_1(\Sigma) \to \mathrm{SL}(2,\mathbb{R}) \), the signature, Toledo invariant, and rho invariant satisfy the relation:
\begin{equation} 
\mathrm{sign}(\phi) = 2\, \mathrm{T}(\phi) + \rro(\phi).
\end{equation}
Moreover, the signature satisfies the following Milnor–Wood type inequality:
\begin{equation} \label{sig-MW}
|\mathrm{sign}(\phi)| \leq 2 |\chi(\Sigma)|,
\end{equation}
see \cite[Theorem 4]{KPW}.

Note that the Toledo invariant \( \mathrm{T}(\phi) \) is continuous (see \cite[Theorem 1]{BIW}), and the rho invariant $\rro(\phi)$ is continuous away from elements with eigenvalue 1 (cf. \cite[Section 4]{KPW}). Consequently, the signature \( \mathrm{sign}(\phi) \) is locally constant and only jumps when the boundary holonomy acquires 1 as an eigenvalue.

\section{Boundary elliptic representations}\label{sec-elliptic}

In this section, we will study surface group representations
$\phi:\pi_1(\Sigma_{g,n}) \to \mathrm{SL}(2,\mathbb{R})$, where $\Sigma_{g,n}$ denotes a surface of genus $g$ with $n$ boundary components. A representation $\phi \in \mathrm{Hom}(\pi_1(\Sigma_{g,n}), \mathrm{SL}(2,\mathbb{R}))$ is called a {\it boundary elliptic} (resp. {\it boundary elliptic-unipotent}) representation if, for each boundary component $c_i \subset \partial \Sigma_{g,n} = \bigcup_{i=1}^n c_i$, the image $\phi(c_i)$ is conjugate to a rotation matrix
$$
R(\theta_i) = \begin{pmatrix}
  \cos\theta_i & -\sin\theta_i \\
  \sin\theta_i & \cos\theta_i
\end{pmatrix} \in \mathrm{SO}(2)\backslash\{\pm I\},
$$
for some angle $\theta_i \in (0,\pi) \cup (\pi,2\pi)$ (resp. $\theta_i \in (0,2\pi)$), see \cite[Definition 4.5]{KPW}.

We denote by $\mathrm{Ell}(\Sigma_{g,n})$ (resp. $\mathrm{E}(\Sigma_{g,n})$) the set of all boundary elliptic (resp. boundary elliptic-unipotent) representations. The goal of this section is to determine the number of connected components of $\mathrm{Ell}(\Sigma_{g,n})$ and $\mathrm{E}(\Sigma_{g,n})$.

\subsection{The case of $\mathrm{SO}(2)$}\label{sec:SO(2)}

In this subsection, we consider a simplified setting, namely, boundary elliptic (resp. boundary elliptic-unipotent) representations whose images lie in $\mathrm{SO}(2)$.

\subsubsection{Boundary elliptic-unipotent representations}

Let $\mathrm{E}(\Sigma)$ denote the space of all representations $\phi \in \mathrm{Hom}(\pi_1(\Sigma), \mathrm{SO}(2))$ such that the image of each boundary loop lies in $\mathrm{SO}(2) \setminus \{I\}$, where $\Sigma = \Sigma_{g,n}$ is a compact surface of genus $g$ with $n$ boundary components.

We define the following {\it signature map}:
$$
\mathrm{sign}: \mathrm{Hom}(\pi_1(\Sigma), \mathrm{SO}(2)) \to \mathbb{Z}.
$$
Since the rho invariant is continuous on $\mathrm{SO}(2) \setminus \{I\}$ (see \cite[Section 4]{KPW}), the boundary rho invariant $\boldsymbol{\rho}(\phi_t)$ varies continuously for any path $\{\phi_t\}_{t\in [0,1]} \subset \mathrm{E}(\Sigma)$. On the other hand, the Toledo invariant vanishes for every representation $\phi \in \mathrm{Hom}(\pi_1(\Sigma), \mathrm{SO}(2))$ (see \cite[Section 6.2]{KPW}). Therefore, for all such paths, we have
$$
\mathrm{sign}(\phi_t) =\boldsymbol{\rho}(\phi_t) \in \mathbb{Z},
$$
and it remains constant along the path. In particular, representations in $\mathrm{E}(\Sigma)$ with different signature values cannot lie in the same path-connected component.

Note that when $n = 1$, $\mathrm{E}(\Sigma) = \emptyset$. Hence, we assume $n \geq 2$ in the following. 
A simple calculation shows that
\begin{equation}\label{eqn1}
\mathrm{sign}(\mathrm{Hom}(\pi_1(\Sigma), \mathrm{SO}(2))) = [-2(n-2), 2(n-2)] \cap 2\mathbb{Z}.
\end{equation}

We now turn to boundary elliptic-unipotent representations and obtain the following result:

\begin{prop}\label{prop-sign1}
We have
$$
\mathrm{sign}(\mathrm{E}(\Sigma)) = [-2(n-2), 2(n-2)] \cap \{2n - 4\mathbb{Z}\}.
$$
\end{prop}
\begin{proof}
Let $m \in [-(n-2), n-2] \cap \{n-2 \mathbb{Z}\}$, so we can write
$m = n - 2a,$
for some integer $a \in [1, n-1]$. Define a representation $\phi$ by
$
\phi(c_i) = R(\theta_i),  \theta_i \in (0, 2\pi),
$
such that
$
\sum_{i=1}^n \theta_i = 2a\pi.
$
 Then $\phi \in \mathrm{E}(\Sigma)$, and the signature is given by
$$
\mathrm{sign}(\phi) = \sum_{i=1}^n 2(1 - \tfrac{\theta_i}{\pi} ) = 2(n - 2a )= 2m.
$$

Conversely, for any $\phi \in \mathrm{E}(\Sigma) \subset \mathrm{Hom}(\pi_1(\Sigma), \mathrm{SO}(2))$, it follows from \eqref{eqn1} that
$$
\mathrm{sign}(\phi) \in [-2(n - 2), 2(n - 2)].
$$
Moreover, since each boundary image satisfies $\phi(c_i) = R(\theta_i)$ with $\theta_i \in (0, 2\pi)$, and the total angle satisfies
$
\sum_{i=1}^n \theta_i \in 2\pi \mathbb{Z}.
$
We compute
$$
\mathrm{sign}(\phi) = \sum_{i=1}^n 2(1 - \tfrac{\theta_i}{\pi}) = 2n - \tfrac{2}{\pi} \sum_{i=1}^n \theta_i \in \{2n - 4\mathbb{Z}\}.
$$
This completes the proof. 
\end{proof}

We next show that any two representations in $\mathrm{E}(\Sigma)$ with the same signature lie in the same connected component.
\begin{prop}\label{ell-prop1}
Let $m \in [-(n-2), n-2] \cap \{n-2 \mathbb{Z}\}$. Then the level set
$
\mathrm{sign}^{-1}(2m) \cap \mathrm{E}(\Sigma)
$
is a connected component of $\mathrm{E}(\Sigma)$.\end{prop}
\begin{proof}
From Proposition \ref{prop-sign1}, the level set $\mathrm{sign}^{-1}(2m) \cap \mathrm{E}(\Sigma)$ is nonempty.
Let $\phi_0, \phi_1 \in \mathrm{E}(\Sigma)$ be two representations with the same signature. We write
$$
\phi_0(c_i) = R(\theta_i), \quad \phi_1(c_i) = R(\tau_i), \quad \theta_i, \tau_i \in (0, 2\pi),
$$
such that
$
\sum_{i=1}^n \theta_i = \sum_{i=1}^n \tau_i \in 2\pi \mathbb{Z}.
$
Define a continuous path $\{\phi_t\}_{t\in [0,1]}$ by
$$
\phi_t(c_i) = R((1 - t)\theta_i + t\tau_i), \quad 1 \leq i \leq n.
$$
Since each interpolated angle lies in $(0, 2\pi)$ and their sum remains invariant, it follows that $\phi_t \in \mathrm{E}(\Sigma)$ for all $t \in [0,1]$, and $\phi_t$ provides a path connecting $\phi_0$ and $\phi_1$ in $\mathrm{E}(\Sigma)$.
\end{proof}
\begin{cor}
The set $\mathrm{E}(\Sigma)\subset \mathrm{Hom}(\pi_1(\Sigma),\mathrm{SO}(2))$ has exactly $n - 1$ connected components.
\end{cor}

\subsubsection{Boundary elliptic representations}
Having considered boundary elliptic-unipotent representations, we now turn our attention to boundary elliptic representations.
 Let $\mathrm{Ell}(\Sigma)$ denote the space of all representations $\phi \in \mathrm{Hom}(\pi_1(\Sigma), \mathrm{SO}(2))$ such that for each boundary component $c_i \subset \partial\Sigma$, the image $\phi(c_i) = R(\theta_i)$, with $\theta_i \in (0,\pi) \cup (\pi, 2\pi)$. In this case, we again have
$$
\mathrm{sign}(\mr{Ell}(\Sigma)) =\mathrm{sign}(\mathrm{E}(\Sigma)) = [-2(n-2), 2(n-2)] \cap \{2n - 4\mathbb{Z}\}.
$$
Define a {\it sign vector map}
$$
\sigma: \mr{Ell}(\Sigma) \to \{-1, 1\}^n, \quad \sigma(\phi) := \left(\mathrm{sgn}(\theta_1 - \pi), \dots, {\mathrm{sgn}}(\theta_n - \pi) \right),
$$
where $\phi(c_i) = R(\theta_i)$. Clearly, if two representations $\phi, \psi \in \mr{Ell}(\Sigma)$ have different $\sigma$-values, then they lie in different connected components of $\mr{Ell}(\Sigma)$.
We now consider representations sharing the same sign vector.

\begin{thm}\label{thm4}
Let $a \in \{-1, 1\}^n$, and let $r_a := \#\{i : a_i = -1\}$. Then, for any
$$
m \in [-n + r_a + 1, \, r_a - 1] \cap \{n - 2\mathbb{Z}\},
$$
the subset $\sigma^{-1}(a) \cap \mathrm{sign}^{-1}(2m) \cap \mr{Ell}(\Sigma)$ is a connected component. For all other values of $m$, this intersection is empty.
\end{thm}
\begin{proof}
If $\phi \in \sigma^{-1}(a) \cap \mr{Ell}(\Sigma)$, then
	\begin{align*}
\tfrac{1}{2}  \mathrm{sign}(\phi)=n-\tfrac{1}{\pi}\sum_{i=1}^n\theta_i&\in \{n-((n-r_a),r_a+2(n-r_a))\}\cap \{n-2\mb{Z}\}\\
  &=(-n+r_a,r_a)\cap \{n-2\mb{Z}\}\\
  &=[-n+r_a+1,r_a-1]\cap \{n-2\mb{Z}\}.
\end{align*}
Therefore, if $m \not\in [-n + r_a + 1, \, r_a - 1] \cap \{n - 2\mathbb{Z}\}$, then
$$
\sigma^{-1}(a) \cap \mathrm{sign}^{-1}(2m) \cap \mr{Ell}(\Sigma) = \emptyset.
$$
Now fix $m \in [-n + r_a + 1, \, r_a - 1] \cap \{n - 2\mathbb{Z}\}$. It is straightforward to construct a representation $\phi \in \sigma^{-1}(a) \cap \mr{Ell}(\Sigma)$ such that $\mathrm{sign}(\phi) = 2m$, hence the intersection is nonempty.

To prove path-connectedness, let $\phi_0, \phi_1 \in \sigma^{-1}(a) \cap \mathrm{sign}^{-1}(2m) \cap \mr{Ell}(\Sigma)$. Without loss of generality, write:
$$
\phi_0(c_i) = R(\theta_i), \, \phi_1(c_i) = R(\tau_i), \, \text{where } \theta_i, \tau_i \in
\begin{cases}
(0, \pi) & \text{for } 1 \leq i \leq r_a, \\
(\pi, 2\pi) & \text{for } r_a + 1 \leq i \leq n,
\end{cases}
$$
with
$
\sum_{i=1}^n \theta_i = \sum_{i=1}^n \tau_i = (n - m)\pi.
$
Define the interpolating path $\phi_t$ by
$$
\phi_t(c_i) := R\big( (1 - t)\theta_i + t \tau_i \big), \quad 1 \leq i \leq n.
$$
Each interpolated angle $(1 - t)\theta_i + t \tau_i $ in $(0,\pi)$ for $i\leq r_a$, and in $(\pi,2\pi)$ for $i>r_a$, and the total angle remains constant,
$
\sum_{i=1}^n \left[ (1 - t)\theta_i + t \tau_i \right] = (n - m)\pi.
$
Hence, the path $\{\phi_t\}_{t \in [0,1]}$ lies entirely within $\sigma^{-1}(a) \cap \mathrm{sign}^{-1}(2m) \cap \mr{Ell}(\Sigma)$, completing the proof. \end{proof}

We now compute the number of connected components of $\mr{Ell}(\Sigma)$, we denote it by $\#\{\mr{Ell}(\Sigma)\}$, then
\begin{align*}
\#\{\mr{Ell}(\Sigma)\} &:= \#\left\{ (a, m) \in \{-1, 1\}^n \times \left( [-n + r_a + 1, \, r_a - 1] \cap \{n - 2\mathbb{Z}\} \right) \right\}\\
&=\sum_{a \in \{-1, 1\}^n} \left\lfloor \tfrac{n - 1 + (r_a \bmod 2)}{2} \right\rfloor = (n - 1) \cdot 2^{n - 1}.
\end{align*}
Thus, we conclude:

\begin{cor}
The number of connected components of $\mr{Ell}(\Sigma)$ is $(n - 1)\cdot 2^{n - 1}$.
\end{cor}

\subsection{The case of genus zero}

In this section, we study boundary elliptic (respectively, boundary elliptic-unipotent) representations into $\mathrm{SL}(2,\mathbb{R})$ for punctured spheres $\Sigma = \Sigma_{0,n}$.

\subsubsection{Boundary elliptic-unipotent representations}Note that
$$
\mathrm{Ell}(\Sigma) = \left\{ \phi \in \mathrm{E}(\Sigma) \,\middle|\, \det(\phi(c_i) + I) \neq 0,\; 1 \leq i \leq n \right\}.
$$
Therefore, $\mathrm{Ell}(\Sigma)$ is a dense subset of $\mathrm{E}(\Sigma)$. In particular, any boundary elliptic-unipotent representation can be deformed into a boundary elliptic representation by a small perturbation, while the signature remains unchanged under such deformation.
By \cite[Theorem 4]{KPW1}, we have
\begin{equation}\label{eqn2}
\mathrm{sign}(\mathrm{E}(\Sigma)) = \{2n - 4\mathbb{Z}\} \cap [-2(n - 2), 2(n - 2)].
\end{equation}
Next, we will show that any two representations in $\mathrm{E}(\Sigma)$ with the same signature lie in the same connected component.
\begin{lemma}\label{lemma1}
Let $A = R(\theta) \in \mathrm{SO}(2)$, where $\theta \in (0, \pi) \cup (\pi, 2\pi)$, and let
$$
B = P R(\theta_1) P^{-1} \in \mathrm{SL}(2, \mathbb{R}), \quad \text{with } P = \begin{pmatrix}
a & b \\
c & d
\end{pmatrix} \in \mathrm{SL}(2, \mathbb{R}),
$$
so that $B$ is conjugate to $R(\theta_1)$, where $\theta_1 \in (0, \pi) \cup (\pi, 2\pi)$. Assume that $\theta + \theta_1 \neq 2\pi$, and $|\mathrm{tr}(AB)| < 2$. Then there exists a continuous path $P(t) \in \mathrm{SL}(2, \mathbb{R})$, defined for $t \in [0,1]$, such that
$
P(0) = P, P(1) = I,
$
and for all $t \in [0,1]$,
$
\left| \mathrm{tr}( A B(t) ) \right| < 2,
$
where $B(t) := P(t) R(\theta_1) P(t)^{-1}$.
\end{lemma}

\begin{proof}
By direct computation,
\begin{equation}\label{eqn5}
\mathrm{tr}(AB) = 2\cos\theta\cos\theta_1 - S\sin\theta\sin\theta_1,
\end{equation}
where
\(
S := a^2 + b^2 + c^2 + d^2 = (a - d)^2 + (b + c)^2 + 2(ad - bc) \geq 2,
\)
since \( \det P = 1 \). Hence \( |\mathrm{tr}(AB)| < 2 \) is equivalent to \( S < S_{\max} \), where
\[
S_{\max} :=
\begin{cases}
\displaystyle \frac{2 + 2\cos\theta\cos\theta_1}{\sin\theta\sin\theta_1} & \text{if } \sin\theta\sin\theta_1 > 0, \\
\displaystyle \frac{-2 + 2\cos\theta\cos\theta_1}{\sin\theta\sin\theta_1} & \text{if } \sin\theta\sin\theta_1 < 0.
\end{cases}
\]

Now write \( P \) using its Iwasawa decomposition:
\[
P =
\begin{pmatrix}
\cos\tau & -\sin\tau \\
\sin\tau & \cos\tau
\end{pmatrix}
\begin{pmatrix}
r & 0 \\
0 & \frac{1}{r}
\end{pmatrix}
\begin{pmatrix}
1 & x \\
0 & 1
\end{pmatrix}, \quad r > 0.
\]
Then
\[
S = \mathrm{tr}(P^\top P) = r^2 + \frac{1}{r^2} + r^2 x^2.
\]
Define a continuous path \( P(t) \in \mathrm{SL}(2, \mathbb{R}) \) from \( P \) to \( I \) as follows:\\
\textbf{Step 1:} For \( t \in [0, \tfrac{1}{2}] \), set
\[
P(t) =
\begin{pmatrix}
\cos((1 - 2t)\tau) & -\sin((1 - 2t)\tau) \\
\sin((1 - 2t)\tau) & \cos((1 - 2t)\tau)
\end{pmatrix}
\begin{pmatrix}
r & 0 \\
0 & \frac{1}{r}
\end{pmatrix}
\begin{pmatrix}
1 & (1 - 2t)x \\
0 & 1
\end{pmatrix}.
\]
Then \( P(0) = P \), \( P(\tfrac{1}{2}) = \mathrm{diag}(r, 1/r) \), and
\[
S(t) = r^2 + \frac{1}{r^2} + r^2 (1 - 2t)^2 x^2 \leq S(0) < S_{\max}.
\]
\textbf{Step 2:} For \( t \in [\tfrac{1}{2}, 1] \), define
\[
P(t) = \begin{pmatrix}
\lambda(t) & 0 \\
0 & \lambda(t)^{-1}
\end{pmatrix}, \quad \lambda(t) := 2(1 - t)r + 2t - 1.
\]
Then \( P(1) = I \), and
\[
S(t) = \lambda(t)^2 + \frac{1}{\lambda(t)^2} \leq r^2 + \frac{1}{r^2} \leq S(0) < S_{\max}.
\]
Therefore, for all \( t \in [0,1] \), we have \( |\mathrm{tr}(AB(t))| < 2 \), as required.
\end{proof}

\begin{prop}\label{ell-prop2}
	For any $m\in [-(n-2),n-2]\cap \{n-2\mb{Z}\}$, $\mathrm{sign}^{-1}(2m)\cap \mathrm{E}(\Sigma)$ is a connected component. 
\end{prop}
\begin{proof}
We proceed by induction on the number of boundary components \( n \). The case \( n = 2 \) is trivial. Assume the statement holds for all \( l \leq n-1 \), that is, for any \( m \in [-(l-2), l-2] \cap \{l-2\mathbb{Z}\} \), the space \( \mathrm{sign}^{-1}(2m) \cap \mathrm{E}(\Sigma_{0,l}) \) is path-connected.

Let \( \phi \in \mathrm{sign}^{-1}(2m) \cap \mathrm{Ell}(\Sigma) \). Suppose
\[
\phi(c_i) = C_i = P_i R(\theta_i) P_i^{-1}, \theta_i \in (0,\pi) \cup (\pi, 2\pi),  P_i \in \mathrm{SL}(2, \mathbb{R}), \quad 1 \leq i \leq n.
\]
By a small perturbation, we may assume \( \mathrm{tr}(C_1 \cdots C_{n-2}) \neq \pm 2 \).

If \( |\mathrm{tr}(C_1 \cdots C_{n-2})| < 2 \), then the product is elliptic, and we can deform \( \phi \) within \( \mathrm{Ell}(\Sigma) \) while keeping the conjugacy classes of $C_1,\cdots, C_{n-2}$ fixed, such that the product \( C_1 \cdots C_{n-2}\) in \( \mathrm{SO}(2) \setminus \{\pm I\} \). By Lemma~\ref{lemma1}, we can fix \( C_1,\ldots,C_{n-2} \) and deform \( C_{n-1} \) into \( \mathrm{SO}(2) \setminus \{\pm I\} \) through a path \( C_{n-1}(t) \) satisfying \( |\mathrm{tr}(C_1 \cdots C_{n-2} C_{n-1}(t))| < 2 \). Then \( C_n(t) := (C_1 \cdots C_{n-2} C_{n-1}(t))^{-1}\)satisfying \(C_n(1) \in \mathrm{SO}(2) \setminus \{\pm I\} \), and the resulting deformed representation \( \phi' \) satisfies \( \phi'(c_{n-1}), \phi'(c_n) \in \mathrm{SO}(2) \setminus \{\pm I\} \).

Cutting the surface along a simple closed curve \( c_{0}\), so that \( \{c_{n-1}, c_n\} \) and \( \{c_1,\ldots,c_{n-2}\} \) lie in different components. Then \( \phi' \) splits into two representations \( \phi_1', \phi_2' \) with
\[
\mathrm{sign}(\phi') = \mathrm{sign}(\phi_1') + \mathrm{sign}(\phi_2').
\]
Since \( \phi_1'(c_0) = (\phi(c_{n-1})\phi(c_n))^{-1} \in \mathrm{SO}(2) \setminus \{\pm I\} \), \( \phi_1' \) is a boundary elliptic representation. By the induction hypothesis, we can deform \( \phi_2' \) into a representation in \( \mathrm{Hom}(\pi_1(\Sigma), \mathrm{SO}(2)) \). For the pair-of-pants subsurface with boundary \( \{c_0,c_{n-1}, c_n\} \), we can deform \( \phi_1' \) along a path \( \phi_1'(t) \) such that \( \phi_1'(t)(c_0) = \phi_2'(t)(c_0)^{-1} \), and glue the two deformations into a path \( \phi'(t) = \phi_1'(t) \cup \phi_2'(t) \), ending in a representation in \(\mathrm{Hom}(\pi_1(\Sigma), \mathrm{SO}(2)) \).

Now consider the case \( |\mathrm{tr}(C_1 \cdots C_{n-2})| > 2 \), so the product is hyperbolic. Without loss of generality, we assume \( C_{n-2} \in \mathrm{SO}(2) \setminus \{\pm I\} \). Write
\[
C_1\cdots C_{n-2} = P \begin{pmatrix}
\lambda & 0 \\
0 & \lambda^{-1}
\end{pmatrix} P^{-1}, \quad \lambda \neq 0, 1.
\]
We aim to deform \( C_{n-1} \) into \( \mathrm{SO}(2) \setminus \{\pm I\} \), while maintaining \( C_n(t) \) elliptic, that is,
\[
|\mathrm{tr}(C_1 \cdots C_{n-2} C_{n-1}(t))| < 2.
\]

Let \( P^{-1} P_{n-1} = \begin{pmatrix} a & b \\ c & d \end{pmatrix} \in \mathrm{SL}(2, \mathbb{R}) \), and write
\begin{equation}\label{eqn16}
\mathrm{tr}(C_1 \cdots C_{n-2} C_{n-1}) = R \cos(\theta_{n-1} - \alpha),
\end{equation}
where
\[
R = \sqrt{(\lambda + \lambda^{-1})^2 + (\lambda - \lambda^{-1})^2(ac + bd)^2} > 2,
\]
and \( \alpha \in [0,2\pi) \) is determined by
\[
\cos\alpha = \tfrac{\lambda + \lambda^{-1}}{R}, \quad \sin\alpha = \tfrac{(\lambda - \lambda^{-1})(ac + bd)}{R}.
\]

Let $P_{n-1}(t)\in \mathrm{SL}(2,\mb{R})$ be a path connecting \(P_{n-1} \) and an element in \( \mathrm{SO}(2) \).
Next, we construct a continuous path \( \theta(t) \) with \( \theta(0) = \theta_{n-1} \), so that
\[
C_{n-1}(t) = P_{n-1}(t) R(\theta(t)) P_{n-1}(t)^{-1}
\]
satisfies the inequality
\begin{equation}\label{eqn3}
|\mathrm{tr}(C_1 \cdots C_{n-2} C_{n-1}(t))| < 2.
\end{equation}

The term \( (ac + bd)(t) \) from \( P^{-1} P_{n-1}(t) \) depends continuously on \( t \), as does the resulting quantity \( R(t) \). Therefore, we seek \( \theta(t) \) such that
\[
-2 < R(t)\cos(\theta(t) - \alpha(t)) < 2.
\]
Define \( \beta(t) := \arccos\left( \frac{2}{R(t)} \right) \in \left(0, \frac{\pi}{2} \right) \), and let
\begin{align*}
S_t &= \{ x \in [0, 2\pi) : |R(t)\cos(x - \alpha(t))| < 2 \} \\
&= \left( \alpha(t) + \beta(t), \alpha(t) + \pi - \beta(t) \right) \cup \left( \alpha(t) - \pi + \beta(t), \alpha(t) - \beta(t) \right) \mod 2\pi \\
&=: I_t^1 \cup I_t^2.
\end{align*}

The sets \( I_t^1 \) and \( I_t^2 \) are two open intervals depending continuously on \( t \). Since \( \theta(0) = \theta_{n-1} \in S_0 \), we may construct a continuous path \( \theta(t) \) staying within the same component \( I_t^1 \) or \( I_t^2 \), thereby ensuring that \eqref{eqn3} holds for all \( t \in [0,1] \).

If \( \theta(t) \neq 0, \pi \) for all \( t \in [0,1] \), then \( \theta(t) \in (0, \pi)\cup(\pi,2\pi) \), and thus \( C_{n-1}(1) \in \mathrm{SO}(2) \setminus \{\pm I\} \). If \( \sin\theta(t) = 0 \) for some \( t \in (0,1] \), let \( t_0 \) be the first such point. Then for some \( t_1 < t_0 \), we can ensure that
\begin{align*}
&\quad |\mathrm{tr}(C_{n-2} C_{n-1}(t_1))|\\
 &= |2\cos\theta_{n-2}\cos\theta(t_1) - \mathrm{tr}(P_{n-1}(t_1)^\top P_{n-1}(t_1)) \sin\theta_{n-2} \sin\theta(t_1)|< 2,
\end{align*}
since \( \sin\theta(t_1) \) can be made arbitrarily small and  \( \mathrm{tr}(P_{n-1}(t)^\top P_{n-1}(t)) \) remains bounded.

Thus, after a slight perturbation and a reparameterization, we may deform $\phi$ to a new representation $\phi'$ such that $\phi'(c_{n-1}) = C_{n-1}(1) \in \mathrm{SO}(2) \setminus \{\pm I\}$, or $C_{n-1}(1)$ is sufficiently close to $\pm I$ so that $|\mathrm{tr}(C_{n-2} C_{n-1}(1))| < 2$. We also ensure that $\phi'(c_{n-2}) = \phi(c_{n-2}) \in \mathrm{SO}(2) \setminus \{\pm I\}$.

As in the elliptic case, we cut the surface along a simple closed curve $c_0$, which separates $\{c_{n-2}, c_{n-1}\}$ from $\{c_n,c_1, \ldots, c_{n-3}\}$. By Lemma \ref{lemma1}, the restriction $\phi'_1$ of $\phi'$ to the subgroup generated by $\{c_{n-2}, c_{n-1}\}$ can be deformed along a path $\phi'_1(t)$ to a representation $\phi'_1(1) \in \mathrm{Hom}(\pi_1(\Sigma), \mathrm{SO}(2))$, and preserves conjugacy class of $\phi'_1(c_0)$. Hence, we may simultaneously deform the representation $\phi'_2$ on the complementary subsurface along a path $\phi'_2(t)$ so that the gluing condition $\phi'_1(t)(c_0) = \phi'_2(t)(c_0)^{-1}$ is satisfied at each time $t$.

Gluing the deformations $\phi'_1(t)$ and $\phi'_2(t)$ together yields a path of representations $\phi'(t) := \phi'_1(t) \cup \phi'_2(t)$ with initial condition $\phi'(0) = \phi'$, and such that $\phi'(1)(c_{n-2}), \phi'(1)(c_{n-1}) \in \mathrm{SO}(2) \setminus \{\pm I\}$.

By induction, we can further deform $\phi'_2(1)$ along a path $\phi''_2(t)$ into a representation in $\mathrm{Hom}(\pi_1(\Sigma), \mathrm{SO}(2))$. Accordingly, we also deform $\phi'_1(1)$ along a path $\phi''_1(t)$ within $\mathrm{Hom}(\pi_1(\Sigma), \mathrm{SO}(2))$ so that $\phi''_1(t)(c_0) = \phi''_2(t)(c_0)^{-1}$ at each time $t$. Gluing these two paths together gives a deformation $\phi''(t)$ such that $\phi''(1) \in \mathrm{Hom}(\pi_1(\Sigma), \mathrm{SO}(2))$.

Since the set of elliptic representations $\mathrm{Ell}(\Sigma)$ is dense in $\mathrm{E}(\Sigma)$, any representation $\phi \in \mathrm{sign}^{-1}(2m) \cap \mathrm{E}(\Sigma)$ can be deformed into a representation in $\mathrm{Hom}(\pi_1(\Sigma), \mathrm{SO}(2))$. By Proposition \ref{ell-prop1}, any two boundary elliptic-unipotent representations in $\mathrm{Hom}(\pi_1(\Sigma), \mathrm{SO}(2))$ can be connected by a continuous path. This completes the proof.
\end{proof}
As a corollary, we obtain 
\begin{cor}
The number of connected components of $\mathrm{E}(\Sigma)$ is $n-1$.
\end{cor}

\subsubsection{Boundary elliptic representations}

Similarly, one may consider boundary elliptic representations. According to \cite[Theorem 4]{KPW1}, we have
\[
  \mathrm{sign}(\mathrm{Ell}(\Sigma)) = \mathrm{sign}(\mathrm{E}(\Sigma)) = [-2(n-2), 2(n-2)] \cap \{2n - 4\mathbb{Z}\}.
\]
Define a map \( \sigma: \mathrm{Ell}(\Sigma) \to \{\pm 1\}^n \) by
\[
  \sigma(\phi) = ({\mathrm{sgn}}(\theta_1 - \pi), \ldots, {\mathrm{sgn}}(\theta_n - \pi)),
\]
where \( \phi(c_i) = R(\theta_i) \). It is clear that if \( \phi, \psi \in \mathrm{Ell}(\Sigma) \) have different values under \( \sigma \), then they lie in different path components of \( \mathrm{Ell}(\Sigma) \). Using the same arguments as in the case of boundary elliptic representations into \( \mathrm{SO}(2) \), we obtain the following result:

\begin{prop}
Let \( a \in \{-1, 1\}^n \), and denote by \( r_a = \#\{ i : a_i = -1 \} \). Then for any
\[
  m \in [-n + r_a + 1, r_a - 1] \cap \{n - 2\mathbb{Z}\},
\]
the set \( \sigma^{-1}(a) \cap \mathrm{sign}^{-1}(2m) \cap \mathrm{Ell}(\Sigma) \) is a connected component of \( \mathrm{Ell}(\Sigma) \). For all other \( m \), this intersection is empty.
\end{prop}

\begin{cor}
The number of connected components of \( \mathrm{Ell}(\Sigma) \) is
\(
(n - 1) \cdot 2^{n - 1}.
\)
\end{cor}

\subsection{The case of genus one}

In this section, we assume that \( g = 1 \).

\subsubsection{The case \( n = 1 \)} \label{secn=1}
We first consider the once-punctured torus \( \Sigma = \Sigma_{1,1} \). In this case, a representation is given by a triple \( ((A, B), C) \) with the relation \( [A, B] C = I \). 

If \( | \mathrm{tr}(A) | = 2 \), then \( A \) is conjugate to one of the unipotent matrices
\[
\alpha^{\pm}(s) = \pm \begin{pmatrix}
1 & s \\
0 & 1
\end{pmatrix}.
\]
For any \( \beta = \begin{pmatrix}
a & b \\
c & d
\end{pmatrix} \in \mathrm{SL}(2, \mathbb{R}) \), it follows that
\begin{equation} \label{eqn10}
\mathrm{tr}([\alpha^{\pm}(s), \beta]) = 2 + s^2 c^2 \geq 2,
\end{equation}
see \cite[Page 29]{KPW1}.

If \( |\mathrm{tr}(A)| < 2 \), then \( A \sim R(\theta) \), and from \eqref{eqn5} we have
\begin{equation} \label{eqn8}
\mathrm{tr}([R(\theta), \beta]) = \mathrm{tr}(R(\theta) \beta R(\theta) \beta^{-1}) = 2 + (\mathrm{tr}(\beta^\top \beta) - 2) \sin^2\theta \geq 2,
\end{equation}
where \( \mathrm{tr}(\beta^\top \beta) = (a - d)^2 + (b + c)^2 + 2 \geq 2\).

If \( A \) is hyperbolic, then it is conjugate to 
\[
\alpha(\lambda) = \begin{pmatrix}
\lambda & 0 \\
0 & \lambda^{-1}
\end{pmatrix}, \quad 0 < |\lambda| < 1.
\]
In this case, one computes
\begin{equation} \label{eqn7}
\mathrm{tr}([\alpha(\lambda), \beta]) = 2 - bc(\lambda - \lambda^{-1})^2.
\end{equation}

Now we analyze the connected components of boundary elliptic-unipotent representations, i.e., when \( C \sim R(\theta) \) for some \( \theta \in (0, 2\pi) \). Since \( \mathrm{tr}(C) = \mathrm{tr}([A, B]) \in [-2, 2) \), both \( A \) and \( B \) must be hyperbolic by \eqref{eqn10} and \eqref{eqn8}. Up to conjugation, we may assume \( A = \alpha(\lambda) \) and \( B = \beta \). Then \( bc > 0 \), and we distinguish the subsets \( \{b > 0, c > 0\} \) and \( \{b < 0, c < 0\} \). Hence, the space of boundary elliptic-unipotent representations has eight connected components, determined by the triple
\[
N({\mathrm{sgn}}(\lambda), {\mathrm{sgn}}(\mathrm{tr}(B)), {\mathrm{sgn}}(B_{21})).
\]

Note that the signature satisfies
\[
\mathrm{sign}(\phi) = 2\mathrm{T}(\phi) + 2\boldsymbol{\rho}_\phi(C) = 2\mathrm{T}(\phi) + 2(1 - \tfrac{\theta}{\pi}).
\]
As \( \theta \to 2\pi \), we have \( T \to 0 \), and hence \( \mathrm{sign}(\phi) = -2 \). Similarly, as \( \theta \to 0 \), we again have \( T \to 0 \), and \( \mathrm{sign}(\phi) = 2 \).

Using the explicit expression
\begin{equation} \label{eqn6}
[\alpha(\lambda), \beta] = 
\begin{pmatrix}
ad - \lambda^2 bc &  ( \lambda^2-1) ba \\
(\frac{1}{\lambda^2} -1) cd & -\frac{cb}{\lambda^2} + da
\end{pmatrix},
\end{equation}
we deduce that
\[
\mathrm{sign}\big(N({\mathrm{sgn}}(\lambda), {\mathrm{sgn}}(\mathrm{tr}(B)), {\mathrm{sgn}}(B_{21}))\big) = 2 \cdot \mathrm{sgn}(|\lambda|-1){\mathrm{sgn}}(\mathrm{tr}(B)) {\mathrm{sgn}}(B_{21}).
\]

Furthermore, since \( ad = bc + 1 > 1 \), we have \( a, b, c, d \neq 0 \). By \eqref{eqn6}, it follows that \( [\alpha(\lambda), \beta] \neq -I \), and therefore each boundary elliptic-unipotent representation is actually boundary elliptic. Consequently, the space \( \mathrm{E}(\Sigma_{1,1}) \), and hence \( \mathrm{Ell}(\Sigma_{1,1}) \), has exactly eight connected components.
\begin{prop}\label{ell-prop3}
The number of connected components of \( \mathrm{E}(\Sigma_{1,1}) \) and \( \mathrm{Ell}(\Sigma_{1,1}) \) is eight.
\end{prop}

Moreover, for any representation $\phi=((A,B),C)\in \mathrm{Hom}(\pi_1(\Sigma_{1,1}),\mathrm{SL}(2,\mb{R}))$, its signature is given by Table \ref{sign-table}.

\begin{table}[htbp]
\centering
\renewcommand{\arraystretch}{1.2}
\caption{The signatrue for $\mr{Hom}(\pi_1(\Sigma_{1,1}),\mathrm{SL}(2,\mb{R}))$}
\begin{tabular}{| c | c | }
\hline
     $C(\neq -I)$   & $\mathrm{sign}(\phi)$\\
\hline
   $\mathrm{tr}(C)> 2$  & $0$\\
\hline
   $C\sim\begin{pmatrix}
  1& 0\\
  \mu&1 
\end{pmatrix}$  & $\mathrm{\mathrm{sgn}}(\mu)$\\
\hline
    $C\sim R(\theta)$, $\theta\in (0,\pi)$   & $2$\\
\hline
$C\sim R(\theta)$, $\theta\in (\pi,2\pi)$  & $-2$\\
\hline
 $\mathrm{tr}(C)\leq -2$  & $\pm 2$\\
\hline
\end{tabular}	
\label{sign-table}
\end{table}

\subsubsection{The case of $n= 2$}

In this section, we assume that \( g = 1 \) and \( n = 2 \). In this case, by \cite[Theorem 4]{KPW1}, the possible values of the signature for boundary elliptic (or boundary elliptic-unipotent) representations are given by the set \( \{-4, 0, 4\} \).

Let \( \phi \in \mathrm{Hom}(\pi_1(\Sigma), \mathrm{SL}(2, \mathbb{R})) \) be a boundary elliptic-unipotent representation. Denote by \( \phi_1 = \phi|_{\pi_1(\Sigma_{1,1})} \) and \( \phi_2 = \phi|_{\pi_1(\Sigma_{0,3})} \) the restrictions of \( \phi \) to the subsurfaces \( \Sigma_{1,1} \) and \( \Sigma_{0,3} \), respectively. In fact, these two subsurfaces $\Sigma_{1,1}$ and $\Sigma_{0,3}$ are obtained by cutting the surface $\Sigma=\Sigma_{1,2}$ along a simple closed curve $c_0$, that is, $\Sigma=\Sigma_{1,1}\cup_{c_0}\Sigma_{0,3}$.
 Then the signature satisfies
\[
  \mathrm{sign}(\phi) = \mathrm{sign}(\phi_1) + \mathrm{sign}(\phi_2).
\]
\begin{lemma}\label{lemma3}
$\mathrm{sign}^{-1}(0)\cap \mathrm{E}(\Sigma)$ is a connected component. 	
\end{lemma}
\begin{proof}
Suppose that \( \phi \in \mathrm{sign}^{-1}(0) \cap \mathrm{E}(\Sigma) \). Then
\[
  \mathrm{sign}(\phi_1) + \mathrm{sign}(\phi_2) = 0.
\]
By a small perturbation, we may assume that \( \mathrm{tr}([A, B]), \mathrm{tr}(A), \mathrm{tr}(B) \neq \pm 2 \).

If \( \mathrm{tr}([A, B]) > 2 \), then \(  [A, B] \) is hyperbolic. Suppose \( C_i=\phi(c_i) \sim R(\theta_i) \), $i=1,2$. Then
\[
  \mathrm{tr}(C_1 C_2) = 2 \cos(\theta_1 + \theta_2) - [(a - d)^2 + (b + c)^2] \sin \theta_1 \sin \theta_2 > 2,
\]
which implies \( \sin \theta_1 \sin \theta_2 < 0 \), where $C_i=P_iR(\theta_i)P_i^{-1}$, $P=P_{1}^{-1}P_2=\begin{pmatrix}
  a&b \\
  c&d 
\end{pmatrix}\in \mathrm{SL}(2,\mb{R})$. Without loss of generality, assume \( \theta_2 \in (\pi, 2\pi) \). Define \( \theta_2(t) = (1 - t)\theta_2 + t(2\pi - \theta_1) \in (\pi,2\pi)\), and deform \( P \) so that \( \mathrm{tr}(C_1 C_2) \) remains invariant. Thus, we may assume \( \theta_1 + \theta_2 = 2\pi \), and
\[
  \mathrm{tr}(C_1 C_2) = 2 + [(a - d)^2 + (b + c)^2] \sin^2 \theta_1 > 2.
\]
Now we deform \( P \) to decrease the quantity \( (a - d)^2 + (b + c)^2 \) toward zero. Then there exists a deformation \( \phi_2(t) \) such that \( \mathrm{tr}(\phi_2(t)(c_1 c_2)) \to 2 \), and \( \phi_2(1)(c_1 c_2) = I \).

If \( A \) is elliptic, write \( A = R(\theta) \), \( B = \beta \). By \eqref{eqn8},
\[
  \mathrm{tr}([R(\theta), \beta]) \geq 2,
\]
with equality if and only if \( B = R(\theta') \). We can deform \( \beta(t) \) such that
\begin{equation} \label{eqn9}
  \mathrm{tr}([R(\theta), \beta(t)]) = \mathrm{tr}(\phi_2(t)(c_1 c_2)).
\end{equation}
Hence we obtain a deformation \( \phi_1(t) \) of \( \phi_1 \) such that \eqref{eqn9} holds. Then \( \phi_1(t)(c_0) \) is conjugate to \( \phi_2(t)(c_0)^{-1} \). Up to a conjugation, we can glue the two paths, and obtain a deformation \( \phi(t) \) with \( \phi(1)(c_1 c_2) = I \).

If \( A \) is hyperbolic, then by \eqref{eqn7},
\[
  \mathrm{tr}([\alpha(\lambda), \beta]) = 2 - bc(\lambda - \lambda^{-1})^2 > 2,
\]
so \( bc < 0 \). We can deform \( \beta \) such that \( \mathrm{tr}([\alpha(\lambda), \beta]) = \mathrm{tr}(\phi_2(t)(c_1 c_2)) \), and ensure \( b = c = 0 \) when \( \mathrm{tr}([\alpha(\lambda), \beta]) = 2 \). Similar to the elliptic case, we glue the two deformations to get a representation satisfying \( \phi(1)(c_1 c_2) = I \).

If \( |\mathrm{tr}([A, B])| < 2 \), then \( [A, B] \) is elliptic, and we can deform \( \phi \) such that \( \phi(c_1), \phi(c_2), \phi(c_1 c_2) \in \mathrm{SO}(2) \setminus \{\pm I\} \) by Lemma \ref{lemma1}. Find a deformation \( \phi_2(t) \) such that \( \phi_2(1)(c_1 c_2) = I \). The angle of \( \phi(t)(c_1 c_2)^{-1} \) converges to \( \pi(1 - \tfrac{1}{2}\mathrm{sign}(\phi_2)) \). From \eqref{eqn7},
\[
  \mathrm{tr}([\alpha(\lambda), \beta]) = 2 - bc(\lambda - \lambda^{-1})^2 < 2,
\]
so \( bc > 0 \). We can deform \( \beta \) as above, and again glue the representations to get \( \phi(1)(c_1 c_2) = I \).

If $\mathrm{tr}([A,B]) < -2$, then $[A,B]$ is hyperbolic. By \eqref{eqn8}, this implies that both $A$ and $B$ are hyperbolic. Up to conjugation, we may assume $A = \alpha(\lambda)$ and $B = \beta$. Then, by \eqref{eqn6} and \eqref{eqn7}, we have
\begin{equation}
  [A,B] = \begin{pmatrix}
    ad - \lambda^2 bc &  ab(\lambda^2-1) \\
   cd(\tfrac{1}{\lambda^2}-1) & -\tfrac{cb}{\lambda^2} + da
  \end{pmatrix},
\end{equation}
and
\begin{equation*}
  \mathrm{tr}([A,B]) = 2 - bc(\lambda - \lambda^{-1})^2 < -2.
\end{equation*}

On the other hand, by \eqref{eqn5}, we have:
\[
(C_1 C_2)_{21} = \sin \theta_1 \cos \theta_2 + \left[(a' c' + b' d')\sin \theta_1 + (c'^2 + d'^2)\cos \theta_1\right] \sin \theta_2,
\]
and
\begin{equation}\label{eqn18}
\mathrm{tr}(C_1 C_2) = 2 \cos \theta_1 \cos \theta_2 - (\mathrm{tr}(P^\top P)) \sin \theta_1 \sin \theta_2 < -2,
\end{equation}
which implies $\sin \theta_1 \sin \theta_2 > 0$. Here, $P = \begin{pmatrix} a' & b' \\ c' & d' \end{pmatrix} \in \mathrm{SL}(2, \mathbb{R})$.

We now fix $P$ and $\theta_1$, and deform $\theta_2$ such that $|\mathrm{tr}(C_1 C_2)(1)| < 2$ while keeping ${\mathrm{sgn}}((C_1 C_2)_{21})$ constant. Since $\sin \theta_1 \sin \theta_2 > 0$, we define
\[
F(\theta_1, \theta_2) := \frac{(C_1 C_2)_{21}}{\sin \theta_1 \sin \theta_2} = \cot \theta_2 + (a' c' + b' d') + (c'^2 + d'^2) \cot \theta_1.
\]
If $F(\theta_1, \theta_2) > 0$, we take $\theta_2 \to 0$ for $\theta_2 \in (0, \pi)$, or $\theta_2 \to \pi$ for $\theta_2 \in (\pi, 2\pi)$. If $F(\theta_1, \theta_2) < 0$, we take $\theta_2 \to \pi$ for $\theta_2 \in (0, \pi)$, or $\theta_2 \to 2\pi$ for $\theta_2 \in (\pi, 2\pi)$. In all cases, $|F(\theta_1, \theta_2(t))|$ increases and ${\mathrm{sgn}}(F(\theta_1, \theta_2))$ remains invariant.

Moreover, $\mathrm{tr}(C_1 C_2)(t) \to \pm 2 \cos \theta_1$, so eventually $|\mathrm{tr}(C_1 C_2)(1)| < 2$. Hence, we obtain a deformation $\phi_2(t)$ of the representation over $\{c_0,c_1, c_2\}$.

Next, we adjust $b, c$ in $B$ so that
\[
\mathrm{tr}([A,B]) = 2 - bc(\lambda - \lambda^{-1})^2 = \mathrm{tr}(\phi_2(t)(c_0)).
\]
This yields a deformed representation $\phi_1(t)$ on $\pi_1(\Sigma_{1,1})$ such that
\[
\mathrm{tr}(\phi_1(t)(c_0)) = \mathrm{tr}(\phi_2(t)(c_0)).
\]
Since $bc > 0$ and $ad > 1$, it follows that $a, b, c, d \neq 0$, so ${\mathrm{sgn}}(cd)$ is invariant. Thus,
\[
({\mathrm{sgn}}([A,B]_{21}) \cdot {\mathrm{sgn}}((C_1 C_2)_{21}))(t)
\]
remains constant. At $t=0$, we have $[A,B] = (C_1 C_2)^{-1}$, hence
\[
{\mathrm{sgn}}(cd) \cdot {\mathrm{sgn}}((C_1 C_2)_{21}) \equiv -1.
\]
It follows that $\phi_1(t)(c)$ is conjugate to $\phi_2(t)(c)^{-1}$. After a conjugate deformation of $\phi_2(t)$, we can glue $\phi_1$ and $\phi_2$ to obtain a new deformed representation $\phi(t)$.
In particular, $[A, B]$ is deformed to elliptic, reducing the situation to the case $|\mathrm{tr}([A, B])| < 2$. 

In summary, for any $\phi \in \mathrm{sign}^{-1}(0) \cap \mathrm{E}(\Sigma)$, we can find a deformation $\phi(t)$ such that $\phi(1)(c_1 c_2) = I$. Since the set of such representations is connected, it follows that $\mathrm{sign}^{-1}(0) \cap \mathrm{E}(\Sigma)$ is connected.
\end{proof}

\begin{lemma}
	$\mathrm{sign}^{-1}(\pm 4)\cap \mathrm{E}(\Sigma)$ has $4$ connected component.
\end{lemma}
\begin{proof}
Suppose \( \phi \in \mathrm{sign}^{-1}(4) \cap \mathrm{E}(\Sigma) \). Then
\[
  \mathrm{sign}(\phi_1) = \mathrm{sign}(\phi_2) = 2.
\]
By a slight perturbation, we may assume that \( | \mathrm{tr}([A, B]) | \neq 2 \). If \( \mathrm{tr}([A, B]) > 2 \), then by the proof of Lemma~\ref{lemma3}, the commutator \( [A, B] \) can be deformed to the identity, which implies \( \mathrm{sign}(\phi_1) = 0 \), a contradiction. See also \cite[Proposition 9.7]{KPW1} for further discussion.

Thus, we must have \( \mathrm{tr}([A, B]) < 2 \). From the proof of Lemma~\ref{lemma3}, the representation \( \phi \) can be deformed so that \( | \mathrm{tr}([A, B]) | < 2 \). This implies that \( \phi_1 \in \mathrm{sign}^{-1}(2) \cap \mathrm{E}(\Sigma_{1,1}) \) and \( \phi_2 \in \mathrm{sign}^{-1}(2) \cap \mathrm{E}(\Sigma_{0,3}) \). By Proposition~\ref{ell-prop2} and Proposition~\ref{ell-prop3}, we have
\[
  \#\{ \mathrm{sign}^{-1}(2) \cap \mathrm{E}(\Sigma_{1,1}) \} = 4, \quad \#\{ \mathrm{sign}^{-1}(2) \cap \mathrm{E}(\Sigma_{0,3})\} = 1.
\]
Hence, \( \mathrm{sign}^{-1}(4) \cap \mathrm{E}(\Sigma) \) has exactly 4 connected components.

By symmetry, the same argument applies to \( \mathrm{sign}^{-1}(-4) \cap \mathrm{E}(\Sigma) \), which therefore also has 4 connected components.
\end{proof}

Next, we consider the boundary elliptic representations \( \mathrm{Ell}(\Sigma_{1,2}) \). Assume each boundary element satisfies \( C_i \sim R(\theta_i) \), where \( \theta_i \in (0,\pi) \cup (\pi,2\pi) \) for \( i = 1,2 \).

Following the same reasoning as in earlier proofs, the set \( \mathrm{sign}^{-1}(0) \cap \mathrm{Ell}(\Sigma) \cap \sigma^{-1}(a) \) is connected when \( a = (1,-1) \) or \( a = (-1,1) \), as these configurations allow a deformation to a representation with \( [A, B] = I \). However, for \( a = (1,1) \) or \( a = (-1,-1) \), such a deformation is not possible—each case contributes four connected components.

For instance, consider \( a = (-1,-1) \). If \( [A, B] \) could be deformed to the identity, then there would exist smooth functions \( \theta_1(t), \theta_2(t) \in (0,\pi) \) for \( t \in [0,1] \) satisfying 
\[
\theta_1(1) + \theta_2(1) = 2\pi,
\]
which is impossible since their sum must lie in $(0,2\pi)$. Therefore,
\[
\#\{\mathrm{sign}^{-1}(0) \cap \mathrm{Ell}(\Sigma) \} = 2 + 4 + 4 = 10.
\]

Next, consider the connected components of \( \mathrm{sign}^{-1}(4) \cap \mathrm{Ell}(\Sigma) \cap \sigma^{-1}(a) \). In this case,
\[
\mathrm{sign}(\phi_1) = \mathrm{sign}(\phi_2) = 2.
\]
From the proof of Lemma~\ref{lemma3}, we may assume that \( [A, B] \) is elliptic. Let \( \theta_{[A,B]} \) denote the rotation angle associated with \( [A,B] \). Then \( \mathrm{sign}(\phi_1) = 2 \) implies
\(
\theta_{[A,B]} \in (\pi,2\pi).
\)
For the representation \( \phi_2 \), we assume all boundary images lie in \( \mathrm{SO}(2) \setminus \{\pm I\} \). The condition \( \mathrm{sign}(\phi_2) = 2 \) yields:
\[
2 = 2(3 - \tfrac{\theta_1 + \theta_2 + \theta_{[A,B]}}{\pi} ),
\]
which implies
\(
\theta_1 + \theta_2 \in (0, \pi).
\)
Hence, if \( \sigma(\phi) \neq (-1,-1) \), this condition cannot be satisfied, and therefore
\[
\mathrm{sign}^{-1}(4) \cap \mathrm{Ell}(\Sigma) \cap \sigma^{-1}(a) = \emptyset \quad \text{for } a \neq (-1,-1).
\]
Thus,
\(
\#\{ \mathrm{sign}^{-1}(4) \cap \mathrm{Ell}(\Sigma) \} = 4.
\)
Similarly,
\(
\#\{ \mathrm{sign}^{-1}(-4) \cap \mathrm{Ell}(\Sigma) \} = 4.
\)
\begin{prop}
The space \( \mathrm{E}(\Sigma_{1,2}) \) has $9$ connected components, and \( \mathrm{Ell}(\Sigma_{1,2}) \) has $18$ connected components
\end{prop}

\subsubsection{The case of $n\geq 3$} In this subsection, $\Sigma=\Sigma_{1,n}$, $n\geq 3$.

\begin{lemma}\label{lemma5}
For any $m \in [-(n-2), n-2] \cap \{n - 2\mathbb{Z}\}$, the set $\mathrm{sign}^{-1}(2m) \cap \mathrm{E}(\Sigma)$ is a connected component.
\end{lemma}

\begin{proof}
Let $\phi = ((A,B), C_1, \ldots, C_n) \in \mathrm{sign}^{-1}(2m) \cap \mathrm{E}(\Sigma)$. By Lemma~\ref{lemma6}, we may assume that $C_i = R(\theta_i) \in \mathrm{SO}(2) \setminus \{\pm I\}$ for all $i$. After a slight perturbation, we can also assume that $\frac{1}{\pi} \sum_{i=1}^n \theta_i \notin \mathbb{Z}$.

Note that $C_0 := [A,B] =( C_1 \cdots C_n )^{-1}\in \mathrm{SO}(2) \setminus \{\pm I\}$. Consider the representation $\phi_1 = ((A,B), C_0^{-1})$. Then
\(
  \mathrm{sign}(\phi_1) = \pm 2,
\)
which implies that
\(
  \mathrm{sign}(\phi_2) = 2m \mp 2,
\)
where $\phi_2 = (C_0, C_1, \ldots, C_n)$.

Now we deform $\phi_2$ to a path $\phi_2(t)$ such that $C_0(1) = I$. There are two possible deformations of $C_0(t)$ to the identity; we choose the one where the angle of $C_0(t)$ converges to $\pi \left(1 + \frac{1}{2} \mathrm{sign}(\phi_1)\right)$. This choice ensures that
\begin{align*}
  \mathrm{sign}(\phi_2(t))&=2(n+1-\tfrac{\theta_0+\theta_1+\cdots+\theta_{n}}{\pi})\\
  &=2m\mp 2\in [-2(n-1),2(n-1)]\cap \{2(n-1)-4\mb{Z}\},
\end{align*}
for all $t \in [0,1)$, where $C_i=R(\theta_i)$, $0\leq i\leq n$.

Assume $A = U \alpha(\lambda) U^{-1}$ and $B = U \beta U^{-1}$. By \eqref{eqn7}, we have
\[
  \mathrm{tr}([A,B]) = \mathrm{tr}([\alpha(\lambda), \beta]) = 2 - bc(\lambda - \lambda^{-1})^2 < 2,
\]
so $bc > 0$. We may deform $\beta$ such that $\mathrm{tr}([\alpha(\lambda), \beta]) \equiv \mathrm{tr}(\phi_2(t)(c_1 \cdots c_n))$, and ensure that $b = c = 0$ when $\mathrm{tr}([\alpha(\lambda), \beta]) = 2$. This deformation makes the angle of $[\alpha(\lambda), \beta]$ converge to $\pi \left(1 + \frac{1}{2} \mathrm{sign}(\phi_1)\right)$, so that $[A,B](t) \sim C_0(t)$.

Up to conjugation, we glue $\phi_1(t)$ and $\phi_2(t)$ to obtain a representation $\phi(t)$ such that $\phi(1)$ is boundary $\mathrm{SO}(2) \setminus \{\pm I\}$ and satisfies $\phi(1)([a,b]) = I$.

Let $W$ denote the set of boundary $\mathrm{SO}(2) \setminus \{\pm I\}$ representations with $[A,B] = I$ and signature $2m$. Then
\[
  W = \{ (\psi, \phi) \mid \psi \in \mathrm{Hom}(\pi_1(\Sigma_{1,0}), \mathrm{SL}(2,\mathbb{R})),\ \phi \in \mathrm{Ell}(\Sigma_{0,n}) \cap \mathrm{sign}^{-1}(2m) \}.
\]
Since both $\mathrm{Hom}(\pi_1(\Sigma_{1,0}), \mathrm{SL}(2,\mathbb{R}))$ and $\mathrm{Ell}(\Sigma_{0,n}) \cap \mathrm{sign}^{-1}(2m)$ are connected, it follows that $W$ is connected. The proof is complete.
\end{proof}

\begin{lemma}\label{lemma8}
The set $\mathrm{sign}^{-1}(\pm 2n) \cap \mathrm{E}(\Sigma)$ has four connected components.
\end{lemma}
\begin{proof}
Suppose $\phi \in \mathrm{sign}^{-1}(2n) \cap \mathrm{E}(\Sigma)$. Then
\[
  \mathrm{sign}(\phi_1) = 2, \quad \mathrm{sign}(\phi_2) = 2(n-1),
\]
where $\phi_1$ is the restriction of $\phi$ to the genus-one subsurface, and $\phi_2$ is the restriction to the remaining punctured sphere.

Since $\mathrm{tr}([A,B]) < 2$, it follows from formulas \eqref{eqn8} and \eqref{eqn10} that both $A$ and $B$ are hyperbolic. Therefore, the possible sign data give rise to at least four connected components, which are permuted by  $-I$.
From the argument in the proof of Lemma~\ref{lemma3}, we know that the representation can be deformed so that $|\mathrm{tr}([A, B])| < 2$. Consequently, we have
\[
  \phi_1 \in \mathrm{sign}^{-1}(2) \cap \mathrm{E}(\Sigma_{1,1}), \quad \phi_2 \in \mathrm{sign}^{-1}(2n-2) \cap \mathrm{E}(\Sigma_{0,n+1}).
\]

By the classification of connected components in the cases $\Sigma_{1,1}$ and $\Sigma_{0,n+1}$, we obtain
\[
  \#\{\mathrm{sign}^{-1}(2) \cap \mathrm{E}(\Sigma_{1,1})\} = 4, \quad \#\{\mathrm{sign}^{-1}(2n-2) \cap \mathrm{E}(\Sigma_{0,n+1})\} = 1.
\]
It follows that $\mathrm{sign}^{-1}(2n) \cap \mathrm{E}(\Sigma)$ has four connected components. The same argument applies to $\mathrm{sign}^{-1}(-2n) \cap \mathrm{E}(\Sigma)$, which also has four connected components.
\end{proof}
Next, we consider the boundary elliptic representations \( \mathrm{Ell}(\Sigma_{1,n}) \). We assume that each boundary element satisfies \( C_i \sim R(\theta_i) \), with \( \theta_i \in (0,\pi) \cup (\pi,2\pi) \), for \( 1 \leq i \leq n \).

For the maximal signature, the set \( \mathrm{sign}^{-1}(2n) \cap \mathrm{Ell}(\Sigma) \cap \sigma^{-1}(a) \) is connected only when \( r_a := \#\{a_i = -1\} = n \), and is empty otherwise. Thus,
\[
\#\{\mathrm{sign}^{-1}(2n) \cap \mathrm{Ell}(\Sigma) \} = 4.
\]
Similarly,
\[
\#\{ \mathrm{sign}^{-1}(-2n) \cap \mathrm{Ell}(\Sigma) \} = 4.
\]

Next, we analyze the number of connected components of the sets \( \mathrm{sign}^{-1}(2m) \cap \mathrm{Ell}(\Sigma) \cap \sigma^{-1}(a) \) for \( m \neq \pm n \). As in the previous proof, we have:
\begin{align*}
\mathrm{sign}(\phi_2(t))& = 2(n+1 - \tfrac{1}{\pi}(\theta_0+\theta_1 + \cdots + \theta_{n})) \\
&= 2m \mp 2 \in [-2(n-1), 2(n-1)] \cap \{ 2(n-1) - 4\mathbb{Z} \}
\end{align*}
for any \( t \in [0,1) \). Compared to the elliptic-unipotent case, the angles \( \theta_i \) cannot cross \( \pi \) during deformation. Since $\phi_2$ is a representation, \( \sum_{i=0}^{n} \theta_i =2k\pi\) for some \( k \in [1,n] \cap \mathbb{Z} \).

If \( \theta_{0} \to 0 \), then the tuple \( (\theta_0,\theta_1,\ldots,\theta_n) \) contains \( r_a + 1 \) angles in \( (0,\pi) \), so
\[
2k = \tfrac{1}{\pi} \sum_{i=0}^{n} \theta_i \in (n - r_a, 2n - r_a + 1).
\]
If \( 2k = 2n - r_a \), then there is no configuration of \( (\theta_1,\ldots,\theta_n) \) with exactly \( r_a \) angles in \( (0,\pi) \), hence \( [A,B] \) cannot be deformed to the identity. In this case, \( r_a \) must be even, and
\[
\mathrm{sign}(\phi) = 2m = 2(n + 1 - 2k) - 2 = 2(r_a - n) \in [-2n + 4, 2n - 4].
\]
This implies \( r_a \in [2,n] \cap 2\mathbb{Z} \) and \( \mathrm{sign}(\phi) \leq 0 \).

If \( \theta_{0} \to 2\pi \), then \( (\theta_0,\theta_1,\ldots,\theta_n) \) has \( r_a \) entries in \( (0,\pi) \), hence
\[
2k = \tfrac{1}{\pi} \sum_{i=0}^{n} \theta_i \in (n + 1 - r_a, 2n - r_a + 2),
\]
and \( \sum_{i=1}^n \theta_i \to 2(k-1)\pi \), so
\[
2(k-1) \in (n - r_a-1, 2n - r_a).
\]
If \( 2(k-1) = n-r_a \), then again the deformation to identity is not possible. In this case, \( r_a \in \{n-2\mb{Z}\}\) ,
\[
\mathrm{sign}(\phi) = 2m = 2(n + 1 - 2k) + 2 = 2r_a \in [-2n + 4, 2n - 4].
\]
This implies that $n-r_a\in[2,n]\cap 2\mb{Z}$.
Thus, in the following situations, \( \theta_{0} \) cannot be deformed to \( 0 \) or \( 2\pi \),
\begin{equation}\label{eqn19}
\mathrm{sign}(\phi) = 
\begin{cases}
2(r_a - n) & \text{if } r_a \in [2, n] \cap 2\mathbb{Z} \text{ and } \mathrm{sign}(\phi|_{\pi_1(\Sigma_{1,1})}) = -2, \\
2r_a & \text{if } n-r_a\in[2,n]\cap 2\mb{Z} \text{ and } \mathrm{sign}(\phi|_{\pi_1(\Sigma_{1,1})}) = 2.
\end{cases}
\end{equation}
The number of connected components is given as follows
\begin{align*}
\#\{ \mathrm{Ell}(\Sigma_{1,n}) &= 4 + 4 + 4 \sum_{r_a \in [2,n] \cap 2\mathbb{Z}} \binom{n}{r_a} + 4 \sum_{r_a \in [0, n-2] \cap \{n-2\mb{Z}\}} \binom{n}{r_a} + (n-1) 2^{n-1} \\
&= 2^{n+2} + (n - 1) 2^{n-1} \\
&= 2^{n - 1}(n + 7).
\end{align*}

\begin{prop}
$\mathrm{E}(\Sigma_{1,n})$ has $n+7$ connected components, and $\mathrm{Ell}(\Sigma_{1,n})$ has $2^{n - 1}(n + 7)$ components.
\end{prop}

\subsection{The case of general genus} In this section, we consider the case where the genus $g \geq 2$. Before proving the main results, we first establish a useful deformation lemma that simplifies subsequent arguments.
\begin{lemma}\label{lemma6}
Let $\phi \in \mathrm{Hom}(\pi_1(\Sigma_{g,n}), \mathrm{SL}(2,\mathbb{R}))$ be a boundary elliptic (or boundary elliptic-unipotent) representation. Then $\phi$ can be deformed to a boundary $\mathrm{SO}(2) \setminus \{\pm I\}$ representation in $\mathrm{Ell}(\Sigma_{g,n})$.
\end{lemma}

\begin{proof}
Let $\phi \in \mr{Ell}(\Sigma_{g,n})$. Decompose the surface as $\Sigma_{g,n} = \Sigma_{g,1} \cup_{c_0} \Sigma_{0,n+1}$, and denote the restrictions by $\phi_1 := \phi|_{\pi_1(\Sigma_{g,1})}$ and $\phi_2 := \phi|_{\pi_1(\Sigma_{0,n+1})}$. Let $a_1, b_1, \dots, a_g, b_g, c_1, \dots, c_n$ be generators of $\pi_1(\Sigma_{g,n})$, and set $\phi(a_i) = A_i$, $\phi(b_i) = B_i$, and $\phi(c_j) = C_j$ for $1 \leq i \leq g$ and $1 \leq j \leq n$.

After a small perturbation, we may assume that
\[
  |\mathrm{tr}(C_{i_1} \cdots C_{i_r})| \neq 2
\]
for all distinct indices $1 \leq i_1, \dots, i_r \leq n$. 

Let $C_0 := (C_1 \cdots C_n)^{-1}$. If $C_0$ is elliptic, then by the proof of Proposition~\ref{ell-prop2}, the representation $\phi_2$ can be deformed to a boundary $\mathrm{SO}(2) \setminus \{\pm I\}$ representation, while preserving the conjugacy class of $\phi_2(c_0)$. Applying a conjugate deformation to $\phi_1$ such that $\phi_1(t)(c_0) = \phi_2(t)(c_0)^{-1}$, we can glue $\phi_1(t)$ and $\phi_2(t)$ to obtain a deformation $\phi(t)$ of $\phi$ satisfying $\phi(1)(c_i) = \phi_2(1)(c_i) \in \mathrm{SO}(2) \setminus \{\pm I\}$ for all $i$. Thus, $\phi(1)$ is a boundary $\mathrm{SO}(2) \setminus \{\pm I\}$ representation.

If $C_0$ is hyperbolic, then since the space 
$\mathrm{Hom}(\pi_1(\Sigma_{g,1}),\mathrm{SL}(2,\mathbb{R}))$ 
is connected, there exists a path $\phi_1(t)$ starting from $\phi_1$ such that 
$\phi_1(1)(c_0) \in \mathrm{SO}(2)$. Moreover, the path $\phi_1(t)$ can be chosen so that there exists a unique
$t \in [0,1]$ at which $|\mathrm{tr}(\phi_1(t)(c_0))| = 2$, and 
\(\phi_1(1)(c_0) = I \) if \( \mathrm{tr}(C_0) > 2,\)
 \(\phi_1(1)(c_0) \in \mathrm{SO}(2)\setminus \{\pm I\}\) if \( \mathrm{tr}(C_0) < -2.\)

Set $P(t) := \phi_1(t)(c_0)$. We now construct deformations $C_1(t), \dots, C_n(t)$ such that
\begin{equation}\label{eqn-deformC}
  C_1(t) \cdots C_n(t) = P(t)^{-1}, \quad \text{and} \quad C_i(1) \in \mathrm{SO}(2) \setminus \{\pm I\}.
\end{equation}

For $n = 2$, the existence of such $C_i(t)$ with the conjugacy class of $C_1$ or $C_2$ fixed follows from the proof of Lemma~\ref{lemma3}.  
Assume that \eqref{eqn-deformC} holds for $n-1$, and that the conjugacy class of any one of the $C_i$ can be fixed during deformation.  
We now prove the case for $n$, under the condition that the conjugacy class of an arbitrarily chosen $C_i$ is fixed.  
Without loss of generality, we fix the conjugacy class of $C_1$ and solve \eqref{eqn-deformC} accordingly.

If $C_2 C_3$ is hyperbolic, then by the proof of Proposition~\ref{ell-prop2} ($C_2,C_3$ play the role of $C_{n-1},C_n$ there), we can obtain the deformtion $\widetilde{C}_i(t)$ of $C_i$ for $i=2,3$ such that $\widetilde{C}_2(1)$ commutes with $C_1$ or very closed to $\pm I$, and $\widetilde{C}_2(t)\widetilde{C}_3(t)$ is constant. Up to a small perturbation on $\wt{C}_2(1)$, we may then regard $C_2' := C_1 C_2$ as elliptic (when $\widetilde{C}_2(1)$ commutes with $C_1$, it is obvious, and when it is very close to $\pm I$,  $|\mathrm{tr}(C_1\widetilde{C}_2(1))|<2$). Thus, equation \eqref{eqn-deformC} reduces to the $(n-1)$-term case, and we can find deformations $C_2'(t), C_3(t), \dots, C_n(t)$ such that $C_2'(1), C_3(1), \dots, C_n(1) \in \mathrm{SO}(2) \setminus \{\pm I\}$ and the conjugacy class of $C_2'$ is preserved.


By Lemma \ref{lemma1} (using $C_1^{-1}C_2'=C_2$ we obtain $C_1''(t)$ such that $C_1''(t)^{-1}C_2'=C_2''(t)$), we can fix \( C_2' \) and the conjugacy class of \( C_1 \), and obtain deformation \( C_1''(t) \) (resp. \( C_2''(t) \)) of $C_1$ (resp. $C_2$)  such that
\( C_2' = C_1''(t) C_2''(t),\) where \( C_2' \), \( C_1''(1) \), and \( C_2''(1) \) commute. Since the conjugacy class of $C_2'(t)$ is invariant, there exists $U(t)\in \mr{SL}(2,\mb{R})$ with $U(0)=I$ such that  \( C_2'(t) = U(t) C_2' U(t)^{-1} \), then define \[  C_1(t) := U(t) C_1''(t) U(t)^{-1}, \quad C_2(t) := U(t) C_2''(t) U(t)^{-1}.\] With this construction, the conjugacy class of \( C_1(t) \) is preserved, and we have \( C_1(t) C_2(t) = C_2'(t).\) Moreover, since \( C_2'(1) \in \mathrm{SO}(2) \setminus \{\pm I\} \) and the matrices \( C_1(1), C_2(1), C_2'(1) \) are commute, it follows that \( C_1(1), C_2(1) \in \mathrm{SO}(2) \setminus \{\pm I\}.\) In one word,  $C_1(t),C_2(t),C_3(t),\cdots,C_n(t)$ solve \eqref{eqn-deformC} and the conjucay class of $C_1(t)$ is preserved.

If $C_2C_3$ is elliptic, we treat it as a single unit, thereby reducing the situation to the previous case.  
More precisely, by Lemma~\ref{lemma1} (again using $C_2^{-1}C_2C_3=C_3$), we may keep $C_2C_3$ fixed while deforming $C_2$ and $C_3$ so that, in the end, they commute. Hence, without loss of generality, we may assume that $C_2$ and $C_3$ commute. Suppose $C_i \sim R(\theta_i)$ for $i=2,3$. Then we can deform $\theta_2$ and $\theta_3$ in such a way that $\theta_2+\theta_3$ remains invariant, while $R(\theta_2)$ approaches $\pm I$. Consequently, the product $C_1C_2$ can be deformed to an elliptic element $C_2'$. The remainder of the proof then follows exactly as in the previous case.
 Now the proof is complete.
\end{proof}

\begin{lemma}\label{lemma12}
Let \( \phi \in \mathrm{E}(\Sigma_{g,1}) \) be a boundary elliptic (or elliptic-unipotent) representation with \( \mathrm{sign}(\phi) < 2|\chi(\Sigma_{g,1})| \) (resp. \( \mathrm{sign}(\phi) > -2|\chi(\Sigma_{g,1})| \)). Then there exists a deformation \( \phi(t) \) such that \( \phi(t) \in \mathrm{E}(\Sigma_{g,1}) \) for all \( t \neq t_0 \), with \( t_0 \in (0,1) \), and such that \( \phi(t_0)(c_0) = I \), where \( c_0 = \partial \Sigma_{g,1} \). In particular, \( \theta_0(t) \to 2\pi \) (resp. \( \theta_0(t) \to 0 \)) as \( t \to t_0^- \), where \( \phi(t)(c_0) \sim R(\theta_0(t)) \).
\end{lemma}

\begin{proof}
Note that reversing the orientation of the surface reverses the sign of the signature, and simultaneously changes each boundary angle \( \theta \) of a boundary elliptic representation to \( 2\pi - \theta \). Hence, the two cases \( \mathrm{sign}(\phi) < 2|\chi(\Sigma_{g,1})| \) and \( \mathrm{sign}(\phi) > -2|\chi(\Sigma_{g,1})| \) are equivalent. Thus, it suffices to prove the lemma in the case \( \mathrm{sign}(\phi) < 2|\chi(\Sigma_{g,1})| \).

For any \( \phi \in \mathrm{Ell}(\Sigma_{g,1}) \), up to conjugation, write
\begin{equation}\label{eqn17}
\phi = ((A_1,B_1),\dots,(A_g,B_g), C_0),
\end{equation}
with \( C_0 = \phi(c_0) = R(\theta_0) \), \( \theta_0 \in (0,2\pi) \). Let \( C_1 = [A_1,B_1] \), and \( C_2 = [A_2,B_2]\cdots[A_g,B_g] \), so that \( C_1 C_2 C_0 = I \).

\begin{figure}[ht]
\centering
\includegraphics[width=0.7\textwidth]{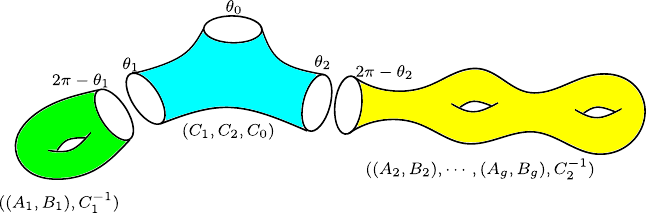}
\caption{The decomposition of $\phi=((A_1,B_1),\dots,(A_g,B_g), C_0)$.}
\end{figure}

For \( g=1 \), if \( \mathrm{sign}(\phi) < 2|\chi(\Sigma_{1,1})| = 2 \), then \( \mathrm{sign}(\phi) = -2 \), so \( \theta_0 \in (\pi,2\pi) \). From Section~\ref{secn=1}, such a representation \( \phi(t) \) exists in which \( \theta_0(t) \) passes through \( 2\pi \), with \( \phi(t)(c_0) = R(\theta_0(t)) \).

Assume the lemma holds for genus \( \leq g-1 \). For genus \( g \), assume (after a small perturbation) that both \( C_1 \) and \( C_2 \) are either elliptic or hyperbolic.

{\bf Case 1: }If both \( C_1 \) and \( C_2 \) are hyperbolic, we may conjugate so that \( C_2 = \mathrm{diag}(\lambda, 1/\lambda) \) and \( C_0 = P R(\theta_0) P^{-1} \) for some \( P \in \mathrm{SL}(2,\mathbb{R}) \). Then, by~\eqref{eqn16},
\[
\mathrm{tr}(C_2 C_0) = R \cos(\theta_0 - \alpha), \quad (C_2 C_0)_{21} = \lambda^{-1}(c^2 + d^2)\sin \theta_0.
\]
For \( \theta_0 = 0, \pi, 2\pi \), we have \( C_0 = \pm I \), and so \( |\mathrm{tr}(C_2 C_0)| = |\mathrm{tr}(C_2)| > 2 \). Define
\[
S := \{ \theta \in [0,2\pi] : |R \cos(\theta - \alpha)| > 2 \},
\]
which has two connected components, one containing \( 0, 2\pi \), and the other containing \( \pi \). The complement \( \overline{S}^c \) also has two components \( I_1, I_2 \), on which \( C_2 C_0 \) is elliptic.

Now we can always find a deformation path \( \theta_0(t) \) starting from the initial angle \( \theta_0 \), such that \( |\lambda^{-1} \sin \theta_0(t)| \) increases monotonically along the path. This ensures that \( \theta_0(t) \) enters one of the elliptic intervals \( I_1 \) or \( I_2 \), where the composition \( C_2 C_0(t) \) becomes elliptic, and the sign of $(C_2C_0(t))_{21}$ remains unchanged.

We now aim to find deformations \( A_1(t), B_1(t) \) of \( A_1, B_1 \) such that
\[
C_1(t) := [A_1(t), B_1(t)] \sim (C_2 C_0(t))^{-1}, \quad \text{for all } t \in [0,1].
\]

If \( \mathrm{tr}(C_1) < -2 \), then by equations~\eqref{eqn7} and~\eqref{eqn6}, we may assume that \( A = P_0 \alpha(\lambda) P_0^{-1} \), \( B = P_0 \beta P_0^{-1} \). According to the construction in Section~\ref{secn=1}, we can fix \( A \) and deform \( \beta \) such that \( \mathrm{tr}([A,B]) \) increases into the range \( (-2,2) \). Along this deformation, the sign of the $(2,1)$-entry of the commutator remains constant, i.e.,
\(
\mathrm{sgn}((C_1(t))_{21}) = \mathrm{sgn}((C_1(0))_{21}).
\)
Therefore, we achieve:
\begin{equation}\label{eqn20}
  \mathrm{tr}(C_1(t)) = \mathrm{tr}(C_2 C_0(t)),
\end{equation}
and
\begin{align}\label{eqn21}
\begin{split}
  &\quad \mathrm{sgn}((C_1(t))_{21}) \cdot \mathrm{sgn}((C_2 C_0(t))_{21}) \\
&  = \mathrm{sgn}((C_1(0))_{21}) \cdot \mathrm{sgn}((C_2 C_0(0))_{21}) = -1.
\end{split}
\end{align}
Hence, \( C_1(t) \sim (C_2 C_0(t))^{-1} \) for all \( t \in [0,1] \).

If \( \mathrm{tr}(C_1) > 2 \), then we can similarly choose deformations \( A_1(t), B_1(t) \) such that \( \mathrm{tr}(C_1(t)) \) decreases into the interval \( (-2,2) \), while preserving the sign of \( (C_1(t))_{21} \). 
In fact, by the path-lifting property \cite[Corollary 7.8]{Gold1}, we can construct deformations \( A_1(t), B_1(t) \) such that \( \mathrm{sgn}((C_1(t))_{21}) \) remains constant throughout the deformation, where $C(t)$ is hyperbolic for $[0,1)$, and $C(1)$ is parabolic. Moreover, we can arrange that
\[
C_1(1) = \begin{pmatrix}
  1 & 0 \\
  \mathrm{sgn}((C_1)_{21}) & 1
\end{pmatrix}.
\]
Hence, as \( C_1(t) \) deforms into an elliptic element, the sign of \( (C_1(t))_{21} \) remains unchanged by continuity. This sign preservation is crucial for verifying condition~\eqref{eqn21}, and thus for ensuring that \( C_1(t) \sim (C_2 C_0(t))^{-1} \) holds throughout the deformation.

Finally, after applying a suitable conjugation to $((A_1(t),B_1(t)),C_1^{-1}(t))$, the product \( C_2 C_0(t) \) becomes elliptic and is the inverse of \( C_1(t) \), so \( C_1 \) is reduced to the elliptic case. Therefore, the situation reduces to {\bf Case 3}.

{\bf Case 2:} If \( C_2 \) is elliptic, we reduce to {\bf Case 3} by considering the boundary elliptic representation \( ((A_1,B_1), C_2, C_0) \), and applying Lemma~\ref{lemma6}, fixing the conjugacy class of \( C_2 \).

\textbf{Case 3.} 
Assume now that \(C_{1}\) is elliptic.  
Then the restriction  
\[
\bigl((A_{2},B_{2}),\dots,(A_{g},B_{g}),\,C_{0},C_{1}\bigr)\;\in \mr{E}(\Sigma_{g-1,2})
\]
is a boundary elliptic representation of the twice-punctured surface \(\Sigma_{g-1,2}\).
By Lemma~\ref{lemma6}, we may deform it so that every boundary image lies in 
\(\mathrm{SO}(2)\setminus\{\pm I\}\).  
Consequently, we may—and do—assume  
\[
C_{i}=R(\theta_{i})\in \mathrm{SO}(2)\setminus\{\pm I\}, 
\qquad\theta_{i}\in(0,\pi)\cup(\pi,2\pi),\quad i=0,1,2,
\]
and therefore 
\[
\theta_{0}+\theta_{1}+\theta_{2}=2\pi\quad\text{or}\quad4\pi.
\]

For \(\theta_{0}+\theta_{1}+\theta_{2}=4\pi\).
Suppose first that one of \(\theta_{1},\theta_{2}\) lies in \((0,\pi)\).
Fix \(\theta_{2}\) and deform $\theta_{0}(t)\nearrow2\pi$ and $\theta_{1}(t)\searrow 2\pi-\theta_{2}$,
so that  
\(\theta_{1}(t)\in(2\pi-\theta_{2},\,\theta_{1})\) for every \(t\).
Hence \(\theta_{1}(t)\neq\pi\) along the path, and we reach \(\theta_{0}=2\pi\). By a slight perturbation, we can find a deformation $\phi(t)$ of $\phi$ such that its angle passes through $2\pi$.

Now consider the subtler situation where \(\theta_{1},\theta_{2}\in(\pi,2\pi)\).
Set  
\[
\phi_2=\bigl((A_{2},B_{2}),\dots,(A_{g},B_{g}),\,C_{2}^{-1}\bigr)
      \in \mr{E}\left(\Sigma_{g-1,1}\right).
\]
If \(\mathrm{sign}(\phi_2)<2\lvert\chi(\Sigma_{g-1,1})\rvert\), then by the induction
hypothesis the angle \(2\pi-\theta_{2}\) can be deformed through \(2\pi\) so that
\(\theta_{2}\) passes through \(0\).  
Using this, choose a small \(\varepsilon>0\) and set
\[
\begin{aligned}
\theta_{0}(t)&=\theta_{0}(1-t)+2\pi t,\\
\theta_{1}(t)&=\theta_{1}(1-t)+(\pi+\varepsilon)t,\\
\theta_{2}(t)&=\theta_{2}(1-t)+(\pi-\varepsilon)t,
\end{aligned}
\qquad t\in[0,1].
\]
Throughout the deformation we have \(\theta_{1}(t)\in(\pi,2\pi)\);
moreover \(\theta_{0}(1)=2\pi\), so the required deformation is achieved while
staying in \(\mathrm{SO}(2)\setminus\{\pm I\}\). By a slight perturbation, we obtain the desired deformation.

{\bf Claim.} For any $\phi \in \mathrm{Ell}(\Sigma_{g,1})$ with maximal signature (i.e., $\mathrm{sign}(\phi) = 2|\chi(\Sigma_{g,1})|$), we can deform $\phi$ to a boundary hyperbolic representation with negative trace.

\begin{proof}[Proof of the Claim]
We first consider the base case $g = 1$. For any $\phi \in \mathrm{Ell}(\Sigma_{1,1})$ with $\mathrm{sign}(\phi) = 2$, it follows from equation~\eqref{eqn6} that $\phi$ can be deformed to a boundary hyperbolic representation with negative trace.

Assume by induction that the claim holds for all genera up to $g-1$. We now prove the case for genus $g$. Suppose $\phi$ is given as in equation~\eqref{eqn17}, and define $C_1, C_2$ accordingly, as in the preceding construction. If both $C_1$ and $C_2$ are elliptic, then by Lemma~\ref{lemma6}, we may assume without loss of generality that $C_0, C_1, C_2 \in \mathrm{SO}(2) \setminus \{\pm I\}$.

Since $\phi$ has maximal signature, the restriction of $\phi$ to each subsurface also attains maximal signature. This implies that the corresponding rotation angles satisfy $\theta_1, \theta_2 \in (\pi, 2\pi)$. Consider now a conjugate deformation of $C_1$ of the form $P^{-1} R(\theta_1) P$. From equation~\eqref{eqn18}, we have:
\begin{equation}
\mathrm{tr}(R(\theta_2) P R(\theta_1) P^{-1}) = 2 \cos \theta_2 \cos \theta_1 - (\mathrm{tr}(P^\top P)) \sin \theta_2 \sin \theta_1.
\end{equation}
Since $\sin(\theta_1)\sin(\theta_2) > 0$, by deforming $P$ so that $\mathrm{tr}(P^\top P)$ is sufficiently large, we can ensure that the trace becomes less than $-2$. Hence, $C_0 = (C_1(t) C_2)^{-1}$ can be deformed to a hyperbolic element with negative trace.

Finally, from the analysis in Cases 1, 2, and 3, we know that $C_1$ and $C_2$ can always be deformed into a configuration where both are elliptic. This completes the proof.
\end{proof}
 Assume now that $\mathrm{sign}(\phi_2) = 2|\chi(\Sigma_{g-1,1})|$. By the Claim above, we may deform $\phi_2$ such that $\mathrm{tr}(C_2)$ decreases monotonically to a value less than $-2$. Since both the Toledo invariant and the rho invariant vary continuously along this deformation path, $\phi_2$ can be deformed to a boundary hyperbolic representation with negative trace, while preserving the signature invariant.

This yields a deformation $C_2(t)$ of $C_2$ such that $C_2(1)$ is hyperbolic. By the solution of \eqref{eqn-deformC}, there exist elliptic paths $C_0(t)$ and $C_1(t)$ such that $(C_0(t) C_1(t))^{-1} = C_2(t)$. Consequently, we obtain a deformation $\psi$ of $\phi$ such that $C_0$ and $C_1$ are elliptic, while $C_2$ is hyperbolic.

Assume further that $C_2 =P^{-1} \mathrm{diag}(\lambda,1/\lambda)P$ with $\lambda < 0$, and $C_0 = R(\theta_0)$. From equation~\eqref{eqn16}, we have
\begin{equation*}
  \mathrm{tr}(C_2 C_0) = R \cos(\theta_0 - \alpha) \in (-2, 2).
\end{equation*}
From {\bf Case 1}, we obtain:
\begin{equation*}
  \mathrm{sgn}(\lambda^{-1} \sin \theta_0) = \mathrm{sgn}((C_2 C_0)_{21}) = -\mathrm{sgn}((C_1)_{21}) = -\mathrm{sgn}(\sin \theta_1).
\end{equation*}
Since $\theta_1 \in (\pi, 2\pi)$ and $\lambda < 0$, it follows that $\theta_0 \in (\pi, 2\pi)$. By reversing the process described in {\bf Case 1}, we can keep $C_2$ fixed while deforming $\theta_0$ across $2\pi$. During this deformation, there exist matrices $A(t), B(t)$ such that $[A(t), B(t)] = C_1(t) = (C_2 C_0(t))^{-1}$. Thus, we obtain a deformation $\phi(t)$ of $\phi$ such that $\phi(t)(c_0) \in \mathrm{SO}(2)$ and the rotation angle of $C_0(t)$ passes through $2\pi$.

Now consider the case when \( \theta_0 + \theta_1 + \theta_2 = 2\pi \). Then the representation $(C_1, C_2, C_0)$ is maximal. If \( \theta_1 \in (\pi, 2\pi) \), then the representation \( ((A_1, B_1), C_1^{-1}) \) is also maximal. Since \( \mathrm{sign}(\phi) < 2|\chi(\Sigma_{g,1})| \), the representation \( \phi_2 \) is not maximal. By the inductive hypothesis, the angle $\theta_2$ $(\in (0,\pi))$ corresponding to $C_2$ can be deformed through zero. Keeping $\theta_1$ fixed and letting $\theta_0 \nearrow 2\pi - \theta_1$, as $\theta_2$ passes through zero (i.e., as $C_2$ crosses the identity), we obtain a path where 
$\theta_2(t) \searrow 2\pi - \varepsilon$ and $\theta_0(t) \nearrow 2\pi - \theta_1 + \varepsilon < 2\pi$. Eventually, we reach the configuration $\theta_0 + \theta_1 + \theta_2 = 4\pi$, which reduces to a previously treated case.

In the case where \( \theta_1 \in (0, \pi) \), we fix $\theta_2$ and deform $\theta_0$ through $2\pi$, forcing $\theta_1$ to pass through zero. This corresponds to a deformation of a representation in $\mathrm{Hom}(\pi_1(\Sigma_{g-1,2}), \mathrm{SL}(2,\mathbb{R}))$ that extends to a representation of $\pi_1(\Sigma_{g,1})$, as explained in Section~\ref{secn=1}.
This completes the proof.
\end{proof}
We now state a lemma which provides a useful deformation of boundary elliptic-unipotent representations.
Let \( \{a_1, b_1, \dots, a_g, b_g, c_1, \dots, c_n \}\) be a standard set of generators for \( \pi_1(\Sigma_{g,n}) \). 
\begin{lemma}\label{lemma7}
If \( \phi \in \mathrm{Ell}(\Sigma_{g,n}) \) satisfies \( |\mathrm{sign}(\phi)| < 2|\chi(\Sigma_{g,n})| \) and any $1\leq i\leq g$, then there exists a deformation $\phi(t)\in \mathrm{E}(\Sigma_{g,n})$ of $\phi$ such that \(\phi(1)= \phi' \in \mathrm{Ell}(\Sigma_{g,n}) \) and $\phi'(a_{i}) = \phi'(b_{i}) = I$.
\end{lemma}

\begin{proof}
Since reversing the orientation of the surface changes the sign of the signature and maps a boundary elliptic representation with angle \( \theta \) to one with angle \( 2\pi - \theta \), we may assume without loss of generality that \( 0 \leq \mathrm{sign}(\phi) < 2|\chi(\Sigma_{g,n})| \). The case \( g = 1 \) is covered by Lemma~\ref{lemma3}, so we assume \( g \geq 2 \). 
On the other hand, since \(\mathrm{Hom}(\pi_1(\Sigma_{1,1}),\mathrm{SL}(2,\mb{R})) \) is connected, any representation \( \phi \in {Ell}(\Sigma_{g,n}) \) satisfying \( \phi([a_{i},b_{i}]) = I \) can be further deformed to a representation \( \phi' \in {Ell}(\Sigma_{g,n}) \) such that \( \phi'(a_{i}) = \phi'(b_{i}) = I \). Therefore, for our purpose, it suffices to deform \( \phi \) to a representation \( \phi' \) with \( \phi'([a_{i},b_{i}]) = I \). For simplicity, we prove the case where $i = 1$; the argument for a general index $i$ is entirely analogous.

We first consider the case \( n = 1 \). By the proof of Lemma~\ref{lemma12}, the representation \( \phi \) can be deformed to a representation \( \phi' \), such that \( C_1, C_2 \in \mathrm{SO}(2) \), and the boundary angles satisfy
\[
  \theta_0 + \theta_1 + \theta_2 = 4\pi \quad \text{or} \quad \theta_0 + \theta_1 + \theta_2 = 2\pi \text{ with } \theta_1 \in (0,\pi).
\]

In the latter case, we consider the two representations \( ((A_1, B_1), C_1^{-1}) \) and \( (C_1, C_2, C_0) \), whose signatures are \( -2 \) and \( 2 \), respectively. Thus, the remaining representation
\[
  \phi_2 := ((A_2, B_2), \dots, (A_g, B_g), C_2^{-1})
\]
has a non-negative signature, which cannot be minimal. By Lemma~\ref{lemma12}, we can deform \( \phi_2 \) such that the angle of \( C_2^{-1} \) passes through 0, i.e., the angle of \( C_2 \) increases through \( 2\pi \). During this deformation, we fix \( \theta_0 \), so \( \theta_2 \) increases to \( 2\pi - \theta_0 \), and \( \theta_1 \) decreases to 0. This deformation of \( \phi_2 \) extends to a deformation \( \phi' \) of \( \phi \) with \( \phi'([a_1, b_1]) = I \).

Now consider the case \( \theta_0 + \theta_1 + \theta_2 = 4\pi \).  If \( \theta_1 \in (0,\pi) \), then both \( ((A_1, B_1), C_1^{-1}) \) and \( (C_1, C_2, C_0) \) have signature \( -2 \), so the remaining piece
\[
  \phi_2 := ((A_2, B_2), \dots, (A_g, B_g), C_2^{-1})
\]
has a positive signature and is not minimal. By Lemma~\ref{lemma12}, \( C_2 \)'s angle can be deformed to cross \( 2\pi \), and fixing \( \theta_0 \), the \( \theta_1 \) still lies in \( (0,\pi) \), reducing to the previous case since now $\theta_0+\theta_1+\theta_2=2\pi$.
If \( \theta_1 \in (\pi, 2\pi) \), then we can deform \( \theta_1 \nearrow 2\pi \), while keeping \( \theta_2 \) fixed, and simultaneously deform \( \theta_0 \searrow 2\pi - \theta_2 \). This ensures that \( \theta_0(t) \in (0, 2\pi) \) throughout the deformation.

We now proceed by induction on \( n \). Suppose the lemma holds for all boundary components up to \( n-1 \), and consider the case of \( n \). By Lemma~\ref{lemma6}, we may assume that
\[
  \phi = ((A_1, B_1), \dots, (A_g, B_g), C_1, \dots, C_n)
\]
is a boundary \( \mathrm{SO}(2) \setminus \{\pm I\} \) representation, and \( C_0 := C_{n-1}C_n \in \mathrm{SO}(2) \setminus \{\pm I\} \). We decompose \( \phi \) along the separating curve \( c_0 \) into two pieces:
\[
  \phi_1 := ((A_1, B_1), \dots, (A_g, B_g), C_1, \dots, C_{n-2}, C_0^{-1}), \quad
  \phi_2 := (C_0, C_{n-1}, C_n).
\]
Then both \( \phi_1 \) and \( \phi_2 \) are boundary elliptic representations, and
\[
  \mathrm{sign}(\phi_1) + \mathrm{sign}(\phi_2) = \mathrm{sign}(\phi) \geq 0.
\]
Hence, \( \mathrm{sign}(\phi_1) \geq -2 > -2|\chi(\Sigma_{g,n-1})| \). 

If \( \mathrm{sign}(\phi_1) < 2|\chi(\Sigma_{g,n-1})| \), then by the inductive hypothesis, we can deform \( \phi_1(t) \in \mathrm{Ell}(\Sigma_{g,n-1}) \) to \( \phi_1' \) such that \( \phi_1'([a_1, b_1]) = I \). On the other hand, we can deform \( \phi_2(t) \) such that \( \phi_2(t)(c_0) \phi_1(t)(c_0) = I \). Gluing these two deformations gives a deformation \( \phi(t) = \phi_1(t) \cup \phi_2(t) \) with \( \phi(1)([a_1, b_1]) = I \).

If instead \( \mathrm{sign}(\phi_1) = 2|\chi(\Sigma_{g,n-1})| \),  then $\phi_1$ has maximal signature and is a boundary $\mathrm{SO}(2)\backslash \{\pm I\}$ representation. Hence,
 we can deform \( \phi_1 \) so that the angle \( 2\pi-\theta_0 \) of \( C_0^{-1} \) crosses 0, i.e., the angle \( \theta_0 \) increases through \( 2\pi \). Since \( \mathrm{sign}(\phi) < 2|\chi(\Sigma_{g,n})| \), we must have \( \mathrm{sign}(\phi_2) = -2 \), so
\[
  \theta_0 + \theta_{n-1} + \theta_n = 4\pi.
\]
As \( \theta_0 \) increases through \( 2\pi\), and to $\epsilon>0$, we fix \( \theta_{n-1} \), and decrease \( \theta_n \) to \( 2\pi - \theta_{n-1} - \epsilon >0\), so that, in the end, their sum is equal to \( 2\pi \). Therefore, \( \phi \) admits a deformation \( \phi(t) \in \mathrm{E}(\Sigma_{g,n}) \) to \( \phi(1) \), where \( \phi(1) \) decomposes as \( \phi_1(1) \) and \( \phi_2(1) \) with
\[
  \mathrm{sign}(\phi_1(1)) = 2|\chi(\Sigma_{g,n-1})| - 4 < 2|\chi(\Sigma_{g,n-1})|, \quad \mathrm{sign}(\phi_2(1)) = 2,
\]
which reduces to the case already covered. The proof is complete.
\end{proof}
\begin{rem}
By equation~\eqref{eqn19}, the deformation $\phi(t)$ does not, in general, remain in $\mathrm{Ell}(\Sigma_{g,n})$. 	
\end{rem}

\begin{lemma}\label{lemma11}
We have
\[
\# \left\{ \mathrm{sign}^{-1} \left(2|\chi(\Sigma_{g,n})| \right) \cap \mathrm{E}(\Sigma_{g,n}) \right\} = 4^g.
\]
\end{lemma}

\begin{proof}
If $\phi \in \mathrm{sign}^{-1} (2|\chi(\Sigma_{g,n})|) \cap \mathrm{E}(\Sigma)$, then $\phi$ has maximal signature. In particular, its restriction to each punctured torus $\Sigma_{1,1}$ contributes signature $2$, so
\[
\mathrm{tr} \left( \phi([a_i,b_i]) \right) = \mathrm{tr}([A_i,B_i]) < 2.
\]
Hence, both $A_i$ and $B_i$ are hyperbolic matrices. The argument in Lemma~\ref{lemma8} shows that each pair $(A_i,B_i)$ with $\mathrm{sign}((A_i,B_i),[A_i,B_i]^{-1}) = 2$ corresponds to $4$ connected components.
Therefore, the total number of components of $\mathrm{sign}^{-1}(2|\chi(\Sigma_{g,n})|) \cap \mathrm{E}(\Sigma_{g,n})$ is $4^g$.
\end{proof}

\begin{lemma}\label{lemma23}
Let $m \in [\chi(\Sigma_{g-1,n}),\, |\chi(\Sigma_{g-1,n})|] \cap \{|\chi(\Sigma_{g-1,n})| - 2\mathbb{Z}\}$. Then $\mathrm{sign}^{-1}(2m) \cap \mathrm{E}(\Sigma_{g,n})$ is a connected component.
\end{lemma}
\begin{proof}
Let \( \Lambda_{g,n} \subset \mathrm{Hom}(\pi_1(\Sigma_{g,n}), \{\pm I\}) \) denote the subgroup generated by the representations \( \alpha_1, \beta_1, \dots, \alpha_g, \beta_g \), where each \( \alpha_i \) (resp. \( \beta_i \)) sends \( a_i \) (resp. \( b_i \)) to \( -I \) and all other generators to the identity.

By repeatedly applying Lemma~\ref{lemma7}, one can deform any representation \( \phi \in \mathrm{sign}^{-1}(2m) \cap \mathrm{E}(\Sigma_{g,n}) \) either to a representation with maximal or minimal signature on a proper subsurface \( \Sigma_{g_0,n} \) where $g_0=1+(|m|-n)/2$, or to a representation defined on a genus-zero surface. Note that a representation on \( \pi_1(\Sigma_{g_0,n}) \) naturally identified to a representation on \( \pi_1(\Sigma_{g,n}) \) by defining it trivially on the remaining generators.

If \( \phi \) is deformed to a representation with maximal or minimal signature on a subsurface \( \Sigma_{g_0,n} \), then by Lemma~\ref{lemma11}, it lies in the orbit \( \gamma \cdot \phi_0 \) for some fixed \( \phi_0 \in \mathrm{sign}^{-1}(2m) \cap \mathrm{E}(\Sigma_{g_0,n}) \) and some \( \gamma \in \Lambda_{g_0,n} \). In the genus-zero case, Proposition~\ref{ell-prop2} implies that the space is already connected.

Therefore, it remains to show that for any \( \gamma \in \Lambda_{g_0,n} \), the representations \( \phi_0 \) and \( \gamma \cdot \phi_0 \) lie in the same connected component of \( \mathrm{E}(\Sigma_{g,n}) \). Suppose \( \gamma = \alpha_i \in \Lambda_{g_0,n} \). Then \( \phi_0 \) and \( \gamma \cdot \phi_0 \) differ only in their values on \( a_i \), by multiplication with \( -I \).
By Lemma~\ref{lemma7}, both \( \phi_0 \) and \( \gamma \cdot \phi_0 \) can be deformed to representations \( \psi_0 \) and \( \psi_1 \), respectively, such that \( \psi_j([a_i, b_i]) = I \) for \( j = 0,1 \), and \( \psi_0 \) and \( \psi_1 \) agree on all generators except possibly \( a_i \) and \( b_i \). Thus, their restrictions to the subsurface \( \Sigma_{1,0} \) define representations in \( \mathrm{Hom}(\pi_1(\Sigma_{1,0}), \mathrm{SL}(2,\mathbb{R})) \), which is connected. Therefore, \( \psi_0 \) and \( \psi_1 \) are path-connected within this space.

It follows that \( \phi_0 \) and \( \gamma \cdot \phi_0 \) lie in the same connected component of \( \mathrm{E}(\Sigma_{g,n}) \). The case \( \gamma = \beta_i \) is handled in the same way. The proof is complete.
\end{proof}

Let $\#\{\mathrm{E}(\Sigma_{g,n})\}$ denote the number of connected components of boundary elliptic-unipotent representations. Then 
\begin{align*}
 &\quad\,\,\#\{\mathrm{E}(\Sigma_{g,n})\}
 =\sum_{m\in [\chi(\Sigma_{g-1,n}),|\chi(\Sigma_{g-1,n})|]}\#\{\mathrm{sign}^{-1}(2m)\cap \mathrm{E}(\Sigma_{g,n})\}\\
&+\#\{\mathrm{sign}^{-1}(2|\chi(\Sigma_{g,n})|)\cap \mathrm{E}(\Sigma_{g,n})\}+\#\{\mathrm{sign}^{-1}(-2|\chi(\Sigma_{g,n})|)\cap \mathrm{E}(\Sigma_{g,n})\}\\
 &=2g+n-3+2\#\{\mathrm{sign}^{-1}(2|\chi(\Sigma_{g,n})|)\cap \mathrm{E}(\Sigma_{g,n})\}\\
 &=2^{2g+1}+2g+n-3,
\end{align*}
where the second equality by Lemma \ref{lemma23}, the last equality follows from Lemma \ref{lemma11}.
It follows that
\begin{equation*}
\#\{\mathrm{E}(\Sigma_{g,n})\}=  2^{2g+1}+2g+n-3,\quad\text{ for } n\geq 1, g\geq 0.
\end{equation*}
Indeed, we obtain the following corollary for closed surfaces.
\begin{cor}[{\cite[Theorem A (i)]{Gold1}}]The number of connected components for $\mathrm{Hom}(\pi_1(\Sigma_{g,0}),\mathrm{SL}(2,\br))$ is $2^{2g+1}+2g-3$.
\end{cor}
\begin{proof}
	For representations with maximal signature, equivalently those with maximal Toledo invariant, there are \( 4^g \) connected components, arising from the action \( (A_i, B_i) \mapsto (\pm A_i, \pm B_i) \). The same holds for representations with minimal Toledo invariant, giving another \( 4^g \) components.

	For representations satisfying \( |\mathrm{T}(\phi)| < |\chi(\Sigma_{g,0})| \), the argument is the same as in Lemmas~\ref{lemma7} and~\ref{lemma23}, as detailed in \cite[Lemma 10.5]{Gold1}. This shows that for each such Toledo invariant value, the set of corresponding representations forms a connected component, and there are \( 2g - 3 \) such values.

	In total, we obtain \( 2 \times 4^g + (2g - 3) \) connected components.
\end{proof}
Next, we determine the number of connected components of the space of boundary elliptic representations, denoted by \( \#\{\mathrm{Ell}(\Sigma_{g,n})\} \).

Assume that \( C_i \sim R(\theta_i) \) with \( \theta_i \in (0,\pi) \cup (\pi,2\pi) \) for each \( 1 \leq i \leq n \). For any representation 
\[
\phi \in \mathrm{sign}^{-1}(2|\chi(\Sigma_{g,n})|) \cap \mathrm{Ell}(\Sigma_{g,n}) \cap \sigma^{-1}(a),
\]
the signature is maximal. By Lemma~\ref{lemma6}, we may assume each \( C_i = R(\theta_i) \). Define \( \phi_1 = (C_0, C_1, \dots, C_n) \), where \( C_0 = (C_1 \cdots C_n)^{-1} \). Since \( \phi \) has maximal signature, so does \( \phi_1 \), and we must have
\[
2|\chi(\Sigma_{0,n+1})| = 2(n-1) = 2(n+1 - \tfrac{1}{\pi}( \theta_0+\theta_1 + \cdots + \theta_n )),
\]
which implies
\[
\theta_0+\theta_1 + \cdots + \theta_n  = 2\pi.
\]
On the other hand, the representation \[ \phi_2 = \phi|_{\pi_1(\Sigma_{g,1})} = ((A_1,B_1),\dots,(A_g,B_g),C_0^{-1}) \] is also maximal, which implies \( \theta_0 \in (\pi,2\pi) \) (see Equation~\eqref{eqn15}). Therefore,
\[
\theta_1 + \cdots + \theta_n \in (0, \pi),
\]
forcing \( a = (-1, \ldots, -1) \). Hence,
\[
\#\left\{\mathrm{sign}^{-1}(2|\chi(\Sigma_{g,n})|) \cap \mathrm{Ell}(\Sigma_{g,n})\right\} = 4^g.
\]
Similarly,
\[
\#\left\{\mathrm{sign}^{-1}(2\chi(\Sigma_{g,n})) \cap \mathrm{Ell}(\Sigma_{g,n})\right\} = 4^g.
\]
We now consider the case where \( |\mathrm{sign}(\phi)| < 2|\chi(\Sigma_{g,n})| \). 

If \( n = 1 \), then for any representation \( \phi \) with \( |\mathrm{sign}(\phi)| < 2|\chi(\Sigma_{g,1})| \), by Lemma~\ref{lemma12}, the corresponding rotation angle \( \theta_1 \) of boundary can be deformed continuously to both \( 0 \) and \( 2\pi \), provided \( \theta_1 \in (0,2\pi) \). This implies that, for each
\[
m \in [-\chi(\Sigma_{g,1}) + 2, |\chi(\Sigma_{g,1})| - 2] \cap \{ |\chi(\Sigma_{g,1})| - 2\mathbb{Z} \},
\]
both of the subsets
\[
\mathrm{sign}^{-1}(2m) \cap \{\theta_1 \in (0, \pi)\} \quad \text{and} \quad \mathrm{sign}^{-1}(2m) \cap \{\theta_1 \in (\pi, 2\pi)\}
\]
are non-empty.
On the other hand, since \( \mathrm{sign}^{-1}(2m) \cap \{\theta_1 \in (0, 2\pi)\} \) is connected, it follows that each of the above two subsets is itself connected. Therefore, for those boundary elliptic representations \( \phi \) with \( |\mathrm{sign}(\phi)| < 2|\chi(\Sigma_{g,1})| \), the space consists of exactly \( 2(2g - 2) \) connected components.
Combining this with the known count for the maximal and minimal signature cases, we obtain
\begin{equation*}
  \#\{ \mathrm{Ell}(\Sigma_{g,1}) \} = 2 \cdot 4^g + 2(2g - 2) = 2^{2g+1} + 4g - 4.
\end{equation*}

We now turn to the case \( n \geq 2 \). We begin by considering representations of neither maximal nor minimal signatures for which the product \( C_1 \cdots C_n \) cannot be deformed to the identity. If \( |\mathrm{sign}(\phi_2)| < 2|\chi(\Sigma_{g,1})| \), then by Lemma~\ref{lemma12}, the angle \( \theta_0 \) can be deformed to any value in \( [0, 2\pi] \). Therefore, there exists a deformation \( \phi_2(t) \) of \( \phi_2 \) such that
\(
\phi_2(t)(c_0) = \phi_1(t)(c_0)^{-1}.
\)
As a result, the glued representation \( \phi(t) = \phi_1(t) \cup_{c_0} \phi_2(t) \) satisfies
\(
\phi(1)(c_1 \cdots c_n)\). Hence, by Lemma~\ref{lemma12}, we must have \( \mathrm{sign}(\phi_2) = \pm 2\chi(\Sigma_{g,1}) \). 
  
 By the proof of \eqref{eqn19}, we obtain, in the following situations, \( C_1\cdots C_n \) cannot be deformed to \( I \),
\begin{equation}\label{eqn24}
\mathrm{sign}(\phi) = 
\begin{cases}
2(1-n+r_a )+\mathrm{sign}(\phi_2) & \text{if } r_a \in [2, n] \cap 2\mathbb{Z},\quad \mathrm{sign}(\phi_2) = 2\chi(\Sigma_{g,1}), \\
2(-1+r_a)+\mathrm{sign}(\phi_2) & \text{if } n-r_a\in[2,n]\cap 2\mb{Z},\,\mathrm{sign}(\phi_2) = 2|\chi(\Sigma_{g,1})|.
\end{cases}
\end{equation}
 These correspond to
\[
2 \cdot 4^g\sum_{r_a \in [2,n] \cap 2\mathbb{Z}} \binom{n}{r_a} =2^{2g+1}(2^{n-1}-1)\]
connected components.

For other representations $\phi\in \mathrm{Ell}(\Sigma_{g,n})$ with $|\mathrm{sign}(\phi)|<2|\chi(\Sigma_{g,n})|$ such that the representation $C_1\cdots C_n$ can be deformed to $I$. Then we have
\[
\mathrm{sign}(\phi) = \mathrm{sign}(\phi_2') + \mathrm{sign}(\phi_1'),
\]
where \( \phi_2' \in \mathrm{Hom}(\pi_1(\Sigma_{g,0}), \mathrm{SL}(2,\mathbb{R})) \) and \( \phi_1' \) is a boundary \( \mathrm{SO}(2)\setminus \{\pm I\} \) representation of \( \pi_1(\Sigma_{0,n}) \). Hence, the possible values of such non-maximal and non-minimal signatures are given by
\begin{align*}
\tfrac{1}{2} \mathrm{sign}(\phi) &= \tfrac{1}{2} \mathrm{sign}(\phi_2') + \tfrac{1}{2} \mathrm{sign}(\phi_1') \\
&\in \mathrm{T}(\phi_2') + \left[ -n + r_a + 1, r_a - 1 \right] \cap \{n - 2\mathbb{Z}\} \\
&\in \left[ -n + r_a + 1 - (2g - 2), r_a - 1 + (2g - 2) \right] \cap \{n - 2\mathbb{Z}\},
\end{align*}
for any \( a \in \{-1,1\}^n \), as in Theorem~\ref{thm4}. Observe that, by Lemma~\ref{lemma12} and \eqref{eqn24}, any two representations with the same signature and the same \( \sigma \)-value lie in the same connected component.
 The number of such values is
\begin{align*}
&\quad\,\,\#\left\{ \left[ -n + r_a + 1 - (2g - 2), r_a - 1 + (2g - 2) \right] \cap \{ n - 2\mathbb{Z} \} \right\} \\
&= \#\left\{ [r_a + 1, r_a - 1 + n + 2(2g - 2)] \cap 2\mathbb{Z} \right\} \\
&= \left\lfloor \tfrac{n - 1 + (r_a \bmod 2)}{2} \right\rfloor + 2g - 2.
\end{align*}
Therefore, summing over all \( a \in \{-1,1\}^n \), we obtain
\begin{align*}
&\quad\,\,\sum_{a \in \{-1,1\}^n} \left( \left\lfloor \tfrac{n - 1 + (r_a \bmod 2)}{2} \right\rfloor + 2g - 2 \right) \\
&= \sum_{r=0}^n \left( \left\lfloor \tfrac{n - 1 + (r \bmod 2)}{2} \right\rfloor + 2g - 2 \right) \binom{n}{r} \\
&= (n - 1)2^{n-1} + (2g - 2)2^n = (4g + n - 5)2^{n-1}.
\end{align*}

Putting everything together, the total number of connected components is
\begin{align*}
\#\{\mathrm{Ell}(\Sigma_{g,n})\} &= 4^g + 4^g +2^{2g+1}(2^{n-1}-1) + (4g + n - 5)2^{n-1} \\
&= 2^{n-1} \left( 2^{2g + 1} + 4g + n - 5 \right).
\end{align*}
In one word, we obtain
\begin{thm}
The number of connected components of the space of boundary elliptic-(unipotent) representations in \( \mathrm{SL}(2,\mathbb{R}) \) is given by
\begin{align*}
\#\{\mathrm{E}(\Sigma_{g,n})\} &= 2^{2g+1} + 2g + n - 3, \\
\#\{\mathrm{Ell}(\Sigma_{g,n})\} &= 2^{n-1} (2^{2g+1} + 4g + n - 5),
\end{align*}
for any \( g \geq 1 \) and \( n \geq 1 \). For $g=0$, we have $\#\{\mathrm{E}(\Sigma_{0,n})\}=n-1$ and $\#\{\mathrm{Ell}(\Sigma_{0,n})\}=2^{n-1}(n-1)$.
\end{thm}

\subsection{The case of $\mathrm{PSL}(2,\mathbb{R})$}

In this section, we compute the number of connected components of boundary elliptic representations in $\mathrm{PSL}(2,\mathbb{R})$.

Let $\pi: \mathrm{SL}(2,\mathbb{R}) \to \mathrm{PSL}(2,\mathbb{R})$ denote the natural projection. This induces a map which is a covering map on its inverse of a connected component (\cite[Lemma 2.2]{Gold1})
\begin{equation*}
  \pi_*: \mathrm{Hom}(\pi_1(\Sigma_{g,n}), \mathrm{SL}(2,\mathbb{R})) \to \mathrm{Hom}(\pi_1(\Sigma_{g,n}), \mathrm{PSL}(2,\mathbb{R})), \quad \phi \mapsto [\phi] := \pi \circ \phi.
\end{equation*}
Let $\mr{Ell}_P(\Sigma_{g,n})$ denote the set of boundary elliptic representations in $\mathrm{PSL}(2,\mathbb{R})$. The restriction of $\pi_*$ induces a surjective map
\begin{equation*}
  \pi_*: \mr{Ell}(\Sigma_{g,n}) \to \mr{Ell}_P(\Sigma_{g,n}),
\end{equation*}
since $n \geq 1$.

For any $\phi \in \mr{Ell}(\Sigma_{g,n})$, recall that $\sigma(\phi) = (\mathrm{\mathrm{sgn}}(\theta_1 - \pi), \dots, \mathrm{\mathrm{sgn}}(\theta_n - \pi))$, and define $r_{\sigma(\phi)} := \#\{i : \sigma(\phi)_i = -1\}$. For $[\phi] \in \mr{Ell}_P(\Sigma_{g,n})$, we define its signature by
\begin{equation*}
  \mathrm{sign}([\phi]) = 2\mathrm{T}(\phi) + 2(n - \tfrac{\sum_{i=1}^n [\theta_i]}{\pi}),
\end{equation*}
where $[\theta_i] := \theta_i$ (or $\theta_i - \pi$) is taken to lie in $(0,\pi)$. This signature is well-defined and continuous on $\mathrm{Ell}_P(\Sigma_{g,n})$. Moreover,
\begin{align*}
  \mathrm{sign}(\phi) &= 2\mathrm{T}(\phi) + 2(n - \tfrac{1}{\pi} \sum_{i=1}^n \theta_i) = \mathrm{sign}([\phi]) - 2(n - r_{\sigma(\phi)}),
\end{align*}
which yields the identity
\begin{equation}\label{eqn22}
  \mathrm{sign}([\phi]) = \mathrm{sign}(\phi) + 2(n - r_{\sigma(\phi)}).
\end{equation}

There is an involutive operation $\mu_{i,j} : \mr{Ell}(\Sigma_{g,n}) \to \mr{Ell}(\Sigma_{g,n})$ that maps $(\theta_i, \theta_j)$ to $(\theta_i \pm \pi, \theta_j \pm \pi)$ in $(0, 2\pi)$. It satisfies $[\mu_{i,j}(\phi)] = [\phi]$.

\begin{lemma}\label{lemma10}
For any $k \in \mathbb{Z}$, the set $\mathrm{sign}^{-1}(k) \cap \mr{Ell}_P(\Sigma_{g,n})$ is either empty or a connected component.
\end{lemma}
\begin{proof}
Suppose \( [\phi], [\psi] \in \mr{Ell}_P(\Sigma_{g,n}) \) have the same signature. By equation~\eqref{eqn22}, we have
\begin{equation*}
  \mathrm{sign}(\phi) + 2(n - r_{\sigma(\phi)}) = \mathrm{sign}(\psi) + 2(n - r_{\sigma(\psi)}),
\end{equation*}
which implies
\begin{equation}\label{eqn23}
  r_{\sigma(\phi)} - r_{\sigma(\psi)} = \tfrac{1}{2}(\mathrm{sign}(\phi) - \mathrm{sign}(\psi)) \in 2\mathbb{Z}.
\end{equation}
Thus, we may apply a finite sequence of transformations 
\(
\mu = \mu_{i_1, j_1} \circ \cdots \circ \mu_{i_s, j_s}
\)
such that
\[
\sigma(\mu(\phi)) = \sigma(\psi)
\quad \text{and hence} \quad
\mathrm{sign}(\mu(\phi)) = \mathrm{sign}(\psi).
\]

 If \( \mathrm{sign}(\mu(\phi)) = \pm 2|\chi(\Sigma_{g,n})| \), then by Lemma~\ref{lemma11}, \( [\phi] = [\mu(\phi)] \) and \( [\psi] \) are connected in \( \mathrm{Ell}_P(\Sigma_{g,n}) \).
If instead \( |\mathrm{sign}(\mu(\phi))| < 2|\chi(\Sigma_{g,n})| \), and we assume that \( \mu(\phi)(c_1 \cdots c_n) \) cannot be deformed to the identity, then by \eqref{eqn24}, we know that \( \mu(\phi) \) can be deformed to \( \gamma \cdot \psi \), where \( \gamma \in \Lambda_{g,n} \). Hence \( [\phi] = [\mu(\phi)] \) and \( [\psi] = [\gamma \cdot \psi] \) are connected in \( \mr{Ell}_P(\Sigma_{g,n}) \).
If, on the other hand, \( \mu(\phi)(c_1 \cdots c_n) \) can be deformed to the identity, then so can \( \psi(c_1 \cdots c_n) \) by \eqref{eqn24}. Let us assume \( \psi \) is deformed to \( \psi' \) with \( \psi'(c_1 \cdots c_n) = I \) and  \( \mu(\phi) \)   to \( \phi' \) with \( \phi'(c_1 \cdots c_n) = I \) by Lemma~\ref{lemma12} and \eqref{eqn24}, and moreover,
\[
\mathrm{sign}(\phi'|_{\pi_1(\Sigma_{g,0})}) = \mathrm{sign}(\psi'|_{\pi_1(\Sigma_{g,0})}).
\]
Since \( [\phi'|_{\pi_1(\Sigma_{g,0})}] \) and \( [\psi'|_{\pi_1(\Sigma_{g,0})}] \) are connected in \( \mathrm{Hom}(\pi_1(\Sigma_{g,0}), \mr{PSL}(2, \mathbb{R})) \) (see \cite[Theorem A (i)]{Gold1}), and \( \phi'|_{\pi_1(\Sigma_{0,n})} \) and \( \psi'|_{\pi_1(\Sigma_{0,n})} \) are also connected, it follows that \( [\phi'] \) and \( [\psi'] \) are connected in \( \mathrm{Ell}_P(\Sigma_{g,n}) \). Consequently, \( [\phi] = [\mu(\phi)] \) and \( [\psi] \) lie in the same connected component of \( \mathrm{Ell}_P(\Sigma_{g,n}) \).
\end{proof}

Thus, the number of connected components of $\mathrm{Ell}_P(\Sigma_{g,n})$ is given by
\begin{equation*}
  \#\{\mathrm{Ell}_P(\Sigma_{g,n})\} = \#\{k \in \mathbb{Z} : \mathrm{sign}^{-1}(k) \cap \mathrm{Ell}_P(\Sigma_{g,n}) \neq \emptyset\}.
\end{equation*}
For $g = 0$, we have
\begin{equation}\label{eqn11}
  \mathrm{sign}(\mathrm{Ell}_P(\Sigma_{0,n})) = [2, 2(n-1)] \cap 2\mathbb{Z}.
\end{equation}
Indeed, for any $[\phi] \in \mathrm{Ell}_P(\Sigma_{0,n})$, we have $\phi \in \mathrm{Ell}(\Sigma_{0,n})$, so
\begin{equation*}
\mathrm{sign}([\phi]) = 2(n - \sum_i \tfrac{[\theta_i]}{\pi}) = 2(n - a), \quad \text{where } a \in [1, n-1] \cap \mathbb{Z}.
\end{equation*}

Recall that
\begin{equation*}
    \mathrm{sign}([\phi]) = \mathrm{sign}(\phi) + 2(n - r_{\sigma(\phi)}) \in [2\chi(\Sigma_{g,n}), 2|\chi(\Sigma_{g,n})|] + 2(n - r_{\sigma(\phi)}).
\end{equation*}
By applying the $\mu$-action, we may assume $r_{\sigma(\mu(\phi))} \geq n-1$ or $r_{\sigma(\mu(\phi))} \leq 1$, with $[\phi]$ invariant. Then
\begin{align*}
  \mathrm{sign}([\phi]) &\in [2\chi(\Sigma_{g,n}) + 2(n-1), 2|\chi(\Sigma_{g,n})| + 2] \cap 2\mathbb{Z} \\
  &= [-4g + 2, 4g - 2 + 2n] \cap 2\mathbb{Z}.
\end{align*}
For $g = 0$, this reduces to $[2, 2(n-1)] \cap 2\mathbb{Z}$, agreeing with \eqref{eqn11}. For $g \geq 1$, note that if $\mathrm{sign}([\phi]) = 2|\chi(\Sigma_{g,n})| + 2$, then $\mathrm{sign}(\phi) = 2|\chi(\Sigma_{g,n})|$ and $r_{\sigma(\phi)} = n - 1$, implying one boundary value lies in $(\pi, 2\pi)$. However, in the previous section, we proved that for any $\phi \in \mathrm{sign}^{-1}(2|\chi(\Sigma_{g,n})|) \cap \mathrm{Ell}(\Sigma_{g,n})$ with $g \geq 1$, the vector $a = (-1,\ldots,-1)$ must hold, implying $r_a = n$, a contradiction. Therefore,
\begin{equation*}
  \mathrm{sign}([\phi]) \neq 2|\chi(\Sigma_{g,n})| + 2 \quad \text{and} \quad \mathrm{sign}([\phi]) \neq 2\chi(\Sigma_{g,n}) + 2(n-1).
\end{equation*}
Hence,
\begin{equation}\label{eqn12}
  \mathrm{sign}(\mathrm{Ell}_P(\Sigma_{g,n})) \subset [-4g + 4, 4g - 4 + 2n] \cap 2\mathbb{Z} \quad \text{for } g \geq 1.
\end{equation}

We now consider the case $n = 1$. For any $\phi \in \mathrm{Ell}(\Sigma_{g,1})$, we have
\begin{equation}\label{eqn14}
  \mathrm{sign}([\phi]) = 
  \begin{cases}
    \mathrm{sign}(\phi) + 2 & \text{if } \theta \in (\pi, 2\pi), \\
    \mathrm{sign}(\phi)     & \text{if } \theta \in (0, \pi).
  \end{cases}
\end{equation}
Note that $\mathrm{sign}(\phi) = 2|\chi(\Sigma)|$ cannot occur when $\theta \in (\pi, 2\pi)$, as
\begin{equation}\label{eqn15}
\mathrm{T}(\phi)= |\chi(\Sigma)| +\tfrac{\theta-\pi}{\pi}>|\chi(\Sigma)|\text{ for }\theta>\pi
\end{equation}
which contradicts the Milnor-Wood inequality \eqref{eqn25}. Similarly, for $\theta \in (0,\pi)$, one cannot have $\mathrm{sign}(\phi) = -2|\chi(\Sigma)|$. Hence,
\begin{align*}
  \mathrm{sign}(\mathrm{Ell}_P(\Sigma_{g,1})) &= [2\chi(\Sigma_{g,1}) + 2, 2|\chi(\Sigma_{g,1})|] \cap 2\mathbb{Z} \\
  &= [-4g + 4, 4g - 2] \cap 2\mathbb{Z}.
\end{align*}

Now let $\phi_1 \in \mathrm{Ell}(\Sigma_{g,1})$ and $\phi_2 \in \mathrm{Ell}(\Sigma_{0,n+1})$ be representations that can be glued along a boundary $c$. Denote their union by $\phi = \phi_1 \cup_c \phi_2$. Then
\begin{equation*}
  \mathrm{sign}([\phi]) = \mathrm{sign}([\phi_1]) + \mathrm{sign}([\phi_2]) - 2,
\end{equation*}
so the signature of $[\phi]$ lies in
\begin{equation*}
  [-4g + 4, 4g + 2n - 4] \cap 2\mathbb{Z} \subset \mathrm{sign}(\mathrm{Ell}_P(\Sigma_{g,n})).
\end{equation*}

Combining with \eqref{eqn12}, we obtain:

\begin{thm}
The possible values of the signature for boundary elliptic representations in $\mathrm{PSL}(2,\mathbb{R})$ are given by
\begin{equation}\label{eqn13}
  \mathrm{sign}(\mathrm{Ell}_P(\Sigma_{g,n})) = 
  \begin{cases}
    [2, 2(n-1)] \cap 2\mathbb{Z} & \text{for } g = 0, \\
    [-4g + 4, 4g - 4 + 2n] \cap 2\mathbb{Z} & \text{for } g \geq 1.
  \end{cases}
\end{equation}
In particular, the number of connected components is given by
\[
\#\{\mathrm{Ell}_P(\Sigma_{0,n})\} = n - 1, \quad \#\{\mathrm{Ell}_P(\Sigma_{g,n})\} = 4g + n - 3 \quad \text{for } g \geq 1.
\]
\end{thm}
Again, we obtain a corollary for the case of closed surfaces.
\begin{cor}[{\cite[Theorem A (i)]{Gold1}}]The number of connected components of $\mathrm{Hom}(\pi_1(\Sigma_{g,0}), \mathrm{PSL}(2,\br))$ is $4g -3$.
\end{cor}

\section{Boundary hyperbolic representations}\label{sec-hyperbolic}

In this section, we study the connected components of the space of boundary hyperbolic representations. Let $\mathrm{Hyp}(\Sigma_{g,n})$ denote the set of all representations $\phi \in \mathrm{Hom}(\pi_1(\Sigma_{g,n}), \mathrm{SL}(2,\mathbb{R}))$ such that each boundary holonomy $\phi(c_i)$ is hyperbolic, where $c_i$, $1 \leq i \leq n$, represent the boundary curves of $\Sigma_{g,n}$.

Any hyperbolic element $C \in \mathrm{SL}(2, \mathbb{R})$ is conjugate to a matrix of the form
\[
\alpha(\lambda) = \begin{pmatrix}
\lambda & 0 \\
0 & \frac{1}{\lambda}
\end{pmatrix}, \quad \lambda \in \mathbb{R} \setminus \{0, \pm 1\}.
\]
For each hyperbolic element $C$, we define the \emph{sign-type} invariant $\sigma(C)$ as
\[
\sigma(C) = \begin{cases}
0 & \text{if } \mathrm{tr}(C) < -2, \\ 
1 & \text{if } \mathrm{tr}(C) > 2.
\end{cases}
\]
For any $\phi \in \mathrm{Hyp}(\Sigma_{g,n})$, we define its boundary type invariant as
\[
\sigma(\phi) = \left( \sigma(\phi(c_1)), \dots, \sigma(\phi(c_n)) \right) \in \{0,1\}^n.
\]

Since hyperbolic elements have vanishing rho invariants, it follows that for any $\phi \in \mathrm{Hyp}(\Sigma_{g,n})$,
\[
\mathrm{T}(\phi) = \tfrac{1}{2} \cdot \mathrm{sign}(\phi) \in [-|\chi(\Sigma_{g,n})|, |\chi(\Sigma_{g,n})|] \cap \mathbb{Z}.
\]

\subsection{The case of pairs of pants}
In this section, we consider the pair of pants. We have the following lemma.
\begin{lemma}\label{lemma0}
Let $a = (a_1, a_2, a_3) \in \{0,1\}^3$. Then the subset
\[
\sigma^{-1}(a) \subset \mathrm{Hyp}(\Sigma_{0,3})
\]
is nonempty. Moreover, $\sigma^{-1}(a)$ is connected if $\sum_{i=1}^3 a_i$ is odd, and has two connected components if $\sum_{i=1}^3 a_i$ is even. In particular, the number of connected components of $\mathrm{Hyp}(\Sigma_{0,3})$ is $12$.
\end{lemma}

\begin{proof}
Let $a = (a_1, a_2, a_3)$ with $a_i \in \{0,1\}$. Consider the following representation:
\begin{equation}\label{hyp-eqn1}
\phi(c_1) = (-1)^{a_1 + 1} \begin{pmatrix}
\lambda_1 & 0 \\
0 & \frac{1}{\lambda_1}
\end{pmatrix}, \quad 
\phi(c_2) = (-1)^{a_2 + 1} \begin{pmatrix}
a & b \\
c & d
\end{pmatrix} \in \mathrm{SL}(2, \mathbb{R}),
\end{equation}
where $\lambda_1 \in (0,1)$, $c \neq 0$, and $a + d > 2$. The $\phi(c_3)$ is determined by the condition $\phi(c_3) = (\phi(c_1 c_2))^{-1}$ and lie in $\sigma^{-1}(a_3)$.

We compute:
\begin{equation}\label{hyp-eqn4.0}
\phi(c_1 c_2) = (-1)^{a_1 + a_2} \begin{pmatrix}
\lambda_1 a & \lambda_1 b \\
\frac{1}{\lambda_1} c & \frac{1}{\lambda_1} d
\end{pmatrix},
\end{equation}
and hence
\begin{align}\label{hyp-eqn4}
\begin{split}
\mathrm{tr}(\phi(c_1 c_2))& = (-1)^{a_1 + a_2} \left( \lambda_1 a + \tfrac{1}{\lambda_1} d \right)\\
& = 
(-1)^{a_1 + a_2} \left[ \lambda_1(a + d) + d \left( \tfrac{1}{\lambda_1} - \lambda_1 \right) \right].
\end{split}
\end{align}
We fix $b$, $c$ and $a + d$, and vary $|d|$ sufficiently large so that the trace $\mathrm{tr}(\phi(c_1 c_2))$ can be either positive or negative. Thus, $\phi(c_3)$ can be placed in $\sigma^{-1}(a_3)$, and $\phi \in \sigma^{-1}(a)$, so this subset is nonempty.

In particular, if $c = 0$, then $ad = 1$, which implies $a,d > 0$, and therefore
\[
\mathrm{sgn}(\mathrm{tr}(\phi(c_1 c_2))) = (-1)^{a_1 + a_2},
\]
which shows that $a_1 + a_2 + a_3$ must be odd. Conversely, if $a_1 + a_2 + a_3$ is odd, then we can take $\phi' \in \sigma^{-1}(a)$ defined by
\begin{equation}\label{hyp-eqn3}
\phi'(c_1) = (-1)^{a_1 + 1} \begin{pmatrix}
\frac{1}{2} & 0 \\
0 & 2
\end{pmatrix}, \quad
\phi'(c_2) = (-1)^{a_2 + 1} \begin{pmatrix}
\frac{1}{2} & 0 \\
0 & 2
\end{pmatrix}.
\end{equation}
Thus, $a_1 + a_2 + a_3$ is odd if and only if there exists a representation $\phi' \in \sigma^{-1}(a)$ with $c = 0$.

To prove connectivity of $\sigma^{-1}(a)$ when $\sum_{i=1}^3 a_i$ is odd, suppose $\phi \in \sigma^{-1}(a)$ is given by \eqref{hyp-eqn1}. When $c \neq 0$, define the deformation:
\[
\lambda_1(t) = \lambda_1 (1 - t) + \tfrac{1}{2}t \in (0,1), \quad c(t) = c(1 - t),
\]
and define $a(t), b(t), d(t)$ by solving
\begin{equation*}
\begin{cases}
\lambda_1(t)a(t) + \lambda_1(t)^{-1} d(t) = (\lambda_1 a + \lambda_1^{-1} d)(1 - t) + \left( \tfrac{1}{4} + 4 \right)t, \\
a(t) + d(t) = (a + d)(1 - t) + \left(2 + \tfrac{1}{2} \right)t, \\
a(t)d(t) - b(t)c(t) = 1.
\end{cases}
\end{equation*}
This gives a smooth path $\phi(t) \in \sigma^{-1}(a)$ connecting $\phi$ to the representation $\phi'$ given by \eqref{hyp-eqn3}.
If $c = 0$, then $a = \tfrac{1}{d}$, and we define $d(t) = d(1 - t) + 2t$, $c(t) \equiv 0$, and $b(t) = b(1 - t)$, which again defines a path in $\sigma^{-1}(a)$ connecting $\phi$ to $\phi'$.

Therefore, $\sigma^{-1}(a)$ is connected when $\sum_{i=1}^3 a_i$ is odd.

If $\sum_{i=1}^3 a_i$ is even, and two representations $\phi, \phi' \in \sigma^{-1}(a)$ have parameters $c, c'$ satisfying $cc' > 0$, then a similar deformation as above shows that $\phi$ and $\phi'$ lie in the same connected component.
However, if $cc' < 0$, then any path $\phi(t)$ connecting $\phi$ to $\phi'$ must pass through some $t_0$ with $c(t_0) = 0$. This implies that $\sum a_i$ is odd at $t_0$, which contradicts the assumption that all $\phi(t) \in \sigma^{-1}(a)$ with fixed $a$. Therefore, $\phi$ and $\phi'$ lie in different connected components.

This completes the proof.
\end{proof}

\begin{lemma}\label{par-lemma10}
For any representation $\phi \in \mathrm{Hyp}(\Sigma_{0,3})$, we have $\mathrm{T}(\phi) = 0$ if and only if $\sum_{i=1}^3 a_i$ is odd, where $a_i = \sigma(\phi(c_i))$.
\end{lemma}

\begin{proof}
Suppose $\phi = (C_1, C_2, C_3) \in \sigma^{-1}(a)$, where each $C_i$ is hyperbolic and $\sigma(C_i) = a_i \in \{0,1\}$. Consider the following identity:
\[
(-1)^{a_1+1} C_1 \cdot (-1)^{a_2+1} C_2 \cdot (-1)^{a_3+1} C_3 = (-1)^{a_1 + a_2 + a_3 + 1} I.
\]
Here, each matrix $(-1)^{a_i + 1} C_i$ is a hyperbolic element with positive trace, since multiplying by $-I$ switches the sign of the trace while preserving hyperbolicity.

The product of three positive-trace hyperbolic elements equals $(-1)^{a_1 + a_2 + a_3 + 1} I$. By the additivity properties of the Toledo invariant and its parity interpretation in this context, we conclude:
\[
\mathrm{T}(\phi) \equiv a_1 + a_2 + a_3 + 1 \mod 2\mathbb{Z}.
\]
Since $\mathrm{T}(\phi) \in \{-1, 0, 1\}$ for representations on a pair of pants, it follows that $\mathrm{T}(\phi) = 0$ if and only if $a_1 + a_2 + a_3$ is odd.
\end{proof}

\subsection{The general case}
Let $\mathrm{Hyp}_P(\Sigma_{g,n})$ denote the space of boundary hyperbolic representations from $\pi_1(\Sigma_{g,n})$ into $\mr{PSL}(2,\mathbb{R})$, that is, those representations in $\mathrm{Hom}(\pi_1(\Sigma_{g,n}), \mr{PSL}(2,\mathbb{R}))$ for which each boundary curve is mapped to a hyperbolic element. The following fundamental result is due to Goldman:

\begin{thm}[{\cite[Theorem 3.3]{Gold1}}]\label{thm:Gold1}
The connected components of $\mathrm{Hyp}_P(\Sigma_{g,n})$ are precisely the level sets $\mathrm{T}^{-1}(k)$ for integers $k$ satisfying $|k| \leq |\chi(\Sigma_{g,n})|$.
\end{thm}

We now define two canonical paths in $\mathrm{SL}(2,\mathbb{R})\backslash\{\pm I\}$.
\begin{defn}\label{defn1}
Let $\gamma_0 = \gamma_0(t)$ and $\gamma_1 = \gamma_1(t)$, for $t \in [0,1]$, be paths in $\mathrm{SL}(2,\mathbb{R}) \setminus \{\pm I\}$ satisfying the following conditions:

\begin{itemize}
  \item[(i)] $\mathrm{tr}(\gamma_i(0)) < -2$ and $\mathrm{tr}(\gamma_i(1)) > 2$, for $i = 0,1$;
  \item[(ii)] The trace function $\mathrm{tr}(\gamma_i(t))$ is strictly increasing in $t$;
  \item[(iii)] For any $t$ such that $\gamma_i(t)$ is elliptic, it is conjugate to a rotation matrix $R(\theta)$ with $\theta \in (0,\pi)$ if $i = 0$, and $\theta \in (\pi,2\pi)$ if $i = 1$.
\end{itemize}
We denote by $\gamma_i^{-1}$ the reversed path of $\gamma_i$.
\end{defn}

For any $\phi \in \mathrm{Hyp}(\Sigma_{g,n})$, suppose we fix the $\sigma$-values on all boundary components except $c_1$. Let $C_1 := \phi(c_1)$. 

If $\mathrm{tr}(C_1) < -2$, then deforming $C_1$ along the path $\gamma_0(t)$ necessarily crosses a parabolic element with trace $2$ and $\sigma$-value $+1$. Since the signature can only jump by $\pm 2$ when passing through a parabolic element with trace $2$, the signature jumps by $-2$ along $\gamma_0(t)$. Similarly, deforming $C_1$ along the path $\gamma_1(t)$ causes the signature to jump by $+2$.

On the other hand, if $\mathrm{tr}(C_1) > 2$, then deforming $C_1$ along $\gamma_0^{-1}(t)$ results in a signature jump of $+2$, while deforming along $\gamma_1^{-1}(t)$ results in a jump of $-2$.

In general, for any boundary hyperbolic representation, the Toledo invariant equals half the signature. Therefore, the Toledo invariant jumps by $k(2i - 1)$ along the path $\gamma_i^k(t)$, where $k = \pm 1$ and $i = 0,1$.

\begin{lemma}\label{hyp-lemma12}
Let $\phi \in \mathrm{Hyp}(\Sigma_{g,n})$ with $\mathrm{T}(\phi) < |\chi(\Sigma_{g,n})|$ (resp. $\mathrm{T}(\phi) > -|\chi(\Sigma_{g,n})|$). Then there exists a deformation $\phi(t)$, $t \in [0,1]$, of $\phi$ such that $\sigma(\phi(t)(c_i))$ remains constant for all $i \geq 2$, and $\phi(t)(c_1)$ satisfies:
\begin{itemize}
  \item[(i)] If $\mathrm{tr}(\phi(c_1)) < -2$ and $\phi(c_1) = \gamma_1(0)$ (resp. $\phi(c_1) = \gamma_0(0)$), then $\phi(t)(c_1) = \gamma_1(t)$ (resp. $\gamma_0(t)$);
  \item[(ii)] If $\mathrm{tr}(\phi(c_1)) > 2$ and $\phi(c_1) = \gamma_0^{-1}(0)$ (resp. $\phi(c_1) = \gamma_1^{-1}(0)$), then $\phi(t)(c_1) = \gamma_0^{-1}(t)$ (resp. $\gamma_1^{-1}(t)$).
\end{itemize}
\end{lemma}

\begin{proof}
By reversing the orientation of the surface, the Toledo invariant changes its sign. Therefore, the case $\mathrm{T}(\phi) > -|\chi(\Sigma_{g,n})|$ can be reduced to $\mathrm{T}(\phi) < |\chi(\Sigma_{g,n})|$. It suffices to prove the lemma under the assumption $\mathrm{T}(\phi) < |\chi(\Sigma_{g,n})|$.

We begin with the case \( g = 0 \). When \( n = 3 \), by Lemma~\ref{lemma0}, we can deform the representation \( \phi \) while keeping the conjugacy classes of \( C_1 \) and \( C_2 \) fixed, in such a way that the trace of the product \( C_1C_2 \) varies monotonically from greater than 2 to less than \(-2\), or vice versa. By Lemma~\ref{par-lemma10}, the Toledo invariant correspondingly changes from \( \pm 1 \) to 0, or from 0 to \( \pm 1 \). Hence, the lemma holds in the case \( n = 3 \).

Assume the lemma holds for all $\leq n $. For $ n+1$, consider the maximal dual-tree decomposition as illustrated on the right in Figure~\ref{fig:decomposition}. By \cite[Lemma 10.1]{Gold1}, we may assume that $D_1, \dots, D_{n-2}$ are hyperbolic elements.

If $\mathrm{T}(\phi|_{\pi_1(P_1)}) < 1$, then by the inductive hypothesis for $n=3$, there exists a deformation $\phi_1(t)$ of $\phi|_{\pi_1(P_1)}$ such that $D_1$ and the conjugacy class of $C_0$ are fixed, and $\phi_1(t)(c_1) = \gamma_1(t)$ or $\gamma_0^{-1}(t)$ depending on the sign of $\mathrm{tr}(C_1)$.

If $\mathrm{T}(\phi|_{\pi_1(P_1)}) = 1$, then $\mathrm{T}(\phi|_{\Sigma_{0,n+1} \setminus P_1})$ is not maximal. By induction, we may deform $\phi|_{\pi_1(P_1)}$ so that the conjugacy classes of $C_0$ and $C_1$ remain unchanged, and modify $D_1 = \phi|_{\pi_1(P_1)}(d_1)$ to alter its $\sigma$ value. Denote this deformation by $\phi_1(t)$, then $\mathrm{T}(\phi_1(1)) = 0$.
Similarly, we can construct a deformation $\phi_2(t)$ of $\phi|_{\Sigma_{0,n+1} \setminus P_1}$ such that its Toledo invariant increases by 1, and $\phi_1(t)(d_1) = \phi_2(t)(d_1)^{-1}$ for all $t$. Thus, we can glue the deformations $\phi_1(t)$ and $\phi_2(t)$ to obtain a deformation of $\phi$ whose restriction to $P_1$ has Toledo invariant zero. This reduces to the previous case $\mathrm{T}(\phi|_{\pi_1(P_1)}) < 1$, completing the proof for $g = 0$.

Now assume the lemma holds for all genus $\leq g-1$. Consider the case of genus $g$. Let 
\[
C_{-1} := C_2 \cdots C_{n} [A_1, B_1] \cdots [A_{g-1}, B_{g-1}], \quad C_0 := [A_g, B_g],
\]
so that $C_{-1} C_0 C_1 = I$, see Figure \ref{fig:decomposition3}.

\begin{figure}[ht]
\centering
\includegraphics[width=0.7\textwidth]{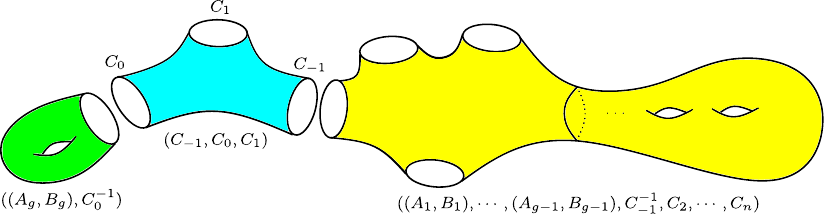}
\caption{Decomposition $\Sigma_{g,n} = \Sigma_{1,1} \cup \Sigma_{0,3}\cup\Sigma_{g-1,n}$}
\label{fig:decomposition3}
\end{figure}

 Define the following subrepresentations:
\begin{align*}
\phi_1 &:= (C_{-1}, C_0, C_1), \\
\phi_2 &:= ((A_1, B_1), \dots, (A_{g-1}, B_{g-1}), C_{-1}^{-1}, C_2, \dots, C_n), \\
\phi_3 &:= ([A_g, B_g], C_0^{-1}).
\end{align*}

By \cite[Lemma 10.1]{Gold1}, we may assume that $C_{-1}$ and $C_1$ are hyperbolic. If $\mathrm{T}(\phi_1) < 1$, then as in the case $n=3, g=0$, we may deform $\phi_1$ along a $\gamma$-path while fixing the conjugacy classes of $C_{-1}$ and $C_0$, so that the Toledo invariant of $\phi_1$ increases by 1.

If $\mathrm{T}(\phi_1) = 1$, then either $\mathrm{T}(\phi_2)$ or $\mathrm{T}(\phi_3)$ is not maximal. Suppose $\mathrm{T}(\phi_2)$ is not maximal. By induction, we can deform $\phi_1$ and $\phi_2$ into paths $\phi_1(t)$ and $\phi_2(t)$ such that the representations at the common boundary are mutual inverses. These can then be glued to obtain a global deformation. Furthermore, $\phi_1(t)$ preserves the conjugacy classes at the other two boundaries. Up to conjugation, we can glue $\phi_1(t)$, $\phi_2(t)$, and $\phi_3$ to construct a global deformation $\phi(t)$. Since $\phi_1(1)$ has Toledo invariant 0, this reduces the problem to the case $\mathrm{T}(\phi_1) < 1$.

The proof is complete.
\end{proof}

\begin{lemma}\label{par-lemma8}
The subset
\[
\mathrm{Hyp}(\Sigma_{g,n}) \cap \sigma^{-1}(a) \cap \mathrm{T}^{-1}(k)
\]
is either empty or a connected component of $\mathrm{Hyp}(\Sigma_{g,n})$ for any integer $k$ with $|k| \leq |\chi(\Sigma_{g,n})| - 1$.
\end{lemma}

\begin{proof}
The proof follows a similar strategy to that of Lemma \ref{lemma23}.
For any representation 
\[
\phi \in \mathrm{Hyp}(\Sigma_{g,n}) \cap \sigma^{-1}(a) \cap \mathrm{T}^{-1}(k),
\]
 By \cite[Lemma 10.1]{Gold1}, we may assume that it is an interior hyperbolic representation with respect to a maximal dual tree decomposition, see also Lemma \ref{par-lemma2}.
 
By repeatedly applying Lemma~\ref{hyp-lemma12} to the boundary hyperbolic representations \( \phi|_{\pi_1(\Sigma_{1,1})} \) and \( \phi|_{\pi_1(\Sigma_{g,n} \setminus \Sigma_{1,1})} \) associated with each subsurface \( \Sigma_{1,1} \subset \Sigma_{g,n} \), the representation \( \phi \) can be deformed into a new representation such that its restrictions to certain subsurfaces of type \( \Sigma_{1,1} \) lie in \( \mathrm{Hyp}(\Sigma_{1,1}) \cap \mathrm{T}^{-1}(0) \), while the remainder of the surface is covered either by maximal or minimal representations, or by representations of type \( \pi_1(\Sigma_{0,g+n}) \). We denote by $\phi_0$ such a fixed representation, which is also in $\mathrm{Hyp}(\Sigma_{g,n}) \cap \sigma^{-1}(a) \cap \mathrm{T}^{-1}(k)$.

Note that \( \mathrm{Hyp}(\Sigma_{1,1}) \cap \mathrm{T}^{-1}(0) \) is connected. This follows from the same argument as in \cite[Lemma 10.5]{Gold1} and the sigma value of each representation in $\mathrm{Hyp}(\Sigma_{1,1}) \cap \mathrm{T}^{-1}(0)$ is always equal to $1$ (see Table \ref{sign-table}), since one can construct a representation \( \phi_1 \in \mathrm{Hyp}(\Sigma_{1,1}) \cap \mathrm{T}^{-1}(0) \) such that \( \phi_1(a_1) \) or \( \phi_1(b_1) \) is elliptic. Furthermore, if the set
\(
\mathrm{Hyp}(\Sigma_{0,g+n}) \cap \mathrm{T}^{-1}(l) \cap \sigma^{-1}(a)
\)
is non-empty, by Theorem~\ref{thm:Gold1} and the fact that $\sigma^{-1}(a)$ is fixed, it is connected.
By Theorem~\ref{thm:Gold1}, any $\phi$ as above can be deformed into a representation of the form $\gamma \phi_0$, where $
\gamma \in \Lambda_{g,n}$. 

Using Lemma~\ref{hyp-lemma12} and the argument from Lemma \ref{lemma23}, one can show that $\phi_0$ and $\gamma \phi_0$ lie in the same connected component. Hence any $\phi$ in the given subset can be connected to $\phi_0$, completing the proof.
\end{proof}

\begin{lemma}\label{lemma9}
Let $\phi \in \mathrm{Hyp}(\Sigma_{g,n})$ be a boundary hyperbolic representation with $\mathrm{T}(\phi) = \pm |\chi(\Sigma_{g,n})|$. Then the number of entries equal to $1$ in $\sigma(\phi)$ is even.
\end{lemma}

\begin{proof}
We first consider the case $n = 1$. For $g = 1$, by the discussion in Section \ref{secn=1}, we know that $\mathrm{T}(\phi) = \pm 1$ if and only if $\sigma(C_1) = 0$, where $C_1 := \phi(c_1)$. Assume inductively that the lemma holds for genus at most $g - 1$, and consider the case $g$.

Write $C_3 := \Pi_{i=1}^{g-1} \phi([a_i, b_i])$ and $C_2 := \phi([a_g, b_g])$, so that $C_3 C_2 C_1 = I$. By \cite[Lemma 10.1]{Gold1}, we may assume that $C_2$ and $C_3$ are both hyperbolic. Since $\mathrm{T}(\phi) = \pm |\chi(\Sigma_{g,1})| = \pm (2g - 1)$, and by the inductive assumption, we have $\sigma(C_2) = \sigma(C_3) = 0$. Furthermore, the triple $(C_1, C_2, C_3)$ forms a representation of $\Sigma_{0,3}$ with Toledo invariant zero if and only if $\sum_{i=1}^3 \sigma(C_i)$ is odd. Hence, $T((C_1, C_2, C_3)) = \pm 1$ implies $\sigma(C_1) = 0$.

We now consider the case $n \geq 2$. Again, by \cite[Lemma 10.1]{Gold1}, we may assume that $C_0$, $D_1, \dots, D_{n-2}$ are all hyperbolic (see Figure~\ref{fig:decomposition}).

\begin{figure}[ht]
\centering
\includegraphics[width=0.7\textwidth]{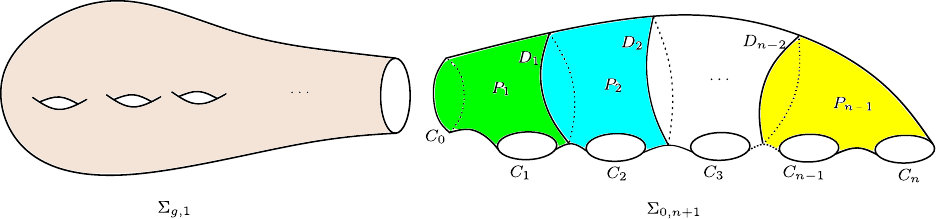}
\caption{Decomposition $\Sigma_{g,n} = \Sigma_{g,1} \cup_{c_0} \Sigma_{0,n+1}$}
\label{fig:decomposition}
\end{figure}

Since $\mathrm{T}(\phi)$ is maximal, the Toledo invariants of the restrictions of $\phi$ to both $\pi_1(\Sigma_{g,1})$ and to each pair of pants $P_i \subset \Sigma_{0,n+1}$ are also maximal. In particular, we have $\mathrm{tr}(C_0) < -2$.

If $\mathrm{tr}(D_1) > 2$, we can apply the sign-change operation $\mu_{12}$, which multiplies $-I$ to the boundary monodromies $C_1$ and $C_2$, producing a new representation $\mu_{12}(\phi)$. This operation preserves the Toledo invariant and stays within $\mathrm{Hyp}(\Sigma_{g,n})$, while changing the sign of the trace. Repeating this procedure, we may assume that all $D_i$ have negative trace, leading to a representation
\[
\phi' = \mu_{i_1j_1} \circ \cdots \circ \mu_{i_rj_r}(\phi),
\]
where each $\mu_{ij}$ acts by multiplying $-I$ on the $i$-th and $j$-th boundaries.

Since each pair of pants contributes $\pm 1$ to the Toledo invariant, the signature of $\phi'$ must satisfy
\[
\sigma(\phi') \in \{(0,\dots,0,0,0),\ (0,\dots,0,1,1)\}.
\]
In other words,
\[
\sigma(\phi') \in \mathcal{S} := \left\{ \mu_{i_1j_1} \circ \cdots \circ \mu_{i_rj_r}(0,\dots,0) \right\}.
\]
Since the operation $\mu_{ij}$ flips two components of the vector, each application preserves the parity of the number of 1's. Thus, every element of $\mathcal{S}$ contains an even number of 1's, implying that $\sigma(\phi) \in \mathcal{S}$ must also contain an even number of 1's. This completes the proof.
\end{proof}

As a consequence of Lemma~\ref{lemma9}, we conclude that the number of connected components in
\[
\mathrm{Hyp}(\Sigma_{g,n}) \cap \mathrm{T}^{-1}(|\chi(\Sigma_{g,n})|)
\]
is given by $4^g \cdot 2^{n-1}$, where the factor $4^g$ comes from the action $(A,B)\to (\pm A,\pm B)$ for each genus one surface. Similarly, the number of connected components on $\mathrm{Hyp}(\Sigma_{g,n}) \cap \mathrm{T}^{-1}(-|\chi(\Sigma_{g,n})|)$ is also $4^g\cdot 2^{n-1}$.
\begin{lemma}
	For any integer $k \in \mathbb{Z}$ with $|k| < |\chi(\Sigma_{g,n})|$, the space 
	\[
	\mr{Hyp}(\Sigma_{g,n}) \cap \mathrm{T}^{-1}(k)
	\]
	consists of exactly $2^{n-1}$ connected components.
\end{lemma}

\begin{proof}
	Let $\phi_0 \in \mr{Hyp}(\Sigma_{g,n}) \cap \mathrm{T}^{-1}(k)$ be a representation, and denote $a := \sigma(\phi_0) \in \{0,1\}^n$. For any pair of indices $i, j$, the representation $\mu_{ij}(\phi_0)$ also lies in $\mr{Hyp}(\Sigma_{g,n}) \cap \mathrm{T}^{-1}(k)$ and differs from $\phi_0$ by flipping the $i$-th and $j$-th entries of $\sigma$. Thus, $\phi_0$ and $\mu_{ij}(\phi_0)$ belong to different connected components.

	By iterating such operations, the set 
	\[
	\left\{ \mu_{i_1j_1} \circ \cdots \circ \mu_{i_rj_r}(\phi_0) \right\}
	\subset \mr{Hyp}(\Sigma_{g,n}) \cap \mathrm{T}^{-1}(k)
	\]
	contains exactly $2^{n-1}$ distinct representations, all in different connected components. Therefore, the space has at least $2^{n-1}$ connected components.

	On the other hand, from the same proof of Lemma~\ref{par-lemma10}, we know that
	\[
	k \equiv \sum_{i=1}^n a_i + n \quad \text{mod }2\mb{Z}.
	\]
	This restricts the possible values of $\sigma(\phi)$ to a subset of $\{0,1\}^n$ of size $2^{n-1}$. Since for each such admissible $\sigma$-value $a$, the set
	\[
	\mr{Hyp}(\Sigma_{g,n}) \cap \mathrm{T}^{-1}(k) \cap \sigma^{-1}(a)
	\]
	is connected by Lemma \ref{par-lemma8}, it follows that $\mr{Hyp}(\Sigma_{g,n}) \cap \mathrm{T}^{-1}(k)$ has at most $2^{n-1}$ connected components.

	Combining both bounds, we conclude that the number of connected components of $\mr{Hyp}(\Sigma_{g,n}) \cap \mathrm{T}^{-1}(k)$ is $2^{n-1}$.
\end{proof}

Hence, we obtain that
\begin{thm}
The number of connected components of the space of boundary hyperbolic representations $\mr{Hyp}(\Sigma_{g,n})$ is given by
\begin{align*}
  \#\left\{ \mr{Hyp}(\Sigma_{g,n}) \right\} 
  &= 2^n \cdot 4^g + (2|\chi(\Sigma_{g,n})| - 1)\cdot 2^{n-1} \\
  &= 2^{2g + n} + 2^{n-1}(4g + 2n - 5).
\end{align*}
\end{thm}

\section{Boundary parabolic representations}\label{sec-parabolic}

In this section, we consider representations \( \phi \colon \pi_1(\Sigma) \to \mathrm{SL}(2,\mathbb{R}) \) of the fundamental group of a surface \( \Sigma=\Sigma_{g,n} \) of genus \( g \) with \( n \) boundary components, such that the image of each boundary curve \( \phi(c_i) \) is parabolic.

\subsection{Pairs of pants and punctured torus}
In this section, we consider the case of surfaces \( \Sigma = \Sigma_{0,3} \) and $\Sigma_{1,1}$. 

\subsubsection{The case of pairs of pants}

Any parabolic element \( C \in \mathrm{SL}(2,\mathbb{R}) \) is conjugate to an element of the form
\begin{equation*}
  \alpha^\pm(s) = \pm\begin{pmatrix}
  1 & s \\
  0 & 1
  \end{pmatrix},
\end{equation*}
where \( s = \pm 1 \). If \( C \sim \alpha^\pm(s) \), we define the associated \( \sigma \)-value as
\begin{equation}\label{par-sigma map}
  \sigma(C) =
  \begin{cases}
    -\mathrm{sgn}(s) & \text{if } \mathrm{tr}(C) = 2, \\
    \mathrm{sgn}(s)\sqrt{-1} & \text{if } \mathrm{tr}(C) = -2.
  \end{cases}
\end{equation}
In particular, the rho invariant is given by \( \boldsymbol{\rho}(C) = \mathrm{Re}(\sigma(C)) \), and we have the identity \( \sigma(C^{-1}) = -\sigma(C) \).

Let \( \mathrm{Par}(\Sigma) \) denote the space of boundary parabolic representations of \( \pi_1(\Sigma) \) into \( \mathrm{SL}(2,\mathbb{R}) \). We define the {\it Toledo map}:
\begin{equation*}
  \mathrm{T} \colon \mathrm{Par}(\Sigma) \to \mathbb{Z}, \quad \phi \mapsto \mathrm{T}(\phi),
\end{equation*}
and the {\it sigma map}:
\begin{equation}\label{par-eqn-sigma}
  \sigma \colon \mathrm{Par}(\Sigma) \to \{ \pm 1, \pm \sqrt{-1} \}^n, \quad \phi \mapsto \left( \sigma(\phi(c_1)), \dots, \sigma(\phi(c_n)) \right).
\end{equation}

For any $\phi \in \mathrm{Par}(\Sigma)$, we typically use a vector 
$a \in \{ \pm 1, \pm \sqrt{-1} \}^n$ to represent its sigma value, 
that is, $a = \sigma(\phi)$.  
For a given sigma value $a$, there is an associated vector, called the \emph{$s$-value}, 
defined by
\begin{equation}\label{par-eqn-s}
  s=s_a = (s_1, \dots, s_n) \in \{\pm 1\}^n, 
  \quad \text{where} \quad s_i = \mathrm{Im}(a_i) - \mathrm{Re}(a_i).
\end{equation}
Note that the $s$-value, like the Toledo invariant, satisfies the following property:  
if two representations in $\mathrm{Par}(\Sigma)$ are both liftings of the same 
representation in $\mathrm{Hom}(\pi_1(\Sigma_{g,n}),\mathrm{PSL}(2,\mathbb{R}))$, 
then their $s$-values are identical.

It is known that two representations in \( \mathrm{Par}(\Sigma) \) lie in different connected components if their Toledo invariants or \( \sigma \)-values differ. Conversely, we have the following:

\begin{lemma}
For any \( a \in \{ \pm 1, \pm \sqrt{-1} \}^3 \), the subset
\[
\sigma^{-1}(a) \subset \mathrm{Par}(\Sigma_{0,3})
\]
is either empty or connected.
\end{lemma}

\begin{proof}
Assume the set \(\sigma^{-1}(a)\subset\mathrm{Par}(\Sigma_{0,3})\) is non-empty. For any two representations 
\(\phi,\psi\in\sigma^{-1}(a)\),  
we must construct a path
\[
\phi(t)\in\sigma^{-1}(a),
\quad t\in[0,1],
\]
with \(\phi(0)=\phi\) and \(\phi(1)=\psi\).

Because each boundary image is parabolic, after conjugation, we may assume
\[
C_1=\Phi(a_1), 
\quad 
C_2=P\Phi(a_2)P^{-1},
\quad
P=\begin{pmatrix}a&b\\ c&d\end{pmatrix}\in \mathrm{SL}(2,\mathbb{R}),
\]
where, for \(a\in\{\pm1,\pm\sqrt{-1}\}\),
\begin{equation}\label{par-eqn-Phi}
\Phi(a)=(-1)^{\mathrm{Re}(a)+1}\begin{pmatrix}1&\mathrm{Im}(a)-\mathrm{Re}(a)\\ 0&1\end{pmatrix}.
\end{equation}
Set \(s=\mathrm{Im}(a_1)-\mathrm{Re}(a_1)\) and \(\mu=\mathrm{Im}(a_2)-\mathrm{Re}(a_2)\). 
A direct computation shows
\begin{equation}\label{par-eqn2}
C_1C_2
=(-1)^{\mathrm{Re}(a_1)+\mathrm{Re}(a_2)}
\begin{pmatrix}
1-\mu ac-\mu sc^2 & s+\mu a^2+\mu sac \\
-\mu c^2           & 1+\mu ac
\end{pmatrix}.
\end{equation}
This product is parabolic precisely if \(c=0\), or \(c=\pm2\) and \(\mu s=1\). The condition \(\sigma(C_1C_2)=-a_3\) is equivalent to
\[
\begin{cases}
(-1)^{\mathrm{Re}(a_1)+\mathrm{Re}(a_2)}(2-\mu sc^2)=
(-1)^{\mathrm{Re}(-a_3)+1}\cdot 2,\\[4pt]
\mathrm{sgn}\bigl((-1)^{\mathrm{Re}(a_1)+\mathrm{Re}(a_2)}(s+\mu a^2)\bigr)\\=
\mathrm{sgn}\bigl((-1)^{\mathrm{Re}(-a_3)+1}(\mathrm{Im}(-a_3)-\mathrm{Re}(-a_3))\bigr),\quad \text{ for }c=0,\\[4pt]
\mathrm{sgn}\Bigl((-1)^{\mathrm{Re}(a_1)+\mathrm{Re}(a_2)+1}\bigl(-\tfrac1{\mu c^2}\bigr)\Bigr)\\=
\mathrm{sgn}\bigl((-1)^{\mathrm{Re}(-a_3)+1}(\mathrm{Im}(-a_3)-\mathrm{Re}(-a_3))\bigr),\quad\text{ for } c=\pm2.
\end{cases}
\]

\smallskip
\emph{Case (a): \(\mathrm{Re}(a_1)+\mathrm{Re}(a_2)+\mathrm{Re}(a_3)\) is odd.}  
Then the first equation forces \(c=0\). By suitably choosing \(P\) and \(P'\), we may assume \(a,a'\ge0\) for
\[
\phi(c_2)=P\Phi(a_2)P^{-1},\quad \psi(c_2)=P'\Phi(a_2)P'^{-1}.
\]
Then define \(a(t)=a+(a'-a)t\). Without loss of generality assume \(a\leq a'\); hence \(a(t)\in[a,a']\) is increasing, and we have
\[
s+\mu a^2\le s+\mu a(t)^2\le s+\mu a'^2.
\]
Thus, \(\mathrm{sgn}(s+\mu a(t)^2)\) is constant and equal to \(\mathrm{sgn}(-\mathrm{Im}(a_3)+\mathrm{Re}(a_3))\). Hence the resulting path \(\phi(t)\) lies entirely in \(\sigma^{-1}(a)\).

\smallskip
\emph{Case (b): \(\mathrm{Re}(a_1)+\mathrm{Re}(a_2)+\mathrm{Re}(a_3)\) is even.}  
In this case, \(\mu s=1\), and we have \(c=2\) by choosing $P\in \mathrm{SL}(2,\mb{R})$ with $c\geq 0$. The third equation reduces to
\[
\mathrm{sgn}(\mu)=\mathrm{sgn}(\mathrm{Im}(a_3)-\mathrm{Re}(a_3)).
\]
We can deform \(P\) into \(P'\) continuously while keeping \(c(t)\equiv 2\). This deformation produces a path \(\phi(t)\) connecting \(\phi\) and \(\psi\) within \(\sigma^{-1}(a)\).

Thus, \(\sigma^{-1}(a)\) is connected when \(n=3\).
\end{proof}

\begin{lemma}\label{par-lemma1}
For any \( a \in \{ \pm 1, \pm \sqrt{-1} \}^3 \) and $\phi\in \sigma^{-1}(a)$, the Toledo invariant is given as follows.
\[
\mathrm{T}(\phi)=\mathrm{T}(a):= \begin{cases}
 	\pm 1&\text{ for }a=\pm (\sqrt{-1},\sqrt{-1},\sqrt{-1}),\pm (-1,-1,\sqrt{-1}),\\
 	&\quad \quad\pm (-1,\sqrt{-1},-1),\pm (\sqrt{-1},-1,-1),\\
 	0&\text{ for }\sum_{i=1}^3 \mathrm{Re}(a_i) \text{ is odd}.
 \end{cases}
\]

\end{lemma}

\begin{proof}
Suppose \(\phi=(C_1,C_2,C_3)\in\sigma^{-1}(a)\), so that each \(C_i\) is parabolic and \(C_i\sim \Phi(a_i)\). We then have the identity
\[
(-1)^{\mathrm{Re}(a_1)+1}C_1\cdot (-1)^{\mathrm{Re}(a_2)+1}C_2\cdot(-1)^{\mathrm{Re}(a_3)+1}C_3=(-1)^{\mathrm{Re}(a_1)+\mathrm{Re}(a_2)+\mathrm{Re}(a_3)+1}I.
\]

Since each element \((-1)^{\mathrm{Re}(a_i)+1}C_i\) lies in the set of parabolic matrices with trace \(2\), we obtain
\[
\mathrm{T}(\phi)\equiv \boldsymbol{\rho}+1 \mod 2\mathbb{Z},
\]
where \(\boldsymbol{\rho}=\sum_{i=1}^3\mathrm{Re}(a_i)\in [-3,3]\) is the rho invariant. Note that \(\mathrm{T}(\phi)\in \{-1,0,1\}\). Thus,
\[
\mathrm{T}(\phi)=0,\quad\text{whenever \(\boldsymbol{\rho}\) is odd.}
\]

If \(\boldsymbol{\rho}\) is even, we have \(\boldsymbol{\rho}\in\{-2,0,2\}\) and consequently \(\mathrm{T}(\phi)\in\{-1,1\}\). For \(\boldsymbol{\rho}=2\), there must be exactly two values of \(\mathrm{Re}(a_i)\) equal to \(1\) and one equal to \(0\). Since
\[
\mathrm{sign}(\phi)=2\mathrm{T}(\phi)+\boldsymbol{\rho}\in [-2,2],
\]
it follows that \(\mathrm{T}(\phi)=-1\). Similarly, if \(\boldsymbol{\rho}=-2\), we obtain \(\mathrm{T}(\phi)=1\).

For \(\boldsymbol{\rho}=0\), we must have \(\{\mathrm{Re}(a_i)\}=\{-1,0,1\}\) or $Re(a)=(0,0,0)$. In this case \(\{\mathrm{Re}(a_i)\}=\{-1,0,1\}\), we can multiply by $-I$ to the boundary with $\sigma$-value of $\pm 1$, such that $\mathrm{Re}(a) = (0,0,0)$.
Hence, we consider the case \(\mathrm{Re}(a)=(0,0,0)\). Multiplying two boundary elements by \(-I\) leaves the Toledo invariant unchanged, but the rho invariant becomes \((-\,\mathrm{Im}(a_i), -\,\mathrm{Im}(a_j),0)\). By the previous argument, in this case, we have \(\mathrm{Im}(a_i)\mathrm{Im}(a_j)=1\), implying \(\mathrm{Im}(a_i)=\mathrm{Im}(a_j)\). Hence, we have
\(
\mathrm{T}(\phi)=\mathrm{Im}(a_i).
\)
\end{proof}

By conjugation, we may assume
\begin{equation}\label{par-eqn11}
C_1=\Phi(a_1),\quad C_2=P\Phi(a_2)P^{-1},\quad P=\begin{pmatrix}a&b\\ c&d\end{pmatrix}\in \mathrm{SL}(2,\mathbb{R}),c\geq 0.
\end{equation}
Then we have
\begin{equation}\label{par-eqn3}
C_1C_2
=(-1)^{\mathrm{Re}(a_1)+\mathrm{Re}(a_2)}
\begin{pmatrix}
1-\mu ac-\mu sc^2 & s+\mu a^2+\mu sac \\
-\mu c^2          & 1+\mu ac
\end{pmatrix},
\end{equation}
where \( s=\mathrm{Im}(a_1)-\mathrm{Re}(a_1) \) and \( \mu=\mathrm{Im}(a_2)-\mathrm{Re}(a_2) \).

For fixed \( a_1, a_2 \), the trace of \( C_1C_2 \) is given by
\begin{equation}\label{par-eqn4}
\mathrm{tr}(C_1C_2)=(-1)^{\mathrm{Re}(a_1)+\mathrm{Re}(a_2)}(2-\mu s c^2).
\end{equation}

If \( C_1C_2 \) is parabolic, we must have either \( c=0 \) or \( c=2 \) with \( \mu s=1 \).

\medskip
\noindent\textbf{Case 1: \( c=0 \).} Then
\begin{equation}\label{par-eqn5}
C_1C_2=(-1)^{\mathrm{Re}(a_1)+\mathrm{Re}(a_2)}
\begin{pmatrix}
1 & s+\mu a^2 \\
0 & 1
\end{pmatrix}.
\end{equation}

If \( s\mu=1 \), then the sign of \( s+\mu a^2 \) is fixed, and the conjugacy class of \( C_3=(C_1C_2)^{-1} \) is uniquely determined. 

If \( s\mu=-1 \), then the sign of \( s+\mu a^2 \) may be positive or negative, hence there are precisely two distinct conjugacy classes for \( C_3=(C_1C_2)^{-1} \).

\medskip
\noindent\textbf{Case 2: \( c=2 \).} Then \(\mu s=1\), and we have \(2-\mu s c^2=-2\) and \(-\mu c^2=-4\mu\). Thus, the conjugacy class of \( C_3=(C_1C_2)^{-1} \) is uniquely determined.

\medskip
In conclusion, the number of connected components of boundary parabolic representations is
\begin{align}\label{par-eqn13}
\begin{split}
\#\{\mr{Par}(\Sigma_{0,3})\}&=\#\{a_1,a_2 : \mu s=1, c\in\{0,2\}\}\\
&\quad +\#\{a_1,a_2 : \mu s=-1, c=0, s+\mu a^2\neq 0\}\\
&=2\,\#\{a_1,a_2 : \mu s=1\}+2\#\{a_1,a_2 : \mu s=-1\}\\
&=2\times 8+2\times 8=32.
\end{split}
\end{align}
\begin{rem}
Since \(\mathrm{tr}(C_3)=(-1)^{\mathrm{Re}(a_3)+1}\cdot 2\), in the case \( c=0 \), the condition 
\[
\mathrm{tr}(C_3)=\mathrm{tr}(C_1C_2)=(-1)^{\mathrm{Re}(a_1)+\mathrm{Re}(a_2)}\cdot 2
\]
implies that \(\sum_{i=1}^3 \mathrm{Re}(a_i)\) is odd, hence \(T=0\). Thus, we have
\begin{align*}
 &\quad \,\,\#\{\mathrm{T}^{-1}(0)\}\\
 &= \#\{(a_1,a_2) : \mu s=1, c=0\}+\#\{(a_1,a_2) : \mu s=-1, c=0, s+\mu a^2\neq 0\}\\[4pt]
 &=\#\{(a_1,a_2) : \mu s=1\}+2\,\#\{(a_1,a_2) : \mu s=-1\}\\[4pt]
 &=8+2\times 8=24.
\end{align*}
Consequently, we have
\(
\#\{\mathrm{T}^{-1}(1)\}=\#\{\mathrm{T}^{-1}(-1)\}=4.
\)
\end{rem}

\subsubsection{The case of  $g=n=1$} \label{sec=4.1.2}

In this subsection, we consider the case $g=n=1$, that is, $\Sigma=\Sigma_{1,1}$. For any boundary parabolic representation $\phi=((A,B),C)$, since $C$ is parabolic, $A,B\neq \pm I$. Note that $\mr{tr}([A,B])=\mr{tr}([B,A])$, 
by Section \ref{secn=1}, we know that one of $A,B$ is hyperbolic, and the other is either parabolic or hyperbolic. Without loss of generality, assume $A$ is hyperbolic with $A=\mathrm{diag}(\lambda,1/\lambda)$, $0<|\lambda|<1$.

If $B$ is parabolic, then by \eqref{eqn10}, we have $\mathrm{tr}([A,B])=2$. If $B$ is hyperbolic, then from \eqref{eqn7}, we know
\[
\mathrm{tr}([A,B])= 
\begin{cases}
2 & \text{if } B_{21}B_{12}=0, \\
-2 & \text{if } B_{21}B_{12}=\tfrac{4}{(\lambda - \lambda^{-1})^2} > 0.
\end{cases}
\]

Next, we compute $\#\{\mathrm{Par}(\Sigma_{1,1})\}$. If $\sigma(C)=1$, then $C\sim \Phi(1)$, which is equivalent to
\begin{equation}
\label{par-eqn6}
[A,B]=C^{-1} \sim 
\begin{pmatrix}
1 & 1 \\
0 & 1
\end{pmatrix}.
\end{equation}
Define $\Phi_{\lambda,\mu} = 
\begin{pmatrix}
\lambda & \mu \\
0 & 1/\lambda
\end{pmatrix}$. Then
\begin{equation}
\label{par-eqn7}
[\Phi_{\lambda,\mu}, \Phi_{a,b}] = 
\begin{pmatrix}
1 & (\lambda^2 - 1)ab + \lambda \mu (1 - a^2) \\
0 & 1
\end{pmatrix}.
\end{equation}

Now we analyze the number of connected components satisfying \eqref{par-eqn6}. Assume $A$ is hyperbolic, so $B$ is either hyperbolic or parabolic. By conjugation, suppose $A = \Phi_{\lambda,0}$ with $0 < |\lambda| < 1$. From \eqref{eqn7}, we know either $b=0$ or $c=0$, and not both. 

If $c=0$, i.e., $B = \Phi_{a,b}$, then from \eqref{par-eqn7} and \eqref{par-eqn6}, we get $ab < 0$. Thus, $(A,B)$ can be deformed into the set
\[
S = \left\{ (\Phi_{\lambda,\mu}, \Phi_{a,b}) : |\lambda|, |a| \in (0,1), ab \leq 0, \lambda\mu \geq 0, -ab + \lambda\mu \neq 0 \right\}.
\]

If $b=0$, then $B = \Phi_{a,c}^{\top}$, so
\[
[A,B] = [\Phi_{\lambda,0}, \Phi_{a,c}^{\top}] = [\Phi_{1/a,-c}, \Phi_{1/\lambda,0}]^{\top} = 
\begin{pmatrix}
1 & 0 \\
-\tfrac{c}{a}(1 - \lambda^{-2}) & 1
\end{pmatrix}.
\]
From \eqref{par-eqn6}, we get $ac < 0$. Thus $(A,B)$ can be deformed into $(\Phi_{\lambda,\mu}, \Phi_{a,c}^{\top})$ with $\lambda\mu \geq 0$, $ac \leq 0$, and $|a|, |\lambda| \in (0,1)$. In particular, $(A,B)$ can be deformed into $(\Phi_{\lambda,\lambda}, \Phi_{a,0}) \in S$.

If $A$ is parabolic, then $B$ is hyperbolic. Suppose $A = \Phi_{\lambda,\mu}$ with $|\lambda| = 1$. From \eqref{eqn10}, we must have $B = \Phi_{a,b}$. Then from \eqref{par-eqn7}, we compute:
\[
D_0 = (\lambda^2 - 1) ab + \lambda\mu (1 - a^2) = \lambda \mu (1 - a^2) > 0.
\]
Fixing $\lambda, \mu, a$, we deform $b$ so that $ab < 0$. Then, fixing $a,b,\mu$, we slightly perturb $\lambda$ to ensure $|\lambda|<1$ and maintain the sign of $D_0$. Fixing $\lambda,b$, we deform $a, \mu$ such that $|a|\in(0,1)$, $\lambda\mu\geq 0$, and $\lambda\mu(1-a^2)\geq 0$ throughout. Thus, $(A,B)$ can be deformed into $S$.

Note that $S$ has four connected components corresponding to the signs of $\lambda$ and $a$. Hence, 
\(
\sigma^{-1}(1) \cap \mathrm{Par}(\Sigma_{1,1})
\)
has four connected components. Similarly, 
\(
\sigma^{-1}(-1) \cap \mathrm{Par}(\Sigma_{1,1})
\)
has four connected components, since $\sigma(C^{-1}) = -\sigma(C)$.

Now consider the case $\sigma(C) = \sqrt{-1}$, so $C\sim \Phi(\sqrt{-1})$, which is equivalent to
\begin{equation}
\label{par-eqn8}
[A,B] = C^{-1} \sim 
\begin{pmatrix}
-1 & 1 \\
0 & -1
\end{pmatrix}.
\end{equation}
In this case, both $A$ and $B$ are hyperbolic. Assume $A = \mathrm{diag}(\lambda,1/\lambda)$ with $0<|\lambda|<1$, and let
\[
B = \begin{pmatrix}
a & b \\
c & d
\end{pmatrix} \in \mathrm{SL}(2,\mathbb{R}).
\]
Then $bc = 4/(\lambda - \lambda^{-1})^2>0$, and so $a\neq 0$ since $ad-bc=1$. From \eqref{eqn6}, the $(1,2)$-entry of $[A,B]$ is $(\lambda^2-1)ab$ which is positive  by \eqref{par-eqn8}, hence $ac =(ab)(bc)b^{-2}< 0$, $ad=bc+1 > 0$. For fixed $A$, the matrix $B$ has two connected components distinguished by the sign of $a$. Since $\lambda$ can also be chosen with either sign, we conclude that
\(
\sigma^{-1}(\sqrt{-1}) \cap \mathrm{Par}(\Sigma_{1,1})
\)
has four connected components. Similarly, 
\(
\sigma^{-1}(-\sqrt{-1}) \cap \mathrm{Par}(\Sigma_{1,1})
\)
also has four connected components. Therefore, we conclude:

\begin{prop}\label{par-prop1}
The space $\mathrm{Par}(\Sigma_{1,1})$ has $16$ connected components.
\end{prop}

From the above discussion, we know that any boundary parabolic representation $\phi \in \mathrm{Hom}(\pi_1(\Sigma_{1,1}), \mathrm{SL}(2,\mathbb{R}))$ can be deformed into a boundary identity representation. Such a representation lies on the boundary of $\sigma^{-1}(\pm 1) \cap \mathrm{Par}(\Sigma_{1,1})$, and hence its Toledo invariant vanishes.
On the other hand, for any $\phi \in \sigma^{-1}(\sqrt{-1}) \cap \mathrm{Par}(\Sigma_{1,1})$, which lying in the closure of the set of boundary elliptic representations $\mathrm{Ell}(\Sigma_{1,1})$ with rotation angle $\theta \in (\tfrac{\pi}{2}, \pi)$, and thus $\mathrm{T}(\phi) = 1$. Similarly, if $\phi \in \sigma^{-1}(-\sqrt{-1}) \cap \mathrm{Par}(\Sigma_{1,1})$, then $\mathrm{T}(\phi) = -1$.

In summary, we obtain the following:
\begin{cor}\label{par-cor1}
For any $\phi \in \mathrm{Par}(\Sigma_{1,1})$, the Toledo invariant is given by
\[
\mathrm{T}(\phi) = \mathrm{Im}(\sigma(\phi)).
\]
\end{cor}

\subsection{Interior hyperbolic representations}

In this section, we consider deforming a given representation to an interior hyperbolic representation with respect to a maximal dual tree decomposition. As an application, we examine the possible boundary sigma values for boundary parabolic representations with maximal Toledo invariant.

\subsubsection{Interior hyperbolic representations}
In \cite[Lemma~10.1]{Gold1}, Goldman proved that any boundary hyperbolic representation $\phi \in \mathrm{Hyp}(\Sigma)$ can be deformed to another boundary hyperbolic representation $\phi'$, such that for a maximal dual tree decomposition $\Sigma = \bigcup_{i=1}^{|\chi(\Sigma)|} M_i$ (see \cite[Section~3.8]{Gold1}), the restriction $\phi'|_{\pi_1(M_i)}$ is hyperbolic on $\partial M_i$ for each $i$. We refer to such a representation $\phi'$ as an \emph{interior hyperbolic representation}. 

In this subsection, we demonstrate that any boundary parabolic representation can similarly be continuously deformed into one that is hyperbolic on every interior pair of pants for $g\geq 1$.

Let $\chi: \mathrm{SL}(2,\mathbb{R}) \times \mathrm{SL}(2,\mathbb{R}) \to \mathbb{R}^3$ be the map
\[
\chi(X,Y) = (\mathrm{tr}(X), \mathrm{tr}(Y), \mathrm{tr}(XY)),
\]
and let $\kappa: \mathbb{R}^3 \to \mathbb{R}$ be the polynomial
\[
\kappa(x,y,z) = x^2 + y^2 + z^2 - xyz - 2.
\]
Then
\[
\kappa(\chi(X,Y)) = \mathrm{tr}([X,Y]).
\]

The following theorem will be essential in our analysis:

\begin{thm}[Goldman {\cite[Proposition 4.1, Theorem 4.3, Corollary 4.5]{Gold1}}]\label{lifting}
Let $(x,y,z) \in \mathbb{R}^3$. Then there exists $(X,Y) \in \mathrm{SL}(2,\mathbb{R}) \times \mathrm{SL}(2,\mathbb{R})$ such that $\chi(X,Y) = (x,y,z)$ if and only if either $\kappa(x,y,z) \geq 2$ or at least one of $|x|, |y|, |z|$ is greater than or equal to 2.
Let
\[
\Omega_{\mathbb{R}} := \left\{ (X,Y) \in \mathrm{SL}(2,\mathbb{R}) \times \mathrm{SL}(2,\mathbb{R}) \mid [X,Y] \neq I \right\}.
\]
Then the map
\[
\chi : \Omega_{\mathbb{R}} \to \Delta := \mathbb{R}^3 \setminus \left( [-2,2]^3 \cap \kappa^{-1}([-2,2]) \right)
\]
has the path-lifting property, and in particular, it is surjective.

Moreover, $\kappa(\chi(X,Y)) = 2$ if and only if the group generated by $X$ and $Y$ is reducible, hence contained in a Borel or abelian subgroup of $\mathrm{SL}(2,\mathbb{R})$. If $\kappa(x,y,z) \neq 2$, then $\chi^{-1}(x,y,z)$ consists of two $\mathrm{SL}(2,\mathbb{R})$-orbits, but only one $\mathrm{GL}(2,\mathbb{R})$-orbit.
\end{thm}

A \emph{maximal dual tree decomposition} of $\Sigma=\Sigma_{g,n}$ is a decomposition $\Sigma = \bigcup_{i=1}^{|\chi(\Sigma)|} M_i$ into $g$ punctured tori and $g - 2 + n$ pairs of pants, satisfying the following conditions:
\begin{enumerate}
    \item Each $M_i$ has Euler characteristic $\chi(M_i) = -1$;
    \item The graph dual to the decomposition $\{M_i\}_{i=1}^{2g-2+n}$ is a tree $T$, where two vertices are connected by an edge if and only if the corresponding $M_i$ share a common boundary component. In particular, adjacent $M_i$ meet along exactly one boundary component.
\end{enumerate}

Let $\Sigma=\Sigma_{g,n}$ be a surface of genus $g$ and $n$ boundaries. We assume that $\chi(\Sigma)\leq -1$. 
Let $\Sigma=\bigcup_{i=1}^{|\chi(\Sigma)|} M_i$ be a maximal dual tree decomposition. Let $\mc{R}_{h}(\Sigma)$ be the set of all representations $\phi\in \mr{Hom}(\pi_1(\Sigma),\mathrm{SL}(2,\mb{R}))$ such that one of boundary representation is hyperbolic.  
Let $\mathcal{R}'_{h}(\Sigma)\subset \mc{R}_{h}(\Sigma)$ be the subset such that $\phi|_{\pi_1(M_i)}$ is non-abelian for each $i$.  Hence, \(\mathcal{R}'_{h}(\Sigma)\) is open and dense in \(\mathcal{R}_{h}(\Sigma)\), since there is at least one boundary component whose holonomy is hyperbolic and can be freely varied.

The following lemmas were originally established by Goldman in \cite{Gold1} for representations into $\mathrm{PSL}(2,\mathbb{R})$. Via the  following map 
\begin{equation}\label{par-eqn}
\pi_*: \mathrm{Hom}(\pi_1(\Sigma), \mathrm{SL}(2,\mathbb{R})) \to \mathrm{Hom}(\pi_1(\Sigma), \mathrm{PSL}(2,\mathbb{R})),
\end{equation}
which is a covering map on the inverse image of a connected component, see \cite[Lemma 2.2]{Gold1}.
Hence, any path in $\mathrm{Hom}(\pi_1(\Sigma), \mathrm{PSL}(2,\mathbb{R}))$ can be lifted to a path in $\mathrm{Hom}(\pi_1(\Sigma), \mathrm{SL}(2,\mathbb{R}))$ starting from a fixed base point. Therefore, we obtain the following:

\begin{lemma}[Goldman {\cite[Lemma 9.3]{Gold1}}]\label{par-lemma5}
Let $\Sigma_{1,2}$ be a surface of genus one with two boundary components $c_1$ and $c_2$. Let $c \subset \Sigma_{1,2}$ be a separating simple closed curve that separates $c_1$ and $c_2$ from the rest of the surface. Let $P \subset \Sigma_{1,2}$ denote the pair of pants bounded by $c_1$, $c_2$, and $c$, and let $T \subset \Sigma_{1,2}$ be the torus with a boundary component $c$ (i.e., a torus with one boundary component). Suppose that 
\(
\phi \in \mathrm{Hom}(\pi_1(\Sigma_{1,2}), \mathrm{SL}(2,\mathbb{R}))
\)
is a representation such that both restrictions $\phi|_{\pi_1(P)}$ and $\phi|_{\pi_1(T)}$ are non-abelian. Then there exists a path $\{\phi_t\}_{0 \leq t \leq 1}$ in $\mathrm{Hom}(\pi_1(\Sigma_{1,2}), \mathrm{SL}(2,\mathbb{R}))$ such that:
\begin{itemize}
  \item[(i)] $\phi_0 = \phi$;
  \item[(ii)] for all $t \in [0,1]$, the elements $\phi_t(c_1)$ and $\phi_t(c_2)$ remain in the conjugacy classes of $\phi(c_1)$ and $\phi(c_2)$, respectively;
  \item[(iii)] $\phi_1(c)$ is hyperbolic.
\end{itemize}
\end{lemma}
\begin{lemma}[Goldman {\cite[Lemma 9.6]{Gold1}}]\label{par-lemma4}
Let $\phi = \phi(a,b,c_1,c_2) = (A,B,C_1,C_2) \in \mathrm{Hom}(\pi_1(\Sigma_{0,4}), \mathrm{SL}(2,\mathbb{R}))$ be a representation such that one of $A,B, C_1,C_2$ is hyperbolic, see Figure \ref{fig:decomposition-4}. Assume that the restrictions $\phi|_{\pi_1(M_i)}$, $i = 1,2$, are non-abelian, where
\[
\phi|_{\pi_1(M_1)} = (A, B, C := (AB)^{-1}) \quad \text{and} \quad \phi|_{\pi_1(M_2)} = (C_1, C_2, C^{-1}).
\]
Then $\phi$ can be deformed to an interior hyperbolic representation $\phi'$ such that the conjugacy classes of the boundary components are preserved and $\phi'(c)$ is hyperbolic, where $c = \partial M_1 \cap \partial M_2$.
\end{lemma}

\begin{lemma}\label{par-lemma6}
Let $\phi \in \mathrm{Hom}(\pi_1(\Sigma_{0,3}), \mathrm{SL}(2,\mathbb{R}))$ be a representation such that $\phi(c_3)$ is hyperbolic. Then for any path $\gamma(t)$ of hyperbolic elements with $\gamma(0) = \phi(c_3)$, there exists a  path of representations $\phi_t$ starting at $\phi$ such that $\phi_t(c_3) = \gamma(t)$ and the conjugacy classes of $\phi_t(c_i)$ remain constant for $i = 1, 2$.
\end{lemma}

\begin{proof}
Since the triple $(\mathrm{tr}(C_1), \mathrm{tr}(C_2), \mathrm{tr}(\gamma(t)))$ lies in $\Delta$ for all $t$, by Theorem~\ref{lifting}, there exists a path of representations $\phi'_t$ such that $\phi'_0 = \phi$ and
\[
\chi(\phi'_t) = (\mathrm{tr}(C_1), \mathrm{tr}(C_2), \mathrm{tr}(\gamma(t))).
\]
Since $\gamma(t)$ remains hyperbolic throughout, there exists a path $U_t \in \mathrm{SL}(2,\mathbb{R})$ with $U_0 = I$ such that
\[
\phi'_t(c_3) = U_t \gamma(t) U_t^{-1}.
\]
Moreover, since $\phi'_t(c_i) \notin \{\pm I\}$ and has constant trace, we conclude that $\phi'_t(c_i)$ remains conjugate to $\phi(c_i)$ for $i = 1, 2$.

Define
\[
\phi_t := U_t^{-1} \phi'_t U_t.
\]
Then $\phi_t$ is a path of representations starting at $\phi$, satisfying $\phi_t(c_3) = \gamma(t)$, and the conjugacy classes of $\phi_t(c_1)$ and $\phi_t(c_2)$ remain unchanged throughout the deformation. This completes the proof.
\end{proof}

We are now ready to prove the following key lemma, which generalizes \cite[Lemma 10.1]{Gold1}.

\begin{lemma}\label{par-lemma222}
Let $\Sigma = \Sigma_{g,n}$ be a surface of genus $g$ with $n \geq 1$ boundary components, and let $\Sigma = \bigcup_{i=1}^{|\chi(\Sigma)|} M_i$ be a maximal dual tree decomposition of $\Sigma$. Then any representation $\phi \in \mathcal{R}_{h}'(\Sigma)$ can be deformed in $\mathcal{R}_{h}'(\Sigma)$ to an interior hyperbolic representation $\phi' \in \mathcal{R}_{h}'(\Sigma)$, i.e., a representation for which the restriction to each component $c \in \partial M_i \setminus \partial \Sigma$ is hyperbolic, while boundary conjugacy classes are fixed.
\end{lemma}

\begin{proof}
We follow the strategy of \cite[Lemma 10.1]{Gold1}. We proceed by induction on the Euler characteristic $\chi(\Sigma)$. The base case  $\chi(\Sigma) = -2$ is established by Lemmas \ref{par-lemma5} and \ref{par-lemma4}.

After possibly reindexing the $M_i$, we assume that $M_1$ and $M_2$ are subsurfaces in the decomposition satisfying the following properties:
\begin{itemize}
  \item[(1)] $\partial M_1$ contains at least one component of $\partial \Sigma$, and the restriction of $\phi$ on this component is hyperbolic;
  \item[(2)] The union $M = M_1 \cup M_2$ is connected.
\end{itemize}
The remaining subsurfaces $\bigcup_{i > 2} M_i$ consist of at most three connected components $N_j$. Let $c_j$ denote the boundary curve $c_j = \partial M_2 \cap \partial N_j$, and let $c = \partial M_1 \cap \partial M_2$; see Figure~\ref{fig:decomposition-4}.

\begin{figure}[ht]
\centering
\includegraphics[width=0.5\textwidth]{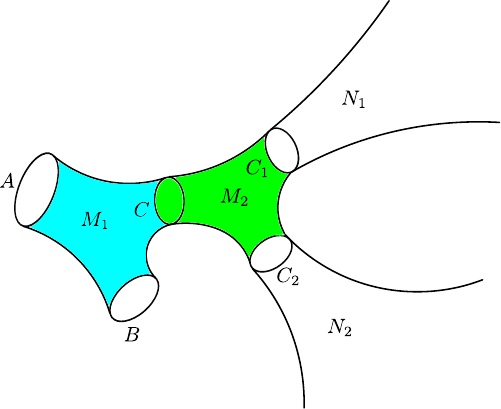}
\caption{Maximal dual tree decomposition}
\label{fig:decomposition-4}
\end{figure}

By Lemmas \ref{par-lemma4} and \ref{par-lemma5}, the restriction of $\phi$ to $\pi_1(M)$ satisfies the hypotheses, so there exists a path $\{\phi_t\}_{0 \leq t \leq 1}$ in $\mc{R}'_h(M)$ such that $\phi_1(c)$ is hyperbolic and for each $c_j$, the elements $\phi_t(c_j)$ remain conjugate to $\phi(c_j)$ for all $t$. Let $\{U^j_t\}_{0 \leq t \leq 1}$ be a path in $\mathrm{SL}(2,\mathbb{R})$ such that
\(
\phi_t(c_j) = U^j_t \phi(c_j) (U^j_t)^{-1}.
\)
We extend $\phi_t$ from $\pi_1(M)$ to $\pi_1(\Sigma)$ by defining
\(
\phi_t(\gamma) = U^j_t \phi(\gamma) (U^j_t)^{-1}
\)
for all $\gamma \in \pi_1(N_j)$. This defines a continuous path $\{\phi_t\}_{0 \leq t \leq 1}$ of representations in $\mathcal{R}'_{h}(\Sigma)$ such that $\phi_1(c)$ is hyperbolic. 

Now the restriction of $\phi_1$ to $\pi_1\left(\Sigma' \right)$ lies in $\mathcal{R}'_{h}(\Sigma')$, where $\Sigma' = N_1\cup N_2\cup M_2$ is a subsurface of $\Sigma$. By the induction hypothesis, there exists a path in $\mathcal{R}_{h}'(\Sigma')$ from this representation to one which maps each interior boundary in decomposition of $\Sigma'$  to a hyperbolic element. Furthermore, by the path-lifting property (see Lemma~\ref{par-lemma6}), this path can be extended to a path in $\mathcal{R}_{h}'(\Sigma)$.

Similarly, if there is a component \(N_3\) attached to the boundary \(B\), then, as in the cases of \(N_1\) and \(N_2\), we can also show that \(\phi\) can be deformed into an interior hyperbolic representation in \(N_3\).

 This completes the proof.
\end{proof}

\begin{lemma}\label{par-lemma2}
Let $\Sigma = \Sigma_{g,n}$ be a surface of genus $g \geq 1$ with $n \geq 1$ boundary components, and let 
\(
\Sigma = \bigcup_{i=1}^{|\chi(\Sigma)|} M_i
\)
be a maximal dual tree decomposition of $\Sigma$. Then any representation $\phi$ such that each $\phi|_{\pi_1(M_i)}$ is non-abelian can be deformed to an interior hyperbolic representation $\phi'$ while keeping the boundary conjugacy classes fixed. Moreover, the deformations is non-abelian on each $M_i$.
\end{lemma}

\begin{proof}
The argument is similar to the proof of Lemma \ref{par-lemma222}.  
Since $g \geq 1$, we choose $M_1$ to be a punctured torus, and $M_2$ to be a pair of pants.    

By Lemma~\ref{par-lemma5}, the representation $\phi$ can be deformed to $\phi_1$ such that $\phi_1(c)$ is hyperbolic. Moreover, $\phi_1$ remains non-abelian when restricted to each $M_i$.  

Now the restriction of $\phi_1$ to $\pi_1\big( \bigcup_{i>1} M_i \big)$ lies in $\mathcal{R}'_{h}(\Sigma')$, where 
\(
\Sigma' = \bigcup_{i>1} M_i
\)
is a subsurface of $\Sigma$.  
From Lemma~\ref{par-lemma222}, $\phi_1|_{\pi_1(\Sigma')}$ can be deformed to an interior hyperbolic representation.  
By the path--lifting property~\cite[Corollary~7.8]{Gold1}, these deformations can be extended to a path in $\mathrm{Hom}(\pi_1(\Sigma),\mathrm{SL}(2,\mathbb{R}))$.  
This completes the proof.
\end{proof}

\subsubsection{The maximal Toledo invariant} 
In this section, we discuss the boundary conjugacy classes of representations with maximal Toledo invariant. The case of minimal Toledo invariant is similar. By using Lemma \ref{par-lemma1}, we obtain the following result.
\begin{lemma}\label{par-lemma14}
Let $\phi=(C_1,C_2,C_3)\in \mathrm{Hom}(\pi_1(\Sigma_{0,3}),\mathrm{SL}(2,\mathbb{R}))$ be a representation such that $C_1, C_2$ are parabolic, and $C_3$ is hyperbolic with negative trace. If $\mathrm{T}(\phi)=1$, then $\sigma(C_1)=\sigma(C_2) \in \{-1,\sqrt{-1}\}$.
\end{lemma}

In the case where two of boundary representations are hyperbolic and the other one is parabolic, we have an analogous characterization for the sigma values associated to boundary parabolic representation.
\begin{lemma}\label{par-lemma13}
Let $\phi=(C_1,C_2,C_3)\in \mathrm{Hom}(\pi_1(\Sigma_{0,3}),\mathrm{SL}(2,\mathbb{R}))$ be a representation such that $C_1$ is parabolic, and $C_2, C_3$ are hyperbolic with negative trace. If $\mathrm{T}(\phi)=1$, then $\sigma(C_1)=\sqrt{-1}$.
\end{lemma}

\begin{proof}
Up to conjugation, we assume that
\begin{equation*}
C_1=\Phi(a_1),\quad C_2=\begin{pmatrix}
a & b \\
c & d
\end{pmatrix}\in \mathrm{SL}(2,\mathbb{R}),
\end{equation*}
with $a+d<-2$. Then we have
\begin{align*}
C_1C_2= (-1)^{\mathrm{Re}(a_1) + 1} \begin{pmatrix} a + s_1c & d s_1 + b \\ c &  d \end{pmatrix},
\end{align*}
and thus
\[
\mathrm{tr}(C_1C_2) = (-1)^{\mathrm{Re}(a_1) + 1}(a + d + c s_1).
\]

We now fix $\mathrm{tr}(C_2)=a+d$ and vary the parameter $c$ to deform $C_3$ continuously into a parabolic element with trace $-2$. By Lemma \ref{par-lemma14}, we have $a_1\in \{-1,\sqrt{-1}\}$.

If $a_1=-1$, then $s_1=1$. When $a+d+c=-2$, we have $c=-2-(a+d)>0$, implying $\sigma(C_1C_2)=\sqrt{-1}$. This corresponds to $a_3=\sigma(C_3)=-\sqrt{-1}\neq a_1$, and thus by Lemma \ref{par-lemma14}, $\mathrm{T}(\phi)\neq 1$. Therefore, we must have $a_1=\sqrt{-1}$.

Indeed, if $a_1=\sqrt{-1}$, then $s_1=1$. When $-(a+d+c)=-2$, we get $c=2-(a+d)>0$, yielding $\sigma(C_1C_2)=-\sqrt{-1}$. Thus, $a_3=\sigma(C_3)=\sqrt{-1}=a_1$. By the same argument as in Lemma \ref{par-lemma14}, the representation $\phi$ can be continuously deformed into a boundary parabolic representation with $a=(\sqrt{-1},\sqrt{-1},\sqrt{-1})$.
\end{proof}
\begin{prop}\label{prop1}
Let $\Sigma_{g,n}$ be a surface of genus $g\geq 1$ with $n$ boundary components, and let $\mathrm{Par}(\Sigma_{g,n})$ denote the space of boundary parabolic representations. Then
\begin{equation*}
\sigma(\mathrm{T}^{-1}(|\chi(\Sigma_{g,n})|)\cap\mathrm{Par}(\Sigma_{g,n}))= \{\mu_{i_1j_1}\circ\dots\circ \mu_{i_rj_r}(\sqrt{-1},\dots,\sqrt{-1}), 1\leq i_k,j_k\leq n\},
\end{equation*}
where $\mu_{ij}\in \mathrm{End}(\mathbb{C}^n)$ is defined by
\[
\mu_{ij}(a_1,\dots,a_n)=(a_1,\dots, -\sqrt{-1}\overline{a_i},\dots, -\sqrt{-1}\overline{a_j},\dots,a_n).
\]
\end{prop}

\begin{proof}
We first decompose $\Sigma_{g,n}$ into the connected sum of $\Sigma_{g,1}$ and $\Sigma_{0,n+1}$ as shown in Figure~\ref{fig:decomposition}. For any $\phi\in \mathrm{T}^{-1}(|\chi(\Sigma_{g,n})|)\cap\mathrm{Par}(\Sigma_{g,n})$, we have $\mathrm{T}(\phi)=|\chi(\Sigma_{g,n})|$, which is maximal. 

We first consider the case where $n \geq 2$. Since $g\geq 1$,
by Lemma \ref{par-lemma2}, we may assume that $C_0, D_1,\dots,D_{n-2}$ are all hyperbolic.

Since $\mathrm{T}(\phi)$ is maximal, the Toledo invariants of the restrictions of $\phi$ to both $\pi_1(\Sigma_{g,1})$ and each $\pi_1(P_i)$, $1\leq i\leq n-1$, are also maximal. If $\mathrm{tr}(C_0)>2$, it can be continuously deformed to the identity, reducing the Toledo invariant, contradicting maximality. Thus, $\mathrm{tr}(C_0)<-2$.

If $\mathrm{tr}(D_1)>2$, by multiplying $-I$ to the boundary representations $C_1$ and $C_2$, we obtain a new representation $\mu_{12}(\phi)$, where $\mu_{ij}$ acts by multiplying $-I$ at boundaries $i$ and $j$. This operation preserves both the Toledo invariant and preverse $\mathrm{Par}(\Sigma_{g,n})$, and transforms traces into negatives. Repeating this process, we may assume all $D_i$ have negative traces, obtaining a representation
\[
\phi'=\mu_{i_1j_1}\circ\dots\circ \mu_{i_rj_r}(\phi).
\]
By Lemmas \ref{par-lemma14} and \ref{par-lemma13}, we conclude
\[
\sigma(\phi')=(\sqrt{-1},\dots,\sqrt{-1},\sqrt{-1},\sqrt{-1}) \quad \text{or}\quad (\sqrt{-1},\dots,\sqrt{-1},-1,-1).
\]
Hence,
\[
\sigma(\phi')\in \mathcal{S}:= \{\mu_{i_1j_1}\circ\dots\circ \mu_{i_rj_r}(\sqrt{-1},\dots,\sqrt{-1})\},
\]
which implies $\sigma(\phi)\in\mathcal{S}$.

Conversely, for each  $P_i$, we can construct $\phi_i\in \mathrm{Hom}(\pi_1(P_i), \mathrm{SL}(2,\mathbb{R}))$ with maximal Toledo invariant, boundary invariants $\sigma(\phi_i(c_i))=\sigma(\phi_{n-1}(c_n))=\sqrt{-1}$, and negative hyperbolic traces elsewhere. Similarly, we construct a representation $\phi_0$ on $\pi_1(\Sigma_{g,1})$ with maximal Toledo invariant and $\mathrm{tr}(\phi_0(c_0))<-2$. After an appropriate deformation, these representations can be glued to form a representation $\phi\in \mathrm{T}^{-1}(|\chi(\Sigma_{g,n})|)\cap\mathrm{Par}(\Sigma_{g,n})$ with boundary invariant $(\sqrt{-1},\dots,\sqrt{-1})$. Thus, $(\sqrt{-1},\dots,\sqrt{-1})\in \sigma(\mathrm{T}^{-1}(|\chi(\Sigma_{g,n})|)\cap\mathrm{Par}(\Sigma_{g,n}))$. Since $\sigma$ commutes with $\mu_{ij}$ and
\[
\mu_{ij}(\mathrm{T}^{-1}(|\chi(\Sigma_{g,n})|)\cap\mathrm{Par}(\Sigma_{g,n}))= \mathrm{T}^{-1}(|\chi(\Sigma_{g,n})|)\cap\mathrm{Par}(\Sigma_{g,n}),
\]
we conclude $\mathcal{S}\subseteq \sigma(\mathrm{T}^{-1}(|\chi(\Sigma_{g,n})|)\cap\mathrm{Par}(\Sigma_{g,n}))$, completing the proof of $n\geq 2$.

Consider the case $n = 1$. If $\mathrm{T}(\phi) = |\chi(\Sigma_{g,1})| = 2g - 1$ is odd, then by the definition of the Toledo invariant, we have
\[
\mathrm{T}(\phi) \equiv \mathrm{Re}(a_1) + 1 \mod 2\mathbb{Z}.
\]
This implies that $\mathrm{Re}(a_1) = 0$, hence $a_1 = \pm \sqrt{-1}$.

For $g = 1$, the condition $\mathrm{T}(\phi) = 1$ implies that $\sigma(C_1) = \sqrt{-1}$ by Section \ref{sec=4.1.2}, where $C_1 := \phi(c_1)$.
We proceed by induction on $g$. Assume that for genus $\leq g - 1$, any boundary hyperbolic representation $\phi$ with $\mathrm{T}(\phi) = |\chi(\Sigma_{g-1,1})|$ satisfies $\sigma(\phi(c_1)) = \sqrt{-1}$.
By Lemma \ref{par-lemma2}, we may assume that both $C_2 := \phi([a_g, b_g])$ and 
\(C_3 := \prod_{i=1}^{g-1} [\phi(a_i), \phi(b_i)]\)
are hyperbolic elements. Since $\mathrm{T}(\phi)$ is maximal, the restrictions of $\phi$ to both the genus $g-1$ subsurface and the genus $1$ subsurface have maximal Toledo invariant. Hence, by the induction hypothesis, both $C_2$ and $C_3$ are hyperbolic with negative trace by Lemma \ref{lemma9}.
Now consider the triple $(C_3, C_2, C_1)$, which satisfies $C_3 C_2 C_1 = I$. Since $T((C_3, C_2, C_1)) = 1$, by Lemma~\ref{par-lemma13}, it follows that $\sigma(C_1) = \sqrt{-1}$.
This completes the proof.
\end{proof} 

\begin{cor}\label{cor1}
The number of connected components of boundary parabolic representations with maximal Toledo invariant is
\[
\#\{\mathrm{T}^{-1}(|\chi(\Sigma_{g,n})|)\cap\mathrm{Par}(\Sigma_{g,n})\} = 2^{2g+n-1}
\]
for $g\geq 1$.
\end{cor}

\begin{proof}
As illustrated in Figure~\ref{fig:decomposition}, by Lemma \ref{par-lemma2} and \ref{lemma9}, any representation $\phi\in \mathrm{T}^{-1}(|\chi(\Sigma_{g,n})|)\cap\mathrm{Par}(\Sigma_{g,n})$ can be continuously deformed into a representation such that $\phi(c_0) = C_0$ is hyperbolic with negative trace. Since the Toledo invariant of \( \phi|_{\pi_1(\Sigma_{g,1})} \) remains maximal under deformations, the boundary holonomy \( C_0 \) cannot be deformed to an elliptic element. Indeed, such a deformation would alter the rho invariant, while the signature remains constant and maximal. According to the relation \eqref{Sig-Tol}, this would force a change in the Toledo invariant, contradicting its maximality. Thus, $\phi|_{\pi_1(\Sigma_{g,1})}$ and $\gamma\phi|_{\pi_1(\Sigma_{g,1})}$ must lie in different components for any $\gamma\in \Lambda_{g,n}\backslash \{\mr{id}\}$. 

Hence, there are $4^g$ connected components of representations of $\pi_1(\Sigma_{g,1})$ with boundary hyperbolic holonomy and maximal Toledo invariant. On the other hand, since $C_0$ is hyperbolic, by Lemma \ref{par-lemma7}, the restriction $\phi|_{\pi_1(\Sigma_{0,n+1})}$ can be deformed into a boundary parabolic representation with $\sigma(C_0)=\sqrt{-1}$.

According to Proposition~\ref{prop1}, the number of such boundary-type invariants of the form
\[
\mu_{i_1j_1}\circ\cdots\circ \mu_{i_rj_r}(\sqrt{-1},\dots,\sqrt{-1}), \quad 1\leq i_k,j_k\leq n,
\]
is precisely $2^{n-1}$. Therefore, the total number of connected components of boundary parabolic representations with maximal Toledo invariant is
\[
\#\{\mathrm{T}^{-1}(|\chi(\Sigma_{g,n})|)\cap\mathrm{Par}(\Sigma_{g,n})\}=4^g \times 2^{n-1} = 2^{2g+n-1}.
\]
\end{proof}

\subsection{The case of $n$ is even and $g\geq 1$}
 Let $\Sigma_{g,n}$ be a surface of genus $g\geq 1$ with $n$ boundary components. Let $\mathrm{Hyp}(\Sigma_{g,n})$ denote the space of boundary hyperbolic representations in $\mathrm{Hom}(\pi_1(\Sigma_{g,n}), \mathrm{SL}(2,\mathbb{R}))$. For a hyperbolic element $C \in \mathrm{SL}(2, \mathbb{R})$, recall the sigma map is $\sigma(C)=\frac{1}{2}(1+\mr{sgn}(\mr{tr}(C)))$.

Now suppose $n$ is even. For any $\phi \in \mathrm{Par}(\Sigma_{g,n})$, let $\{\phi\}$ denote the connected component containing $\phi$. By Lemma \ref{par-lemma2}, we can deform $\phi$ to a representation $\phi' \in \mathrm{Par}(\Sigma_{g,n})$ such that each $\phi'(c_{2i-1}c_{2i})$ is hyperbolic, for $1 \leq i \leq n/2$. 

By Equation~\eqref{par-eqn4}, we have
\[
\mathrm{sgn}(\mathrm{tr}(\phi(c_{2i-1}c_{2i}))) = (-1)^{\mathrm{Re}(a_{2i-1} + a_{2i}) + 1} s_{2i-1} s_{2i},
\]
where $a_j = \sigma(\phi(c_j)) \in \{\pm1, \pm \sqrt{-1}\}$ and $s_j = \mathrm{Im}(a_j) - \mathrm{Re}(a_j)$.

Let $c'_{2i}$ denote a simple closed curve that separates $\{c_{2i-1}, c_{2i}\}$ from the remaining boundary generators of $\pi_1(\Sigma_{g,n})$. Let $\Sigma_{g, n/2}$ denote the subsurface of $\Sigma_{g,n}$ cut along the curves $c'_2, c'_4, \dots, c'_n$, see Figure \ref{fig:decomposition1}.
\begin{figure}[ht]
\centering
\includegraphics[width=0.6\textwidth]{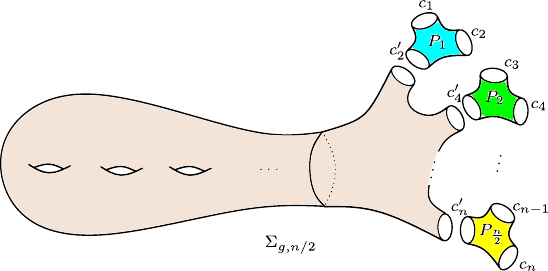}
\caption{Decomposition $\Sigma_{g,n} = \Sigma_{g,n/2} \cup \cup_{i=1}^{n/2}P_i$}
\label{fig:decomposition1}
\end{figure}

 Then the restriction $\psi := \phi'|_{\pi_1(\Sigma_{g,n/2})}$ is a representation in $\mathrm{Hyp}(\Sigma_{g,n/2})$ satisfying
\begin{equation}\label{par-eqn14}
\sigma(\psi(c'_{2i})) = \frac{1}{2} \left[ (-1)^{\mathrm{Re}(a_{2i-1} + a_{2i}) + 1} s_{2i-1} s_{2i} + 1 \right].
\end{equation}

Let $\{\psi\}$ denote the connected component of $\mathrm{Hyp}(\Sigma_{g,n/2})$ containing $\psi$. Define $\mathcal{P}_{g,n}$ and $\mathcal{H}_{g,n}$ to be the sets of connected components of $\mathrm{Par}(\Sigma_{g,n})$ and $\mathrm{Hyp}(\Sigma_{g,n})$, respectively. Denote by $\widehat{\mathcal{H}}_{g,n} := \mathcal{H}_{g,n}/\Lambda_{g,n}$ the set of connected components of boundary hyperbolic representations modulo the action of $\Lambda_{g,n}$. That is, two components $\{\psi_0\}, \{\psi_1\} \in \mathcal{H}_{g,n}$ are considered equivalent, written $\{\psi_0\} \sim \{\psi_1\}$, if there exists $\gamma \in \Lambda_{g,n}$ such that $\{\psi_1\} = \gamma \cdot \{\psi_0\}:=\{\gamma\psi_0\}$, where $\Lambda_{g, n} \subset \mathrm{Hom}\left(\pi_1\left(\Sigma_{g, n}\right),\{ \pm I\}\right)$ denote the subgroup generated by the representations $\alpha_1, \beta_1, \ldots, \alpha_g, \beta_g$, where each $\alpha_i$ (resp. $\beta_i$ ) sends $a_i$ (resp. $b_i$ ) to $-I$ and all other generators to the identity. We denote the equivalence class of $\{\psi\}$ in $\widehat{\mathcal{H}}_{g,n}$ by $\langle \psi \rangle$. Similarly, we can define $\widehat{\mc{P}}_{g,n}=\mc{P}_{g,n}/\Lambda_{g,n}$.

\begin{lemma}\label{par-lemma15}
The following map 
\[
\Phi: {\mathcal{P}}_{g,n} \to \widehat{\mathcal{H}}_{g,n/2}, \quad \{\phi\} \mapsto \left\langle \psi \right\rangle,
\]
is well-defined and surjective.
\end{lemma}

\begin{proof}
We first prove that $\Phi$ is well-defined. Suppose $\phi_0, \phi_1 \in \{\phi\}$ lie in the same connected component of $\mathrm{Par}(\Sigma_{g,n})$ and are connected by a path $\phi(t) \in \mathrm{Par}(\Sigma_{g,n})$, $t \in [0,1]$. 

By Lemma \ref{par-lemma2}, both $\phi_0$ and $\phi_1$ can be deformed to representations $\phi_0', \phi_1' \in \mathrm{Par}(\Sigma_{g,n})$ such that each of $\phi_0'(c_{2i-1}c_{2i})$ and $\phi_1'(c_{2i-1}c_{2i})$ is hyperbolic of the form described in \eqref{par-eqn14}. Let $\psi_0 := \phi_0'|_{\pi_1(\Sigma_{g,n/2})}$ and $\psi_1 := \phi_1'|_{\pi_1(\Sigma_{g,n/2})}$ be the restrictions to the subsurface. Then $\psi_0, \psi_1 \in \mathrm{Hyp}(\Sigma_{g,n/2})$ and satisfy $\sigma(\psi_0) = \sigma(\psi_1)$.

Moreover, since $\phi_0'$ and $\phi_1'$ lie in the same component $\{\phi\}$, we also have $\mathrm{T}(\phi_0') = \mathrm{T}(\phi_1')$. Now for each pair of pants $P_i = \{c_{2i-1}, c_{2i}, c_{2i}'\}$, we make the following claim:

\medskip
\noindent\textbf{Claim.} Let $\phi = (C_1, C_2, C_3) \in \mathrm{Hom}(\pi_1(P_i), \mathrm{SL}(2,\mathbb{R}))$ such that $C_1, C_2$ are parabolic and $C_3$ is hyperbolic. Then
\begin{equation}\label{par-eqn16}
\mathrm{T}(\phi) = \tfrac{1}{2}(s_1 + s_2),
\end{equation}
where $s_i := \mathrm{Im}(\sigma(C_i)) - \mathrm{Re}(\sigma(C_i))$ for $i = 1, 2$.
\begin{proof}[Proof of Claim]
Up to conjugation, we may assume that $C_1$ and $C_2$ are given by \eqref{par-eqn11}. By \eqref{par-eqn4}, the trace of $C_1C_2$ is
\begin{equation}\label{par-eqn12}
\mathrm{tr}(C_1C_2) = (-1)^{\mathrm{Re}(a_1)+\mathrm{Re}(a_2)}(2 - \mu s c^2), \quad \text{with } c > 0.
\end{equation}
The matrix product $C_1C_2$ is explicitly given by \eqref{par-eqn3}:
\begin{equation}
C_1C_2 = (-1)^{\mathrm{Re}(a_1)+\mathrm{Re}(a_2)}
\begin{pmatrix}
1 - \mu a c - \mu s c^2 & s + \mu a^2 + \mu s a c \\
- \mu c^2               & 1 + \mu a c
\end{pmatrix}.
\end{equation}

Suppose $\mathrm{Re}(a_1) + \mathrm{Re}(a_2)$ is even and $\mu = s$. Then
\[
\mathrm{tr}(C_1C_2) = 2 - c^2 < -2,
\]
which implies $c > 2$ and $C_1C_2 = C_3$ is hyperbolic with negative trace. We may deform $c$ monotonically decreasing to $c = 2$, so that $C_1C_2$ becomes parabolic with $\sigma(C_1C_2) = -\sqrt{-1} \mu$. In this case, the signature vector $a = \sigma(C_1,C_2,C_3)$ can be either $\pm (-1, -1, \sqrt{-1})$ or $\pm (\sqrt{-1}, \sqrt{-1}, \sqrt{-1})$. 
By Lemma \ref{par-lemma1} and the invariance of the Toledo invariant under deformation, we conclude $\mathrm{T}(\phi) = \pm 1$. In particular, the Toledo invariant is determined solely by $a_1$ and $a_2$.

Now consider the case $\mu = -s$. Then
\[
\mathrm{tr}(C_1C_2) = 2 + c^2 > 2,
\]
and $C_1C_2$ is hyperbolic with positive trace. We may deform $c$ decreasing to $0$, and in the limit $C_1C_2$ becomes a parabolic element with $\sigma(C_1C_2) = \pm 1$. Again, since the Toledo invariant remains unchanged throughout the deformation, by Lemma \ref{par-lemma1} we obtain $\mathrm{T}(\phi) = 0$.

Similarly, when $\mathrm{Re}(a_1) + \mathrm{Re}(a_2)$ is odd, the conclusions remain the same. When $\mu s = 1$, we again have $c > 2$, and as $c \to 2$ we obtain signature vectors $\pm (-1, \sqrt{-1}, -1)$ or $\pm (\sqrt{-1}, -1, -1)$, yielding $\mathrm{T}(\phi) = \pm 1$. When $\mu s = -1$, we get $\mathrm{T}(\phi) = 0$.
\end{proof}

\medskip
By the above claim, we conclude that for each $i$, the local Toledo invariants $\mathrm{T}(\phi_0'|_{\pi_1(P_i)}) = \mathrm{T}(\phi_1'|_{\pi_1(P_i)})$ are equal. Hence,
\[
\mathrm{T}(\psi_0) = \mathrm{T}(\phi_0') - \sum_{i=1}^{n/2} \mathrm{T}(\phi_0'|_{\pi_1(P_i)}) = \mathrm{T}(\phi_1') - \sum_{i=1}^{n/2} \mathrm{T}(\phi_1'|_{\pi_1(P_i)}) = \mathrm{T}(\psi_1).
\]
This shows that $\psi_0$ and $\psi_1$ have the same sigma invariant and the same Toledo invariant. Therefore, they lie in the same equivalence class in $\widehat{\mathcal{H}}_{g,n/2}$, i.e., $\langle \psi_0 \rangle = \langle \psi_1 \rangle$, and so $\Phi$ is well-defined.

Finally, we prove surjectivity. For any given class $\langle \psi \rangle \in \widehat{\mathcal{H}}_{g,n/2}$, one can glue a pair of pants along each boundary component of $\Sigma_{g,n/2}$ to obtain a surface $\Sigma_{g,n}$, and extend the representation by specifying a representation on each pair of pants that maps the two unglued boundaries to parabolic elements and the glued boundary to a hyperbolic element of the given type. This produces a preimage of $\langle \psi \rangle$ under $\Phi$.
\end{proof}

\begin{lemma}\label{par-lemma11}
The following map 
\[
\widehat{\Phi}: \widehat{\mathcal{P}}_{g,n} \to \widehat{\mathcal{H}}_{g,n/2}, \quad \left\langle\phi \right\rangle \mapsto \left\langle \psi \right\rangle,
\]
is well-defined and surjective. Moreover, each fiber $\widehat{\Phi}^{-1}(\left\langle\psi\right\rangle)$ consists of exactly $8^{n/2}$ elements.
\end{lemma}

\begin{proof}
For any $\{\phi\} \in \mathcal{P}_{g,n}$ and any $\gamma \in \Lambda_{g,n}$, we have $\Phi(\{\phi\}) = \Phi(\gamma\{\phi\})$. Hence, the map $\Phi$ descends to a composition
\[
\Phi = \widehat{\Phi} \circ p,
\]
where $p: \mathcal{P}_{g,n} \to \widehat{\mathcal{P}}_{g,n}$ is the natural projection, and $\widehat{\Phi}: \widehat{\mathcal{P}}_{g,n} \to \widehat{\mathcal{H}}_{g,n/2}$ is well-defined and surjective.

We now compute the cardinality of each fiber $\widehat{\Phi}^{-1}(\langle \psi \rangle)$. Consider a pair of pants $P$ and a representation $\phi = (C_1, C_2, C_3)$, where $C_1$ is a given hyperbolic element. Then there are exactly 8 distinct choices of $C_2$ and $C_3$ such that both $C_2$ and $C_3$ are parabolic and $C_1C_2C_3 = I$.

Indeed, up to conjugation, we may assume
\[
C_2 = \Phi(a_2), \quad \text{and} \quad C_1  = \begin{pmatrix} a & b \\ c & d \end{pmatrix} \in \mathrm{SL}(2, \mathbb{R}).
\]
Then
\begin{align*}
C_1C_2 
&= (-1)^{\mathrm{Re}(a_2) + 1} \begin{pmatrix} a & a s_2 + b \\ c & c s_2 + d \end{pmatrix},
\end{align*}
and thus
\begin{equation}\label{par-eqn15}
\mathrm{tr}(C_1C_2) = (-1)^{\mathrm{Re}(a_2) + 1}(a + d + c s_2),
\end{equation}
where $|a + d| > 2$. For any fixed $a_2$, we can choose
\[
c = \frac{(-1)^{\mathrm{Re}(a_2) + 1}(\pm 2) - (a + d)}{s_2} \neq 0
\]
so that $C_1C_2$ becomes parabolic. Hence, there are exactly 2 such choices for $C_3 = (C_1C_2)^{-1}$. Since there are 4 choices for $a_2 \in \{\pm 1, \pm \sqrt{-1}\}$ with fixed trace type, we obtain $4 \times 2 = 8$ possible choices for $(C_1, C_2)$ for a fixed $C_1$. Therefore, each fiber $\widehat{\Phi}^{-1}(\langle \psi \rangle)$ has at least $8^{n/2}$ elements.

Conversely, suppose $\phi_0, \phi_1 \in \mathrm{Par}(\Sigma_{g,n})$ with 
\[
\left\langle \phi_0 \right\rangle, \left\langle \phi_1 \right\rangle \in \widehat{\Phi}^{-1}(\left\langle \psi \right\rangle) \cap \sigma^{-1}(a).
\]
Then $\mathrm{T}(\phi_0) = \mathrm{T}(\phi_1)$ and $\sigma(\phi_0) = \sigma(\phi_1)$. Let $\psi_0 := \phi_0'|_{\pi_1(\Sigma_{g,n/2})}$ and $\psi_1 := \phi_1'|_{\pi_1(\Sigma_{g,n/2})}$. By assumption, there exists $\gamma \in \Lambda_{g,n}$ such that $\gamma \cdot \{\psi_0\} = \{\psi_1\}$. For representations on a pair of pants with one hyperbolic and two parabolic boundary components and with fixed conjugacy classes, the corresponding representation space is connected. Therefore, the path connecting $\gamma \cdot \psi_0$ and $\psi_1$ can be extended to a path connecting $\gamma \cdot \phi_0'$ and $\phi_1'$, which implies
\[
\gamma \cdot \{\phi_0\} = \gamma \cdot \{\phi_0'\} = \{\phi_1'\} = \{\phi_1\},
\]
and hence $\left\langle \phi_0 \right\rangle = \left\langle \phi_1 \right\rangle$. This shows that for fixed conjugacy classes on each boundary, each fiber $\widehat{\Phi}^{-1}(\left\langle \psi \right\rangle)$ contains exactly one element.

It follows that $\widehat{\Phi}^{-1}(\left\langle \psi \right\rangle)$ has precisely $8^{n/2}$ elements, completing the proof.
\end{proof}
By Lemma~\ref{par-lemma11}, we have
\begin{equation}\label{par-eqn22}
  \#\{\widehat{\mathcal{P}}_{g,n}\} = 8^{n/2} \cdot\{ \#\widehat{\mathcal{H}}_{g,n/2}\}
  = 8^{n/2} \cdot 2^{n/2 - 1}(4g + n - 3)
  = 2^{2n - 1}(4g + n - 3).
\end{equation}

\begin{lemma}\label{par-lemma3}
Let $\phi_0, \phi_1 \in \mathrm{Hom}(\pi_1(\Sigma_{0,3}), \mathrm{SL}(2,\mathbb{R}))$ be two representations such that $\phi_0(c_i)$ and $\phi_1(c_i)$ are parabolic and satisfy $\sigma(\phi_0(c_i)) = \sigma(\phi_1(c_i))$ for $i=1,2$. Suppose further that $\phi_0(c_3)$ and $\phi_1(c_3)$ are hyperbolic with the same sigma value. 
If $\gamma(t) \in \mathrm{SL}(2,\mathbb{R})$ is a hyperbolic path connecting $\phi_0(c_3)$ to $\phi_1(c_3)$,
then there exists a path  $\phi(t)$ connecting $\phi_0$ to $\phi_1$ such that $\phi(t)(c_3) = \gamma(t)$.
\end{lemma}

\begin{proof}
Up to conjugation, we may assume $\phi_0(c_1)=\Phi(a_1)$, where $a_i=\sigma(\phi_0(c_i))$, $i=1,2$. From equation~\eqref{par-eqn2}, we have
\begin{equation}\label{par-eqn1}
  \mathrm{tr}(\phi_0(c_1 c_2))=(-1)^{\mathrm{Re}(a_1)+\mathrm{Re}(a_2)}(2-s_1 s_2 P_{21}^2),
\end{equation}
where $s_i=\mathrm{Im}(a_i)-\mathrm{Re}(a_i)$, and $\phi_0(c_2)=P\Phi(a_2)P^{-1}$ with $P_{21}$ denoting the $(2,1)$-entry of $P$. Without loss of generality, we may assume $P_{21}\geq 0$, since $\phi_0(c_2)$ is invariant under the transformation $P\mapsto -P$. Similarly, assume $\phi_1(c_1)=\Phi(a_1)$, $\phi_1(c_2)=Q\Phi(a_2)Q^{-1}$ with $Q_{21}\geq 0$.

Let $\gamma(t)$ be a path connecting $\phi_0(c_3)$ to $\phi_1(c_3)$. Since $s_1 s_2=\pm 1 \neq 0$, define
\[
  F(t)=\frac{2-(-1)^{\mathrm{Re}(a_1)+\mathrm{Re}(a_2)}\mathrm{tr}(\gamma(t))}{s_1 s_2}.
\]
Then $F(0)=P_{21}^2\geq 0$ and $F(1)=Q_{21}^2\geq 0$. Since $|\mathrm{tr}(\gamma(t))| > 2$ (as $\gamma(t)$ is hyperbolic), $F(t)\neq 0$, and hence $F(t) > 0$ for all $t \in [0,1]$.

Define $c(t)=\sqrt{F(t)}$, so that $c(0)=P_{21}$ and $c(1)=Q_{21}$. Choose a path $P(t) \in \mathrm{SL}(2,\mathbb{R})$ connecting $P$ and $Q$ such that $P(t)_{21}=c(t)$. Define a path of representations by
\[
  \phi'(t)(c_1)=\Phi(a_1), \,
  \phi'(t)(c_2)=P(t)\Phi(a_2)P(t)^{-1}, \,
  \phi'(t)(c_3)=[\phi'(t)(c_1 c_2)]^{-1}.
\]

Then $\phi'(t)$ connects $\phi_0$ and $\phi_1$, and satisfies $\mathrm{tr}(\phi'(t)(c_3))=\mathrm{tr}(\gamma(t))$.

Since both $\phi_0(c_3)$ and $\phi_1(c_3)$ are hyperbolic, their traces are either both greater than $2$ or both less than $-2$. Thus, for each $t$, $\phi'(t)(c_3)$ and $\gamma(t)$ lie in the same conjugacy class. Therefore, there exists $U(t) \in \mathrm{SL}(2,\mathbb{R})$ such that
\[
\gamma(t) = U(t)\phi'(t)(c_3)U(t)^{-1}, \quad \text{with } U(0)=U(1)=I.
\]
Define the path of representations $\phi(t)$ by
\[
  \phi(t) := U(t)\phi'(t)U(t)^{-1}.
\]
Then $\phi(t)$ connects $\phi_0$ and $\phi_1$, and satisfies $\phi(t)(c_3) = \gamma(t)$, as desired.
\end{proof}

\begin{lemma}\label{par-lemma7}
Let $\phi \in \mathrm{Hom}(\pi_1(\Sigma_{g,n}), \mathrm{SL}(2,\mathbb{R}))$ be a representation such that $\phi(c_1), \dots, \phi(c_{n-1})$ are parabolic and $\phi(c_n)$ is hyperbolic, $g\geq 1$. Then there exists a deformation $\phi(t)$ of $\phi$ such that the conjugacy classes of $\phi(t)(c_1), \dots, \phi(t)(c_{n-1})$ remain fixed, the trace $\mathrm{tr}(\phi(t)(c_n))$ varies monotonically, and $\phi(1)(c_n)$ is parabolic with $\mathrm{tr}(\phi(1)(c_n))=2\mathrm{sgn}(\mathrm{tr}(\phi(c_n)))$. 
\end{lemma}

\begin{proof}
We first assume $n \geq 2$. By a small perturbation fixing the conjugacy classes of $\phi(c_1), \dots, \phi(c_{n-1})$, we may assume that $\phi(\Pi_{i=1}^g[a_i,b_i]c_1 \cdots c_{n-2})$ is either elliptic or hyperbolic. This can be done since $g\geq 1$. Up to conjugation, we may write
\[
\phi(c_{n-1}) = \Phi(a_{n-1}), \quad \phi(\Pi_{i=1}^g[a_i,b_i]c_1 \cdots c_{n-2}) = P = \begin{pmatrix} a & b \\ c & d \end{pmatrix} \in \mathrm{SL}(2, \mathbb{R}),
\]
where $a + d \neq \pm 2$. Then we have:
\[
\phi(c_n)^{-1} = \phi(\Pi_{i=1}^g[a_i,b_i]c_1 \cdots c_{n-1}) = (-1)^{\mathrm{Re}(a_{n-1}) + 1} \begin{pmatrix} a & a s_{n-1} + b \\ c & c s_{n-1} + d \end{pmatrix},
\]
and the trace becomes
\[
\mathrm{tr}(\phi(c_n)) = (-1)^{\mathrm{Re}(a_{n-1}) + 1}(a + d + c s_{n-1}).
\]
We can fix $a+d$ and deform $c$ such that $\mathrm{tr}(\phi(c_n))$ varies monotonically towards $2 \cdot \mathrm{sgn}(\mathrm{tr}(\phi(c_n)))$. In particular, if $a + d \in (-2, 2)$, then the sign of $c$ remains constant throughout. This yields a deformation $P(t)$ of $P$ fixing its conjugacy class and inducing a monotonic deformation of $\phi(c_n)$ into a parabolic element.
By a conjugacy deformation, we can then construct a family $\phi(t)$ such that the conjugacy classes of $\phi(t)(c_1), \dots, \phi(t)(c_{n-2}),  \phi(t)(c_{n-1})$ are fixed,  and $\mathrm{tr}(\phi(t)(c_n))$ varies monotonically to $\pm 2$, with $\phi(1)(c_n)$ parabolic.

Now consider the case $n = 1$. For $g = 1$, Section \ref{secn=1} shows that any hyperbolic $\phi(c_1)$ can be deformed into a parabolic element.
For $g \geq 2$, by Lemma \ref{par-lemma2}, we may assume that the product $\prod_{i=1}^{g-1} \phi([a_i,b_i]) = C_3$ and $\phi([a_g, b_g]) = C_2$ are hyperbolic. Then $C_3 C_2 C_1 = I$, where $C_1 := \phi(c_1)$.

In this case, we may assume $C_2 = \mathrm{diag}(\lambda, 1/\lambda)$ and $C_3 = \begin{pmatrix} a & b \\ c & d \end{pmatrix}$, so
\[
C_3 C_2 = \begin{pmatrix} \lambda a & \lambda^{-1} b \\ \lambda c & \lambda^{-1} d \end{pmatrix}, \quad \mathrm{tr}(C_3 C_2) = \lambda a + \lambda^{-1} d = (\lambda - \lambda^{-1})a + \lambda^{-1} \mathrm{tr}(C_3).
\]
Fixing $\mathrm{tr}(C_3)$, we may slightly perturb $c \neq 0$ and adjust $\lambda$ and $a$ so that $\mathrm{tr}(C_1) = \mathrm{tr}(C_3 C_2)$ monotonically deforms to $\pm 2$ while preserving sign. Since the conjugacy classes of $C_2$ and $C_3$ remain unchanged, this yields the desired deformation via conjugation.
This completes the proof.
\end{proof}

For any $\phi \in \mathrm{Par}(\Sigma_{g,n}) \cap \sigma^{-1}(a)$ with even $n \geq 2$, by \eqref{par-eqn16}, we have
\begin{equation}\label{par-eqn30}
  \mathrm{T}(\phi) = \mathrm{T}(\psi) + \tfrac{1}{2}s_a,
\end{equation}
where $s_a = \sum_{i=1}^n s_i \in 2\mathbb{Z}$, and $s_i = \mathrm{Im}(a_i) - \mathrm{Re}(a_i)$. Here, $\psi = \phi'|_{\pi_1(\Sigma_{g,n/2})}$ for some deformation $\phi'$ of $\phi$ in $\mathrm{Par}(\Sigma_{g,n})$ such that $\phi'(c_{2i}')$ is hyperbolic. Hence,
\begin{equation}\label{par-eqn20}
  \mathrm{T}(\phi) \in \left[-|\chi(\Sigma_{g,n/2})| + \tfrac{1}{2}s_a,\; |\chi(\Sigma_{g,n/2})| + \tfrac{1}{2}s_a\right] \cap \mathbb{Z}
\end{equation}
for any $\phi \in \mathrm{Par}(\Sigma_{g,n}) \cap \sigma^{-1}(a)$.

\begin{lemma}\label{lemma15}
Let $\phi \in \mathrm{Par}(\Sigma_{g,n})$ with even $n \geq 2$ and $g\geq 1$. Then $\phi$ and $\gamma\phi$ lie in the same connected component of $\mathrm{Par}(\Sigma_{g,n})$ for any $\gamma \in \Lambda_{g,n}$ if and only if $\mathrm{T}(\psi)$ is neither maximal nor minimal.
\end{lemma}
\begin{proof}
\textbf{($\Rightarrow$)}  
Let $a := \sigma(\phi)$ and suppose that $\phi$ and $\gamma\phi$ lie in the same connected component for any $\gamma \in \Lambda_{g,n}$.

If $\mathrm{T}(\psi) = |\chi(\Sigma_{g,n/2})|$ is maximal, then by \eqref{par-eqn20}, 
\[
\mathrm{T}(\phi) = |\chi(\Sigma_{g,n/2})| + \tfrac{1}{2}s_a,
\]
which is the maximal value in the set $\mathrm{Par}(\Sigma_{g,n}) \cap \sigma^{-1}(a)$.

If $\phi'$ can be deformed to $\gamma\phi'$ for any $\gamma \in \Lambda_{g,n}$, then in particular, taking $\gamma = \alpha_1$ implies that $\phi'(a_1)$ can be deformed to an elliptic element (since $\alpha_1(\phi'(a_1)) = -\phi'(a_1)$). Consider a representation $\psi_1 = ((A, B), C)$ of $\pi_1(\Sigma_{1,1})$ such that either $A$ or $B$ is elliptic. By equation~\eqref{eqn8}, we know that $\mathrm{tr}(C) \geq 2$, with equality if and only if $C = I$. This implies $\mathrm{T}(\psi_1) = 0$ by Table \ref{sign-table}.

Therefore, when $\phi'(a_1)$ is deformed to an elliptic element, denote the resulting representation by $\phi'_{t_0}$. Then the restriction $\phi'_{t_0}|_{\pi_1(\Sigma_{1,1})}$ has Toledo invariant zero. Let $\Sigma_{g-1,n+1}$ denote the complement of $\Sigma_{1,1}$ in $\Sigma_{g,n}$. Then, by \eqref{par-eqn20},
\begin{equation}\label{par-eqn21}
\mathrm{T}(\phi'_{t_0}|_{\pi_1(\Sigma_{g-1,n+1})}) = \mathrm{T}(\phi) - \mathrm{T}(\phi'_{t_0}|_{\pi_1(\Sigma_{1,1})}) = |\chi(\Sigma_{g,n/2})| + \tfrac{1}{2}s_a.
\end{equation}

If $\phi'_{t_0}([a_1,b_1]) = I$, then $\phi'_{t_0}|_{\pi_1(\Sigma_{g-1,n})}$ can be viewed as a representation in $\mathrm{Par}(\Sigma_{g-1,n}) \cap \sigma^{-1}(a)$. By \eqref{par-eqn20}, its maximal Toledo invariant is 
\[
|\chi(\Sigma_{g-1,n/2})| + \tfrac{1}{2}s_a<|\chi(\Sigma_{g,n/2})| + \tfrac{1}{2}s_a,
\]
which contradicts \eqref{par-eqn21}.

If $\phi'_{t_0}([a_1,b_1])$ is hyperbolic, then by Lemma~\ref{par-lemma7}, we may deform the restriction $\phi'_{t_0}|_{\pi_1(\Sigma_{g-1,n+1})}$ to a representation in $\mathrm{Par}(\Sigma_{g,n}) \cap \sigma^{-1}(a_0, a)$ for some $a_0 \in \{\pm 1, \pm \sqrt{-1}\}$, such that the Toledo invariant remains unchanged. Furthermore, by Lemma \ref{par-lemma2}, we may assume that its restriction to each $c_{2i}'$, for $1 \leq i \leq n/2$, is hyperbolic. Then its maximal Toledo invariant is
\[
|\chi(\Sigma_{g-1,n/2+1})| + \tfrac{1}{2}s_a = |\chi(\Sigma_{g,n/2})| + \tfrac{1}{2}s_a - 1,
\]
again contradicting \eqref{par-eqn21}.
A similar argument rules out the case where $\mathrm{T}(\psi) = -|\chi(\Sigma_{g,n/2})|$.

\textbf{($\Leftarrow$)}  
If $\mathrm{T}(\psi)$ is neither maximal nor minimal, then by Lemma~\ref{par-lemma8}, $\psi$ and $\gamma\psi$ are connected by a path $\psi(t) \in \mathrm{Hyp}(\Sigma_{g,n/2})$. By Lemma \ref{par-lemma3}, we can extend $\psi(t)$ to obtain a path connecting $\phi'$ and $\gamma\phi'$. Hence, $\phi$ and $\gamma\phi$ lie in the same connected component.
\end{proof}

As a corollary, we obtain that 
\begin{cor}\label{par-cor2}
The following map 
\[
{\Phi}: {\mathcal{P}}_{g,n} \to{\mathcal{H}}_{g,n/2}, \quad \left\{\phi \right\} \mapsto \left\{ \psi \right\},
\]
is well-defined and surjective. Moreover, each fiber ${\Phi}^{-1}(\left\{\psi\right\})$ consists of exactly $8^{n/2}$ elements.
\end{cor}
\begin{proof}
	For any \( \phi_0, \phi_1 \in \{\phi\} \), the proof of Lemma~\ref{par-lemma15} shows that the corresponding representations \( \psi_0 \) and \( \psi_1 \in \mathrm{Hyp}(\Sigma_{g,n/2}) \) have the same Toledo invariant and the same sigma values. 

	If \( \mathrm{T}(\psi_0) = \mathrm{T}(\psi_1) \) is neither maximal nor minimal, then by Lemma~\ref{par-lemma8}, \( \psi_0 \) and \( \psi_1 \) lie in the same connected component. 

	If \( \mathrm{T}(\psi_0) = \mathrm{T}(\psi_1) \) is maximal or minimal, then \( \psi_0 \) and \( \gamma \cdot \psi_1 \) are connected for some \( \gamma \in \Lambda_{g,n} \) by the fact that there are $4^g$ different components  with the same sigma value for $\mathrm{Hyp}(\Sigma_{g,n/2})$ coming from the action of $\Lambda_{g,n}$. 
	From the proof of Lemma~\ref{lemma15}, we know that each of the elements \( \phi'_i(a_j) \) and \( \phi'_i(b_j) \), for \( i = 1,2 \) and \( 1 \leq j \leq g \), cannot be elliptic. Since \( \phi'_0 \) and \( \phi'_1 \) are connected, the corresponding elements \( \psi(a_j) = \phi'_i(a_j) \) and \( \psi(b_j) = \phi'_i(b_j) \) must lie in the closure of the same hyperbolic region. This implies that \( \gamma = \mathrm{id} \), and hence \( \psi_0 \) and \( \psi_1 \) are connected. 

	In conclusion, we have shown that \( \{\psi_0\} = \{\psi_1\} \), so \( \Phi \) is well-defined. The proofs of surjectivity and the degree of \( \Phi \) follow in the same manner as in the proof of Lemma~\ref{par-lemma11}.
\end{proof}

Hence, the number of connected components of the space $\mathrm{Par}(\Sigma_{g,n})$ of boundary-parabolic representations, assuming $n \geq 2$ is even and $g\geq 1$, is given by 
\begin{align*}
\begin{split}
  \#\{\mathrm{Par}(\Sigma_{g,n})\} =8^{n/2}\cdot \#\{\mr{Hyp}(\Sigma_{g,n/2})\}= 2^{2g+2n}+2^{2n - 1}(4g + n - 5).
 \end{split}
\end{align*}
\begin{thm}\label{par-thm1}
The number of connected components of the space of boundary-parabolic representations with an even number of boundary components is given by
\[
\#\{\mathrm{Par}(\Sigma_{g,n})\} =
2^{2g+2n}+2^{2n - 1}(4g + n - 5)
\]
for $g\geq 1$.
\end{thm}

\subsection{The case of $n$ is odd and $g\geq 1$}

In this section, we consider the connected components of the space of boundary-parabolic representations in the case where the number of boundary components $n$ is odd.

\subsubsection{Connected components of $\mathrm{HP}(\Sigma_{g,n})$}

Let $\mathrm{HP}(\Sigma_{g,n})$ denote the space of representations $\phi \in \mathrm{Hom}(\pi_1(\Sigma_{g,n}), \mathrm{SL}(2,\mathbb{R}))$ such that $\phi(c_n)$ is parabolic and $\phi(c_i)$ is hyperbolic for $1 \leq i \leq n-1$. Here, $n \geq 1$ is not necessarily assumed to be odd.

\begin{lemma}\label{par-lemma20}
For any $a = (a_1, a_2, a_3) \in \{0,1\} \times \{0,1\} \times \{\pm 1, \pm \sqrt{-1}\}$, the subset $\sigma^{-1}(a) \subset \mathrm{HP}(\Sigma_{0,3})$ is non-empty and connected. In particular, the space $\mathrm{HP}(\Sigma_{0,3})$ has $16$ connected components.
\end{lemma}

\begin{proof}
For any $a = (a_1, a_2, a_3) \in \{0,1\} \times \{0,1\} \times \{\pm 1, \pm \sqrt{-1}\}$, define
\[
C_2 = (-1)^{a_2 + 1} \begin{pmatrix} a & b \\ c & d \end{pmatrix}, \quad 
C_3 = (-1)^{\mathrm{Re}(a_3) + 1} \begin{pmatrix} 1 & s_3 \\ 0 & 1 \end{pmatrix},
\]
where $a + d > 2$, and $s_3 = \mathrm{Im}(a_3) - \mathrm{Re}(a_3) \in \{\pm 1\}$.

Then the product becomes
\[
C_2 C_3 = (-1)^{a_2 + \mathrm{Re}(a_3)} \begin{pmatrix}
a & a s_3 + b \\
c & c s_3 + d
\end{pmatrix},
\]
and thus
\[
\mathrm{tr}(C_2 C_3) = (-1)^{a_2 + \mathrm{Re}(a_3)} (a + d + c s_3).
\]
Suppose $c\neq 0$.
If $a_1 + a_2 + \mathrm{Re}(a_3)$ is even, then $\mathrm{tr}(C_2 C_3) = (-1)^{a_1}(a + d + c s_3)$. One can choose parameters $a, d, c$ such that $a + d > 2$ and $a + d + c s_3 < -2$. Hence $C_1 = (C_2 C_3)^{-1}$ is hyperbolic and satisfies $\sigma(C_1) = a_1$. In particular, in this case, we have  
\begin{equation}\label{par-eqn33}
  cs_3 < -(a+d) - 2 < 0,
\end{equation}
which implies that the sign of $c$ is uniquely determined by $s_3$.

Similarly, when $a_1 + a_2 + \mathrm{Re}(a_3)$ is odd, we can choose $a, d, c$ such that $a + d + c s_3 > 2$ and again $C_1$ is hyperbolic with the correct sigma value.
In both cases, we conclude that $\sigma^{-1}(a)$ is non-empty. 

For any $\phi \in \sigma^{-1}(a)$, if the parameter $c = 0$, then we have
\[
\mathrm{tr}(C_2 C_3) = (-1)^{a_2 + \mathrm{Re}(a_3)} (a + d) = \mathrm{tr}(C_1),
\]
which implies that $a_1 + a_2 + \mathrm{Re}(a_3)$ is odd.
Conversely, if $a_1 + a_2 + \mathrm{Re}(a_3)$ is odd, consider the representation $\phi' = (C_1', C_2', C_3')$ defined by
\[
C_2' = (-1)^{a_2 + 1} \begin{pmatrix}
2 & 0 \\
0 & \frac{1}{2}
\end{pmatrix}, \quad C_3' = C_3, \quad C_1' = (C_2' C_3')^{-1}.
\]
Then $\phi' \in \sigma^{-1}(a)$, i.e., the signature vector is preserved.

Next, we study the connectedness of the set $\sigma^{-1}(a)$. Let $\phi = (C_1, C_2, C_3) \in \sigma^{-1}(a)$. 
If $a_1 + a_2 + \mathrm{Re}(a_3)$ is odd, then the trace condition gives $a + d + c s_3 > 2$. If $c \neq 0$, define a path of matrices $C_2(t)$ by setting
\[
a(t) = a(1 - t) + 2t, \, d(t) = d(1 - t) + \tfrac{1}{2} t, \, c(t) = c(1 - t), \, b(t) = \tfrac{a(t)d(t) - 1}{c(t)}.
\]
Then $C_2(t) = \begin{pmatrix} a(t) & b(t) \\ c(t) & d(t) \end{pmatrix}$ defines a continuous deformation from $C_2$ to $C_2(1)$, and setting $\phi(t) = (C_1(t), C_2(t), C_3)$ defines a path in $\sigma^{-1}(a)$ such that
\[
\mathrm{tr}(C_1(t)) = \mathrm{tr}(C_1(0))(1 - t) + \mathrm{tr}(C_1(1))t,
\]
which remains hyperbolic throughout. Hence, $\sigma^{-1}(a)$ is connected when $a_1+a_2+\mr{Re}(a_3)$ is odd.

Now suppose that $\sum_{i=1}^3 \mathrm{Re}(a_i)$ is even.  
Let $\phi, \phi' \in \sigma^{-1}(a)$ be two representations with corresponding parameters $c$ and $c'$.  
By~\eqref{par-eqn33}, we have $cc' > 0$.  
Similarly to the case when $\sum_{i=1}^3 \mathrm{Re}(a_i)$ is odd, we can show that $\phi$ and $\phi'$ lie in the same connected component.
This completes the proof.
\end{proof}
\begin{rem}\label{par-rem1}
By the definition of the Toledo invariant, we have $\mathrm{T}(\phi) = 0$ for any $\phi \in \sigma^{-1}(a)$ with $\sum_{i=1}^3 \mathrm{Re}(a_i)$ odd.  
When $\sum_{i=1}^3 \mathrm{Re}(a_i)$ is even, the set $\sigma^{-1}(a)$ is connected, corresponding to $\mathrm{T}^{-1}(\pm 1)$.  
In fact, for any $\phi \in \sigma^{-1}(a) \cap \mathrm{HP}(\Sigma_{0,3})$ with $\sum_{i=1}^3 \mathrm{Re}(a_i)$ even, we have $\mathrm{T}(\phi) = \pm 1$, and the sign is completely determined by $a = (a_1, a_2, a_3)$.  
Up to multiplying by $-I$ on $C_3$ and $C_2$, we may assume $a_3 = \pm 1$.  
If $a_3 = 1$, then $s_3 = -1$, and by the Milnor--Wood inequality for signature$=2\mathrm{T}+\boldsymbol{\rho}(C_3)$, we know $\mathrm{T}(\phi) \neq 1$, which implies $\mathrm{T}(\phi) = -1$.  
Similarly, if $a_3 = -1$, then $\mathrm{T}(\phi) = 1$.  
In summary, we obtain
\begin{equation*}
  \mathrm{T}(\phi) =
  \begin{cases}
    1 & \text{for } a = (0,1,-1),\ (1,0,-1),\ (0,0,\sqrt{-1}),\ (1,1,\sqrt{-1}), \\
    -1 & \text{for } a = (0,1,1),\ (1,0,1),\ (0,0,-\sqrt{-1}),\ (1,1,-\sqrt{-1}), \\
    0 & \text{for } \sum_{i=1}^3 \mathrm{Re}(a_i) \ \text{odd}.
  \end{cases}
\end{equation*}
\end{rem}
\begin{lemma}\label{lemma17}
Let $\phi = (C_1, C_2, C_3) \in \mathrm{Hom}(\pi_1(\Sigma_{0,3}), \mathrm{SL}(2, \mathbb{R}))$ be a representation such that $C_2$ is hyperbolic and $C_3$ is parabolic. Then there exists a deformation $\phi(t) = (C_1, C_2(t), C_3(t))$ of $\phi$ such that both $C_2(t)$ and $C_3(t)$ are hyperbolic for all $t > 0$.
\end{lemma}

\begin{proof}
Up to conjugation, we may assume
\[
C_1 = \begin{pmatrix} a & b \\ c & d \end{pmatrix}, \quad C_2 = \begin{pmatrix} t_0 & 0 \\ 0 & 1/t_0 \end{pmatrix}.
\]
Then the trace of the product is
\[
\mathrm{tr}(C_1 C_2) = a t_0 + d t_0^{-1} = c_0,
\]
where $c_0 = \mathrm{tr}(C_3) = \pm 2$. Consider the function $f(t) = a t + d/t$. Then $f(t_0) = c_0$, and
\[
f'(t) = a - \frac{d}{t^2}, \quad f''(t) = \frac{2d}{t^3}.
\]

If $f'(t_0) \neq 0$, then a small perturbation of $t_0$ changes the value of $f(t)$ so that $|f(t)| > 2$ for $t \neq t_0$, implying that $C_1 C_2(t)$ is hyperbolic.

If $f'(t_0) = 0$, then $a = d/t_0^2$. In this case, since $f(t_0) = c_0$, we get $d = \frac{1}{2} c_0 t_0$, and thus
\[
f''(t_0) = \frac{c_0}{t_0^2}.
\]
It follows that
\[
f(t_0 + \Delta t) = c_0 + \frac{c_0}{2 t_0^2} (\Delta t)^2 + o((\Delta t)^2),
\]
which implies that $\frac{1}{c_0} f(t_0 + \Delta t) > 1$ for sufficiently small $\Delta t > 0$. Hence, for $t$ close to but not equal to $t_0$, we have $|f(t)| > 2$, so $C_1 C_2(t)$ is hyperbolic.

Therefore, we may fix $C_1$ and perturb $t_0$ slightly to deform $C_2$ such that the resulting product $C_3(t) = (C_1 C_2(t))^{-1}$ is also hyperbolic for all $t > 0$. The proof is complete.
\end{proof}

Now we assume $(g,n) \neq (1,1)$. 
If $n \geq 2$, let $c_0$ be a simple closed curve separating $\{c_{n-1}, c_n\}$ from the rest of the boundary components. Let $P_n$ be the pair of pants with boundary components $c_{n-1}, c_n$, and $c_0$, and define $\phi_0 := \phi|_{\pi_1(P_n)}$. By Lemma \ref{lemma17}, we may fix $\phi_0(c_0)$ and perturb $\phi_0(c_{n-1})$ slightly such that $\phi_0(c_n)$ becomes hyperbolic. This yields a deformation $\phi_0(t)$ of $\phi_0$. By gluing $\phi_0(t)$ with $\phi|_{\pi_1(\Sigma_{g,n} \setminus P_n)}$ along $c_0$, we obtain a deformation $\phi(t)$ of $\phi$ such that $\phi' := \phi(1)$ is boundary hyperbolic.

If $n = 1$ and $g \geq 2$, let 
\(
C_3 := \prod_{i=1}^{g-1} [\phi(a_i), \phi(b_i)]\) and \( C_2 := [\phi(a_g), \phi(b_g)],
\)
so that $C_3 C_2 C_1 = I$. By Lemma \ref{par-lemma2}, we may assume that both $C_3$ and $C_2$ are hyperbolic. Using Lemma \ref{lemma17}, we may fix $C_3$ and perturb $C_2$ slightly so that $C_1$ becomes hyperbolic. This yields a deformation $\phi(t)$ of $\phi$ such that $\phi' := \phi(1)$ is boundary hyperbolic.

In both cases, the conjugacy classes of $\phi(c_i)$, $1 \leq i \leq n-2$, are unchanged during the deformation, $\phi(c_{n-1})$ is perturbed slightly, and $\phi(c_n)$ becomes hyperbolic with sigma value
\(|\mathrm{Re}(\sigma(\phi(c_n)))|\)
for all $t > 0$. Hence, the Toledo invariant remains constant throughout the deformation.

Let $\mathcal{HP}_{g,n}$ denote the set of connected components of $\mathrm{HP}(\Sigma_{g,n})$, and let $\{\phi\} \in \mathcal{HP}_{g,n}$ denote the connected component containing $\phi \in \mathrm{HP}(\Sigma_{g,n})$.
\begin{lemma}\label{lemma18}
For $(g,n) \neq (1,1)$, the map
\[
\Psi: \mathcal{HP}_{g,n} \to \mathcal{H}_{g,n}, \quad \{\phi\} \mapsto \{\phi'\}
\]
is well-defined. Moreover, it is a double covering when restricted to 
$\mathcal{HP}_{g,n} \cap \mathrm{T}^{-1}(k)$ with $|k| < |\chi(\Sigma_{g,n})|$,  
and is of degree one when restricted to 
$\mathcal{HP}_{g,n} \cap \mathrm{T}^{-1}(\pm|\chi(\Sigma_{g,n})|)$.
\end{lemma}

\begin{proof}
We first show that $\Psi$ is well-defined. Suppose $\phi_0, \phi_1 \in \mathrm{HP}(\Sigma_{g,n})$ lie in the same connected component, i.e., there exists a path $\phi(t)$ in $\mathrm{HP}(\Sigma_{g,n})$ connecting $\phi_0$ and $\phi_1$. By construction (see Lemma \ref{lemma17}), each $\phi(t)$ can be deformed via a small perturbation into a representation $\phi'(t) \in \mathrm{Hyp}(\Sigma_{g,n})$ such that $\phi'(0) = \phi_0'$ and $\phi'(1) = \phi_1'$. Furthermore, the Toledo invariant and boundary sigma invariants remain constant during this deformation:
\[
\mathrm{T}(\phi_0') = \mathrm{T}(\phi_0) = \mathrm{T}(\phi_1) = \mathrm{T}(\phi_1'), \quad \sigma(\phi_0') = \sigma(\phi_1').
\]

For $n \geq 2$, the deformation from $\phi(t)$ to $\phi'(t) \in \mathrm{Hyp}(\Sigma_{g,n})$ can be performed continuously for all $t \in [0,1]$ via small perturbations on the parabolic boundary component (see Lemma~\ref{lemma17}). As a result, $\phi'(t)$ is a continuous path in $\mathrm{Hyp}(\Sigma_{g,n})$ connecting $\phi_0'$ and $\phi_1'$. Hence, $\{\phi_0'\} = \{\phi_1'\}$ in $\mathcal{H}_{g,n}$, and $\Psi$ is well-defined.
For $n = 1$, we use Lemma \ref{par-lemma2} to assume that the path $\phi(t)$ restricts to hyperbolic representations on $\Pi_{i=1}^{g-1} [a_i, b_i]$ and $[a_g, b_g]$ for all $t$. Then, as in the $n \geq 2$ case, $\phi(t)$ can be perturbed slightly into $\phi'(t) \in \mathrm{Hyp}(\Sigma_{g,1})$, so $\phi_0'$ and $\phi_1'$ lie in the same component. Therefore, $\Psi$ is also well-defined for $n=1$.

We now prove that $\Psi$ is a covering map when restricted to $\mathrm{T}^{-1}(k)$ and determine its degree. Consider a pair of pants $(C_1, C_2, C_3)$ with $C_1$, $C_2$ and $C_3$ hyperbolic. By  \eqref{hyp-eqn4}, there is a deformation of $C_3$ to a parabolic element, keeping the conjugacy class of $C_1$ and the sigma values of $C_2$ fixed. During this deformation, the trace of $C_3$ maintains its sign, and the Toledo invariant remains constant. In particular, by Lemma \ref{par-lemma10}, if $\sum_{i=1}^3 \mathrm{Re}\big(\sigma(C_i)\big)$ is odd (i.e., the Toledo invariant is zero), then the conjugacy class of the parabolic element can be $\pm \sqrt{\mathrm{sgn}(\mathrm{tr}(C_3))}$.  
If $\sum_{i=1}^3 \mathrm{Re}\big(\sigma(C_i)\big)$ is even (i.e., the Toledo invariant is $\pm 1$), then the conjugacy class of the parabolic element is unique (see the proof of Lemma~\ref{par-lemma20} or Remark \ref{par-rem1}).

For any $\phi' \in \mathrm{Hyp}(\Sigma_{g,n})$ with $n \geq 2$ and $g \geq 1$,  
if $|\mathrm{T}(\phi')| < |\chi(\Sigma_{g,n})|$, then by Lemma~\ref{hyp-lemma12}  
we can deform $\phi'$ to a new representation $\phi'' \in \mathrm{Hyp}(\Sigma_{g,n})$  
which is still interior hyperbolic and satisfies that  
$\phi''|_{\pi_1(P_n)} $ has Toledo invariant zero, 
where $P_n$ is a pair of pants containing $c_n$.  
From the above discussion for pairs of pants,  
there exist at least two representations $\phi_0, \phi_1 \in \mathrm{HP}(\Sigma_{g,n})$  
with different $\sigma$--values on $c_n$ such that  
\(
\Psi(\{\phi_0\}) = \Psi(\{\phi_1\}) = \{\phi''\} = \{\phi'\}.
\)
If $|\mathrm{T}(\phi')| = \pm |\chi(\Sigma_{g,n})|$,  
then, by the above discussion, there exists a unique representation  
$\phi_0 \in \mathrm{HP}(\Sigma_{g,n})$ such that  
\(
\Psi(\{\phi_0\}) = \{\phi'\}.
\)

Suppose $\phi_0, \phi_1 \in \mathrm{HP}(\Sigma_{g,n}) \cap \sigma^{-1}(a) \cap \Psi^{-1}(\{\phi'\})$, then $\phi_0'$ and $\phi_1'$ are connected by a path $\phi'(t) \in \mathrm{Hyp}(\Sigma_{g,n})$. Reversing the deformation from $\phi'(t)(c_n)$ back to a parabolic element (as in Lemma \ref{lemma17}), we obtain a path $\phi(t)$ in $\mathrm{HP}(\Sigma_{g,n}) \cap \sigma^{-1}(a)$ connecting $\phi_0$ and $\phi_1$. Thus, the fiber $\Psi^{-1}(\{\phi'\}) \cap \sigma^{-1}(a)$ is connected, and so has at most one connected component. If $|\mathrm{T}(\phi')| < |\chi(\Sigma_{g,n})|$, then the $\sigma$--values on $c_n$  
distinguish $\phi_0$ and $\phi_1$, and the entire fiber consists of exactly two elements.  
If $|\mathrm{T}(\phi')| = |\chi(\Sigma_{g,n})|$, then the $\sigma$--values of $\phi_0$ and $\phi_1$ coincide,  
and the entire fiber consists of exactly one element.

The argument for the case $n = 1$ is analogous: the deformation occurs on the boundary curve $c_1$, while the part corresponding to $\Pi_{i=1}^{g-1}[a_i, b_i]$ remains unchanged, and only a slight perturbation is performed on $[a_g, b_g]$, which remains hyperbolic throughout the deformation. This allows the construction to proceed in the same way.
\end{proof}

From Lemma \ref{lemma18}, we conclude that
\begin{align}
\begin{split}
& \quad  \#\{\mathrm{HP}(\Sigma_{g,n})\}\\
&=2\cdot\#\{\mathrm{Hyp}(\Sigma_{g,n})\cap \mr{T}^{-1}(k), |k|<|\chi(\Sigma_{g,n})|\}\\
&\quad +\{\mathrm{Hyp}(\Sigma_{g,n})\cap \mr{T}^{-1}(\pm|\chi(\Sigma)|)\}\\
&  =2^n(4g+2n-5)+4^g\cdot 2^{n}\\
&=2^n(4^g+4g+2n-5).
  \end{split}
\end{align}
for any $(g,n)\neq (1,1)$ and $g\geq 1$.

\subsubsection{Connected components of $\mr{Par}(\Sigma_{g,n})$}

We now compute the number of connected components of $\mr{Par}(\Sigma_{g,n})$ for odd $n$.

For any $\phi \in \mr{Par}(\Sigma_{g,n})$, by Lemma \ref{par-lemma2}, $\phi$ can be deformed to $\phi' \in \mr{Par}(\Sigma_{g,n})$ such that
\[
\psi := \phi'|_{\pi_1(\Sigma_{g,(n+1)/2})} \in \mathrm{HP}(\Sigma_{g,(n+1)/2}),
\]
where $\Sigma_{g,(n+1)/2}$ is the subsurface with boundary components $c_2', c_4', \dots, c_{n-1}', c_n$.

Similar to the proof of Lemma~\ref{lemma15}, we obtain:
\begin{lemma}
Let $n \geq 1$ be odd and $(g,n) \neq (1,1)$, $g\geq 1$. For any $\phi \in \mathrm{Par}(\Sigma_{g,n})$, $\phi$ and $\gamma\phi$ lie in the same connected component of $\mathrm{Par}(\Sigma_{g,n})$ for any $\gamma \in \Lambda_{g,n}$ if and only if $\mathrm{T}(\psi)$ is neither maximal nor minimal.
\end{lemma}
As a corollary, similar to the proof of Corollary \ref{par-cor2}, we obtain that 
\begin{cor}
The following map 
\[
{\Phi}: {\mathcal{P}}_{g,n} \to{\mathcal{H}}\mc{P}_{g,(n+1)/2}, \quad \left\{\phi \right\} \mapsto \left\{ \psi \right\},
\]
is well-defined and surjective. Moreover, each fiber ${\Phi}^{-1}(\left\{\psi\right\})$ consists of exactly $8^{(n-1)/2}$ elements.
\end{cor}

Hence, the number of connected components of the space $\mathrm{Par}(\Sigma_{g,n})$ of boundary-parabolic representations for odd $n \geq 1$ and $(g,n) \neq (1,1)$ is given by 
\begin{align}\label{par-eqn34}
\begin{split}
  \#\{\mathrm{Par}(\Sigma_{g,n})\} &=8^{(n-1)/2}\cdot \#\{\mr{HP}(\Sigma_{g,(n+1)/2})\}\\
  &= 2^{2g+2n-1}+2^{2n - 1}(4g + n - 4).
 \end{split}
\end{align}
Combining with Proposition \ref{par-prop1}, we obtain
\begin{thm}\label{par-thm2}
The number of connected components of the space $\mr{Par}(\Sigma_{g,n})$ with $n$ odd and $g\geq 1$ is given by
\[
\#\{ \mr{Par}(\Sigma_{g,n}) \} =
\begin{cases}
16 & \text{if } (g,n) = (1,1), \\
2^{2g+2n-1}+2^{2n-1} (4g+n-4)& \text{otherwise}.
\end{cases}
\]
\end{thm}

\subsection{The case of $g=0$}In this section, we focus on the case \( g = 0 \).  
For any representation \( \phi \in \mathrm{Par}(\Sigma_{0,n}) \), we take a maximal dual tree decomposition  
$\Sigma = \bigcup_{i=1}^{|\chi(\Sigma)|} M_i$
such that $c_{2i-1},c_{2i}$ lie in the same subsurface $M_i$. 
If \( \phi \) can be deformed to a representation \( \phi' \) such that the restriction of \( \phi' \) to the boundary of some \( M_i \) is hyperbolic, then by Lemma~\ref{par-lemma222}, \( \phi' \) can be further deformed into an interior hyperbolic representation.  

We therefore divide \( \mathrm{Par}(\Sigma_{0,n}) \) into two disjoint classes.  
The first class, denoted by \(\mathcal{P}_{\mathrm{I}}\), consists of representations that can be deformed to an interior hyperbolic representation.  
The second class, denoted by \(\mathcal{P}_{\mathrm{II}}\), consists of representations that cannot be deformed to an interior hyperbolic representation.  
Clearly, \(\mathcal{P}_{\mathrm{I}} \cap \mathcal{P}_{\mathrm{II}} = \emptyset\) and \(\mathcal{P}_{\mathrm{I}} \cup \mathcal{P}_{\mathrm{II}} = \mathrm{Par}(\Sigma_{0,n})\). Note that the representations in $\mathcal{P}_{\mathrm{II}}$ were studied in~\cite{DT0}, where they are referred to as \emph{super--maximal representations}.

For \(n = 2\), it is immediate that \(\#\{\mathrm{Par}(\Sigma_{0,n})\} = 4\).  
For \(n = 3\), by~\eqref{par-eqn13}, we have \(\#\{\mathrm{Par}(\Sigma_{0,n})\} = 32\).  
From now on, we assume \(n \geq 4\).  
We first consider the class \(\mathcal{P}_{\mathrm{I}}\).  
By~\eqref{par-eqn22}, and noting that \(\widehat{\mathcal{P}}_{0,n}\) coincides with the set of connected components of \(\mathcal{P}_{\mathrm{I}}\), we obtain  
\begin{equation}\label{par-eqn23}
\#\{\mathcal{P}_{\mathrm{I}}\} = 2^{2n-1}(n-3)
\end{equation}
for even \(n\). Similarly, for odd \(n\), by \eqref{par-eqn34}, we also obtain \eqref{par-eqn23}.

We now compute the number of connected components of $\mathcal{P}_{\mathrm{II}}$.  

\begin{lemma}\label{par-lemma16}
For the surface $\Sigma = \Sigma_{0,4}$, $\phi \in \mathcal{P}_{\mathrm{II}}$ if and only if the two vectors $a = (a_1,a_2,a_3,a_4)$ and $s = (s_1,s_2,s_3,s_4)$ satisfy either:
\begin{itemize}
    \item[(i)]  The rho invariant $\boldsymbol{\rho}(\phi)=\mathrm{Re}( \sum_{i=1}^4 a_i)$ is odd and all $s_i$ are equal; or
    \item[(ii)] The rho invariant $\boldsymbol{\rho}(\phi)=\mathrm{Re}( \sum_{i=1}^4 a_i )$ is even and exactly three of the $s_i$ are equal,
\end{itemize}
where $a = \sigma(\phi)$ is the vector of $\sigma$--values and $s_i = \mathrm{Im}(a_i) - \mathrm{Re}(a_i)$.  
Moreover, for such $a$, the set $\sigma^{-1}(a) \subset \mathcal{P}_{\mathrm{II}}$ is connected.
\end{lemma}

\begin{proof}
We assume that $\Sigma=M_1\cup M_2$ with $c=\partial M_1\cap \partial M_2$.
Assume $\partial M_1 = c_1 \cup c_2$ and $\partial M_2 = c_3 \cup c_4$.  
If $\phi \in \mathcal{P}_{\mathrm{II}}$, then $|\mathrm{tr}(\phi(c_1c_2))| \leq 2$.  
As before $\phi(c_1)=C_1=\Phi(a_1)$, $\phi(c_2)=C_2=P_{M_1}\Phi(a_2)P_{M_1}^{-1}$ and $\phi(c_3)=C_3=Q\Phi(a_3)Q^{-1}$, $C_4=QP_{M_2}\Phi(a_4)P_{M_2}^{-1}Q^{-1}$
 with $c_{M_1},c_{M_2}\geq 0$ so that
$$\mathrm{tr}(C_1C_2)=(-1)^{\mathrm{Re}(a_1+a_2)}(2-s_1s_2 c^2_{M_1}),$$
$$\mathrm{tr}(C_3C_4)=(-1)^{\mathrm{Re}(a_3+a_4)}(2-s_3s_4 c^2_{M_2}),$$
where $c_{M_1}$ (resp. $c_{M_2}$) is the $(2,1)$-entry of $P_{M_1}$ (resp. $P_{M_2}$).
From~\eqref{par-eqn4} and~\eqref{par-eqn3}, the condition $|\mathrm{tr}(\phi(c_1c_2))| < 2$ implies that $s_1s_2=1, s_3s_4=1$.  Since $\phi(c_1c_2)=\phi(c_3c_4)^{-1}$, the $(2,1)$-entries of them have opposite sign. Hence
$$(-1)^{\mathrm{Re}(a_1+a_2+a_3+a_4)}s_1=-s_3.$$ This implies that
\[
s_1 = s_2 = -s_3 = -s_4, \quad \mathrm{Re}(a_1 + a_2 + a_3 + a_4) \ \text{is even},
\]
or  
\[
s_1 = s_2 = s_3 = s_4, \quad \mathrm{Re}(a_1 + a_2 + a_3 + a_4) \ \text{is odd},
\]
where $a = (a_1,a_2,a_3,a_4) = \sigma(\phi)$.  

Note that when $\mathrm{Re}(\sum_{i=1}^4 a_i)$ is even, $\mathrm{tr}(\phi(c_1c_2))=\mathrm{tr}(\phi(c_3c_4))$ implies that by making $c_{M_1},c_{M_2}$ large, $\phi$ can be deformed into an interior hyperbolic representation.  
Hence, for $\phi \in \mathcal{P}_{\mathrm{II}}$, we must have $\mathrm{Re}(\sum_{i=1}^4 a_i)$ odd and all $s_i$ equal.  
Since $a_i \in \{\pm 1, \pm \sqrt{-1}\}$, this is equivalent to $a$ having exactly one entry $\mp 1$ and the other three $\pm \sqrt{-1}$, or exactly one entry $\pm \sqrt{-1}$ and the other three $\mp 1$.  As a set, it is $\pm\{-1,\sqrt{-1},\sqrt{-1},\sqrt{-1}\}$ and $\pm\{1,1,1,-\sqrt{-1}\}$.
It is straightforward to verify that for such $a$, the set $\sigma^{-1}(a)$ is connected.

If $\phi(c)$ cannot be deformed into an elliptic element, then it remains parabolic under deformation.  
In this case, \eqref{par-eqn4} and~\eqref{par-eqn3} yield
\begin{equation}\label{par-eqn25}
s_1 = \pm s_2, \quad s_3 = \mp s_4, \quad \mathrm{Re}(\sum_{i=1}^4 a_i) \ \text{is even}.
\end{equation}
The condition $\mathrm{Re}(\sum_{i=1}^4 a_i)$ even, together with $c_{M_1}=c_{M_2}=0$, ensures that the traces of $\phi(c_1c_2)$ and $\phi(c_3c_4)$ coincide.  
Moreover, $s_1 = s_2$ and $s_3 = -s_4$ guarantee that $\phi(c_1c_2)$ can only be deformed to an elliptic element, while $\phi(c_3c_4)$ can only be deformed to a hyperbolic one, so both must be parabolic.  
Equation~\eqref{par-eqn25} is equivalent to exactly three of the $s_i$ being equal and $\mathrm{Re}(\sum_{i=1}^4 a_i)$ being even. As a collection of sets, it is $\pm\{\sqrt{-1},\sqrt{-1},\sqrt{-1},-\sqrt{-1}\},\pm\{1,1,1,-1\},\\\pm\{1,1,-\sqrt{-1},\sqrt{-1}\},\pm\{1,-1,-\sqrt{-1},-\sqrt{-1}\}$. Similarly, one can check that the set $\sigma^{-1}(a)$ is connected.

Conversely, if $a$ and $s$ satisfy the above conditions, then by~\eqref{par-eqn4} and~\eqref{par-eqn3}, the trace of $\phi(c_1c_2)$ always lies in $[-2,2]$, implying that $\phi \in \mathcal{P}_{\mathrm{II}}$.  
\end{proof}

As a direct consequence, we obtain that the number of connected components of $\mathcal{P}_{\mathrm{II}}$ is 
\begin{equation*}
  \#\{\mathcal{P}_{\mathrm{II}}\} = 2 \times 2^3 + 2 \times 4 \times 2^3 = 80.
\end{equation*}

For a general surface \(\Sigma_{0,n}\), an analogue of Lemma~\ref{par-lemma16} also holds, that is, $\phi\in\mc{P}_{\mr{II}}$ if and only if either $\mathrm{Re}\!\left( \sum_{i=1}^n a_i \right)+n$ is odd and all $s_i$ are equal; or $\mathrm{Re}\!\left( \sum_{i=1}^n a_i \right)+n$ is even and exactly $n-1$ of the $s_i$ are equal. For simplicity, we directly invoke the result of Ryu and Yang~\cite[Theorem~1.2, (2)\text{ and }(3)]{RY} for counting the number of connected components of $\mc{P}_{\mr{II}}$.

Recall that the projection
\begin{equation*}
\pi_* \colon \mathrm{Hom}(\pi_1(\Sigma_{0,n}), \mathrm{SL}(2,\mathbb{R}))
\longrightarrow \mathrm{Hom}(\pi_1(\Sigma_{0,n}), \mathrm{PSL}(2,\mathbb{R}))
\end{equation*}
is defined by composing with the double covering map $\pi:\mathrm{SL}(2,\mathbb{R}) \to \mathrm{PSL}(2,\mathbb{R})$.  
Restricting $\pi_*$ to $\mathcal{P}_{\mathrm{II}}$ gives
\begin{equation}\label{par-eqn-pi}
\pi_*|_{\mathcal{P}_{\mathrm{II}}} \colon \mathcal{P}_{\mathrm{II}} \to \pi_*(\mathcal{P}_{\mathrm{II}}),
\end{equation}
which is a covering map of degree $2^{n-1}$, since $\pi_1(\Sigma_{0,n})$ has $n-1$ free generators, each contributing two lifts.

The image $\pi_*(\mathcal{P}_{\mathrm{II}}) \subset \mathrm{Hom}(\pi_1(\Sigma_{0,n}), \mathrm{PSL}(2,\mathbb{R}))$ has been studied in~\cite{DT0}, where it is referred to as the space of \emph{super-maximal} representations.  
In~\cite[Theorem~1.2, (2) and (3)]{RY}, Ryu and Yang proved that the relative Euler class for a representation in $\pi_*(\mathcal{P}_{\mathrm{II}})$ takes values $\pm 1$ or $0$.  
Moreover, the fiber $e^{-1}(0)$ consists of $2n$ connected components, while $e^{-1}(\pm 1)$ each consists of a single component.  
Therefore,
\[
\#\{\pi_*(\mathcal{P}_{\mathrm{II}})\} = 2n + 2.
\]
Since $\pi_*|_{\mathcal{P}_{\mathrm{II}}}$ has degree $2^{n-1}$, we conclude that
\[
\#\{\mathcal{P}_{\mathrm{II}}\} = 2^{n-1} (2n + 2) = 2^{n} (n + 1).
\]

Combining this with~\eqref{par-eqn23}, we obtain the total number of connected components of $\mathrm{Par}(\Sigma_{0,n})$:
\[
\#\{\mathrm{Par}(\Sigma_{0,n})\}
= 2^{2n-1} \big(n - 3 \big) + 2^n (n + 1).
\]

Combining with Theorem~\ref{par-thm1} and \ref{par-thm2}, we conclude
\begin{thm}\label{par-thm3}
The number of connected components of $\mr{Par}(\Sigma_{g,n})$ is given by
\[
\#\{ \mr{Par}(\Sigma_{g,n}) \} =
\begin{cases}
16 & \text{for } g=n=1, \\
2^{2g+2n}+2^{2n-1} (4g+n-5) & \text{for } g\geq 2 \text{ is even}\\
2^{2g+2n-1}+2^{2n-1} (4g+n-4) & \text{for } g\geq 1 \text{ is odd}, g+n> 2\\
 2^{2n-1} \big(n - 3 \big) + 2^n (n + 1) &\text{for } g=0.
\end{cases}
\]

\end{thm}

\subsection{The case of $\mathrm{PSL}(2,\mathbb{R})$}

We denote by 
\[
\mathrm{Par}_P(\Sigma_{g,n}) := \left\{ \phi \in \mathrm{Hom}(\pi_1(\Sigma_{g,n}), \mathrm{PSL}(2,\mathbb{R})) \mid \text{each } \phi(c_i) \text{ is parabolic} \right\}
\]
the space of boundary-parabolic representations into $\mathrm{PSL}(2,\mathbb{R})$. In this section, we compute the number of connected components of $\mathrm{Par}_P(\Sigma_{g,n})$. 

Note that $\mathrm{PSL}(2,\mathbb{R}) = \mathrm{SL}(2,\mathbb{R})/\{\pm I\}$, so the set of parabolic elements splits into two conjugacy classes. We choose the representative classes with positive trace, which correspond to sigma values $\pm 1$. 

\begin{thm}\label{par-thm4}
The number of connected components of the space of boundary-parabolic representations into $\mathrm{PSL}(2,\mathbb{R})$ is 
\begin{equation}\label{par-eqn33}
\#\{ \mr{Par}_P(\Sigma_{g,n}) \} =
\begin{cases}
2^{n} (4g+n-3) &  \text{for } g\geq 1 \\
2^{n} \big(n - 3 \big) + 2 (n + 1) &\text{for } g=0.
\end{cases}
\end{equation}
\end{thm}

\begin{proof}
For $(g,n) = (0,2)$, the number of connected components is precisely the number of parabolic conjugacy classes in $\mathrm{PSL}(2,\mathbb{R})$, so $\#\{\mathrm{Par}_P(\Sigma_{0,2})\} = 2$.

For $(g,n) = (1,1)$, we have $\#\{\mathrm{Par}(\Sigma_{1,1})\} = 16$ in $\mathrm{SL}(2,\mathbb{R})$, and each sigma value in $\{\pm1, \pm\sqrt{-1}\}$ corresponds to 4 connected components. These are permuted by the action of $\Lambda_{1,1}$. By Corollary~\ref{par-cor1}, the Toledo invariant is $\mathrm{T}(\phi) = \mathrm{Im}(\sigma(\phi))$. When passing to $\mathrm{PSL}(2,\mathbb{R})$, we identify $\pm I$, and only the sigma values $\pm 1$ remain for parabolic elements. Thus, there are exactly 4 components in $\mathrm{Par}_P(\Sigma_{1,1})$, corresponding to the subsets:
\[
\sigma^{-1}(1)\cap \mathrm{T}^{-1}(0),\,
\sigma^{-1}(-1)\cap \mathrm{T}^{-1}(0),\,
\sigma^{-1}(-1)\cap \mathrm{T}^{-1}(1),\,
\sigma^{-1}(-1)\cap \mathrm{T}^{-1}(-1).
\]

Now consider the case $g+n > 2$ with $g \geq 1$.  
When $n$ is even, by the proof of Lemma~\ref{par-lemma11}, for any hyperbolic element $C_3 \in \mathrm{PSL}(2,\mathbb{R})$ there exist four distinct types of compatible parabolic elements $C_1, C_2 \in \mathrm{PSL}(2,\mathbb{R})$.  
In contrast to the $\mathrm{SL}(2,\mathbb{R})$ case, there is no $4^g$ factor here, since the action of $\Lambda_{g,n}$ is connected in the $\mathrm{PSL}(2,\mathbb{R})$ setting.  
Therefore, the number of connected components of $\mathrm{Par}_P(\Sigma_{g,n})$ is
\(2^n (4g + n - 3)\).
When $n$ is odd, a similar argument shows that the number of connected components is also
\(2^n (4g + n - 3)\).

Finally, for the case $g = 0$,  note that the map $\pi_*$ (see \eqref{par-eqn-pi})
is a covering map of degree \(2^{n-1}\), since 
\(\pi: \mathrm{SL}(2,\mathbb{R}) \to \mathrm{PSL}(2,\mathbb{R})\) 
is a double covering and \(\pi_1(\Sigma_{0,n})\) has \(n-1\) free generators.  
Therefore, the result for \(g = 0\) follows directly from Theorem~\ref{par-thm3}.
This completes the proof.
\end{proof}

Next, using formula~\eqref{Sig-Tol}, we give an alternative proof of a generalized Milnor--Wood inequality, originally established in~\cite[Theorem~1.1 and 1.2]{RY}.

Let $\Sigma = \Sigma_{g,n}$ be a surface $\chi(\Sigma)\leq -1$ and $\phi \in \mathrm{Hom}(\pi_1(\Sigma), \mathrm{PSL}(2,\mathbb{R}))$.  
In $\mathrm{PSL}(2,\mathbb{R})$, there are only two conjugacy classes of parabolic elements, indexed by $s = \pm 1$, namely
\[
  \pm \begin{pmatrix} 1 & s \\ 0 & 1 \end{pmatrix}.
\]
Let $s_+$ (resp. $s_-$) denote the number of boundary components whose image under~$\phi$ is parabolic with $s=1$ (resp. $s=-1$), so that $n = s_+ + s_-$.  
Since $n \geq 1$, there exists a lift $\tilde{\phi} \in \mathrm{Hom}(\pi_1(\Sigma), \mathrm{SL}(2,\mathbb{R}))$ such that $\pi \circ \tilde{\phi} = \phi$, where $\pi: \mathrm{SL}(2,\mathbb{R}) \to \mathrm{PSL}(2,\mathbb{R})$ is the natural projection.  
The Toledo invariant is preserved under lifting:
\[
  \mathrm{T}(\tilde{\phi}) =  \mathrm{T}(\phi).
\]

Let $s_+^+$ (resp. $s_+^-$) be the number of boundary components of $\tilde{\phi}$ whose conjugacy class is
\[
\Phi(-1)=  +\begin{pmatrix} 1 & 1 \\ 0 & 1 \end{pmatrix}
  \quad \text{(resp. } \Phi(\sqrt{-1})=-\begin{pmatrix} 1 & 1 \\ 0 & 1 \end{pmatrix}\text{)},
\]
and similarly define $s_-^+$ and $s_-^-$, see \eqref{par-eqn-Phi} for the definition $\Phi$. 
Then
\[
  s_+ = s_+^+ + s_+^-, \quad s_- = s_-^+ + s_-^-.
\]
From~\eqref{Sig-Tol}, we have
\begin{equation}\label{par-eqn28}
  \mathrm{sign}(\tilde{\phi}) 
  = 2  \mathrm{T}(\phi) + \boldsymbol{\rho}(\tilde{\phi}) 
  = 2  \mathrm{T}(\phi) + \big(-s_+^+ + s^+_-\big),
\end{equation}
since the rho invariant equals $-s$ when $\mathrm{tr} = 2$ and is zero when $\mathrm{tr} = -2$. By~\cite[Proposition~8.9]{KPW}, for $g\geq 1$, we have
\begin{equation}\label{par-eqn29}
  |\mathrm{sign}(\tilde{\phi}) |\leq 2|\chi(\Sigma)| - s_+^+ - s^+_-.
\end{equation}
and for $g=0$, we have 
\begin{equation}\label{par-eqn33}
  |\mathrm{sign}(\tilde{\phi}) |\leq 2|\chi(\Sigma)| - s_+^+ - s^+_-+2\dim \mr{H}^0(\Sigma,\mc{E}).
\end{equation}
In particular, if $\tilde{\phi}$ can be deformed to an interior hyperbolic representation, then $\dim \mr{H}^0(\Sigma,\mc{E})$ can be zero. 

We first consider the case where $g\geq 1$, then \eqref{par-eqn29} holds.
Combining~\eqref{par-eqn28} and~\eqref{par-eqn29} yields
\[
   \mathrm{T}(\phi) \leq |\chi(\Sigma)| - s_-^+ 
  = (|\chi(\Sigma)| - s_-) + s_-^-
\]
and
\begin{equation*}
  \mathrm{T}(\phi)\geq \chi(\Sigma)+s^+_+=(\chi(\Sigma)+s_+)-s_+^-.
\end{equation*}

If $s_-^-$ is even, we may multiply by  $-I$ on each of these boundary components to obtain a new representation with $s_-^- = 0$.  

If $s_-^-$ is odd and $s_+ \geq 1$, we can also multiply by $-I$ on each of the $s_-^-$ boundary components whose conjugacy class is $\Phi(-\sqrt{-1})$, and additionally on one boundary component whose conjugacy class is $\Phi(-1)$ or $\Phi(\sqrt{-1})$.  
This operation again forces $s_-^- = 0$.  

If $s_-^-$ is odd and $s_+ = 0$, we may multiply by $-I$ on $s_-^- - 1$ boundary components whose conjugacy class is $\Phi(-\sqrt{-1})$, which forces $s_-^- = 1$.  

Since this operation does not change the Toledo invariant, we obtain the following Milnor--Wood type inequality:
\begin{equation}\label{par-eqn31}
   \mathrm{T}(\phi) \leq |\chi(\Sigma)| - s_- + 1 = 2g - 1 + s_+,
\end{equation}
where equality holds only when  $s_-^-$ is odd and $s_+ = 0$.

When $n$ is even, any boundary parabolic representation can be deformed into an interior hyperbolic representation (see Lemma~\ref{par-lemma222}).  
Therefore, by the decomposition in Figure~\ref{fig:decomposition1} and equation~\eqref{par-eqn30}, we have
\begin{equation*}
  \mathrm{T}(\tilde{\phi}) 
   = \mathrm{T}\big(\tilde{\phi}|_{\pi_1(\Sigma_{g,n/2})}\big) 
     + \frac{1}{2} \sum_{i=1}^n s_i .
\end{equation*}
If equality holds in~\eqref{par-eqn31}, then $s_+ = 0$, so $s_- = n$. In this case,
\begin{equation*}
   \mathrm{T}\big(\tilde{\phi}|_{\pi_1(\Sigma_{g,n/2})}\big) - \frac{n}{2} = 2g - 1,
\end{equation*}
which implies
\begin{equation*}
   \mathrm{T}\big(\tilde{\phi}|_{\pi_1(\Sigma_{g,n/2})}\big) 
      = 2g - 1 + \frac{n}{2} > |\chi(\Sigma_{g,n/2})|,
\end{equation*}
contradicting the Milnor--Wood inequality.

When $n$ is odd, if equality in~\eqref{par-eqn31} holds, then
\[
   \mathrm{T}\big(\tilde{\phi}|_{\pi_1(\Sigma_{g,(n+1)/2})}\big)
   = 2g - 1+\frac{n-1}{2} = |\chi(\Sigma_{g,(n+1)/2})|.
\]
Let $\psi = \tilde{\phi}|_{\pi_1(\Sigma_{g,(n+1)/2})} \in \mathrm{HP}(\Sigma_{g,(n+1)/2})$. 
If $\psi$ has maximal Toledo invariant, then the $s$-value of $\psi(c_n)$ must be $1$.  
Indeed, when $n\geq 3$, up to multiplying by $-I$ on two boundary components, we may assume $\psi(c_n)$ has positive trace.  
If its $s$-value is $-1$, then $\boldsymbol{\rho}(\psi) = 1$, hence 
\[
   \mathrm{sign}(\psi) = 2\mathrm{T}(\psi) + 1 > 2|\chi(\Sigma_{g,(n+1)/2})|,
\]
contradicting the Milnor--Wood inequality for the signature. When $n = 1$, we have $\psi = \tilde{\phi} \in \mathrm{Par}(\Sigma_{g,1})$ with maximal Toledo invariant. 
By Proposition~\ref{prop1}, $\sigma(\psi(c_n)) = \sqrt{-1}$, so the $s$-value of $\psi(c_n)$ is $1$.

Thus $s_+ \geq 1$, contradicting the condition for equality in~\eqref{par-eqn31}.  

Hence, when $g \geq 1$, equality in~\eqref{par-eqn31} never occurs, and we obtain
\begin{equation}\label{par-eqn32}
   \mathrm{T}(\phi) \leq |\chi(\Sigma)| - s_- .
\end{equation}
When $g=0$, for the subset $\mathcal{P}_{\mathrm{I}}$ of representations deformable to interior hyperbolic representations, the same argument applies, so~\eqref{par-eqn32} still holds.

Similarly, for $g\geq 1$ or for the subset $\mathcal{P}_{\mathrm{I}}$, we obtain the opposite inequality
\[
   \mathrm{T}(\phi) \geq -|\chi(\Sigma)| + s_+ .
\]
Combining our discussions of boundary parabolic representations,  
we obtain the following theorem, which is also proved in~\cite[Theorem~1.1 and Theorem~1.2]{RY}.
\begin{thm}
Let $\Sigma = \Sigma_{g,n}$ be a surface with $\chi(\Sigma) \leq -1$.  
If $\phi \in \mathrm{Par}_P(\Sigma_{g,n})$ can be deformed to an interior hyperbolic representation, then
\[
   -|\chi(\Sigma)| + s_+ \ \leq\ \mathrm{T}(\phi) \ \leq\ |\chi(\Sigma)| - s_-.
\] Moreover, each subset $\sigma^{-1}(s)\cap \mathrm{T}^{-1}(k)$ is non-empty and connected for $k$ and $s$ satisfying the above inequality. Here $s=(s_1,\dots,s_n)\in\{\pm 1\}^n$, $s_\pm$ is defined from $s$ as above.  The sigma map
 $\sigma:\mathrm{Par}_P(\Sigma_{g,n}) \to \{\pm 1\}^n$ is induced from the map defined in \eqref{par-eqn-sigma}. 
\end{thm}

From the above generalized Milnor--Wood inequality, we can also deduce the number of connected components of the representations that can be deformed to an interior hyperbolic representation, which is given by
\begin{align*}
  \sum_{s_+=0}^n \ \sum_{\mr{T}=-|\chi(\Sigma)|+s_+}^{\,|\chi(\Sigma)|-(n-s_+)} \binom{n}{s_+}
  &= (2|\chi(\Sigma)| - n + 1) \sum_{s_+=0}^n \binom{n}{s_+} \\
  &= \big( 4g + n - 3 \big)\cdot 2^n.
\end{align*}

\begin{rem}
For $\phi \in \mathcal{P}_{\mathrm{II}}$, that is, for representations which cannot be deformed into an interior hyperbolic representation, suppose instead that $\phi$ can be deformed into an interior elliptic representation. Then we have $\mr{T}(\phi) = \pm 1$ whenever $s_{\mp} = 0$, see~\cite[Theorem~1.2 (3)]{RY}. In this case, $\dim \mr{H}^0(\Sigma,\mathcal{E})$ can be zero, which precisely corresponds to the situation where equality holds in~\eqref{par-eqn31}.  

On the other hand, if $\phi$ can only be interior parabolic, then it is abelian, its Toledo invariant is zero and either $s_+ = 1$ or $s_- = 1$, see~\cite[Theorem~1.2 (2)]{RY}. In this case, $\dim \mr{H}^0(\Sigma,\mathcal{E}) = 1$. From the discussion above, we also know that $s^-_- = 0$, and therefore by \eqref{par-eqn33}, the inequality~\eqref{par-eqn31} becomes
\[
   \mathrm{T}(\phi) \leq |\chi(\Sigma)| - s_- + \dim \mr{H}^0(\Sigma,\mathcal{E}) \;=\; -1 + s_+.
\]
Similarly, we obtain the opposite inequality,
\[
   1 - s_- \;\leq\; \mathrm{T}(\phi) \;\leq\; -1 + s_+.
\]
In particular, we see that $\mathrm{T}(\phi) = 0$ occurs precisely when equality holds in one of the above inequalities.
\end{rem}

\end{CJK}
\end{document}